\documentclass[twoside,12pt]{article}

\usepackage{amsmath}
\usepackage{amssymb}
\usepackage{graphicx}
\usepackage{theorem}
\usepackage{pifont}
\usepackage{color}
\usepackage{enumitem}
\usepackage{hyperref}
\usepackage{todonotes}
\usepackage{siunitx}
\usepackage{cancel}
\usepackage[normalem]{ulem}
\usepackage{tikz-cd}
\usepackage{pgfkeys}
\usepackage{quiver}
\usepackage[left=1.2in,top=0.6in,right=1.2in,bottom=0.6in]{geometry}

\usepackage{xcolor} 

\usepackage{amsmath}
\usepackage{amssymb}
\usepackage{graphicx}
\usepackage{eucal}
\usepackage{mathrsfs}
\usepackage{pifont}
\usepackage{color}
\usepackage{cjhebrew}

\usepackage[all]{xy}
\usepackage{tikz} 
\usepackage{tkz-euclide}
\usepackage{pgfplots}

\newcommand{\btkz}{\begin{tikzpicture}}
\newcommand{\etkz}{\end{tikzpicture}}

\newcommand{\brk}[1]{\left(#1\right)}          
\newcommand{\Brk}[1]{\left[#1\right]}          
\newcommand{\BRK}[1]{\left\{#1\right\}}        

\newcommand{\Norm}[1]{\left\| #1 \right\|}     

\newcommand{\secref}[1]{Section~\ref{#1}}

\newcommand{\thmref}[1]{Theorem~\ref{#1}}
\newcommand{\defref}[1]{Definition~\ref{#1}}

\newcommand{\propref}[1]{Proposition~\ref{#1}}
\newcommand{\lemref}[1]{Lemma~\ref{#1}}
\newcommand{\corrref}[1]{Corollary~\ref{#1}}

\newcommand{\beq}{\begin{equation}}
\newcommand{\eeq}{\end{equation}}
\newcommand{\bsplit}{\begin{split}}
\newcommand{\esplit}{\end{split}}
\newcommand{\baligned}{\begin{aligned}}
\newcommand{\ealigned}{\end{aligned}}

\providecommand{\e}{\varepsilon}
\providecommand{\half}{\frac{1}{2}}

\providecommand{\R}{\bbR}
\newcommand{\smallhalf}{\tfrac{1}{2}}

\newcommand{\textand}{\quad\text{ and }\quad}
\newcommand{\Textand}{\qquad\text{ and }\qquad}

\providecommand{\vp}{\varphi}

\newcommand{\End}{{\operatorname{End}}}

\newcommand{\id}{{\operatorname{Id}}}
\newcommand{\image}{{\operatorname{Image}}}
\newcommand{\Hom}{{\operatorname{Hom}}}

\newcommand{\sgn}{{\operatorname{sgn}}}

\newcommand{\trace}{{\operatorname{Tr}}}

\newcommand{\calC}{{\mathcal C}}

\newcommand{\calE}{{\mathcal E}}

\newcommand{\calL}{{\mathcal L}}
\newcommand{\calM}{{\mathcal M}}
\newcommand{\calN}{{\mathcal N}}

\newcommand{\calP}{{\mathcal P}}

\newcommand{\calS}{{\mathcal S}}

\newcommand{\frakB}{\mathfrak{B}}

\newcommand{\frakF}{\mathfrak{F}}
\newcommand{\frakG}{\mathfrak{G}}

\newcommand{\frakR}{\mathfrak{R}}

\newcommand{\frakT}{\mathfrak{T}}

\newcommand{\frakX}{\mathfrak{X}}

\newcommand{\frakg}{\mathfrak{g}}
\newcommand{\frakh}{\mathfrak{h}}

\newcommand{\frakn}{\mathfrak{n}}

\newcommand{\frakt}{\mathfrak{t}}

\newcommand{\bbC}{{\mathbb C}}

\newcommand{\bbF}{{\mathbb F}}

\newcommand{\bbH}{{\mathbb H}}

\newcommand{\bbN}{{\mathbb N}}

\newcommand{\bbP}{{\mathbb P}}

\newcommand{\bbR}{{\mathbb R}}

\newcommand{\scrD}{\mathscr{D}}

\newcommand{\scrH}{\mathscr{H}}

\newtheorem{theorem}{Theorem}[section]
\newtheorem{lemma}[theorem]{Lemma}
\newtheorem{proposition}[theorem]{Proposition}
\newtheorem{corollary}[theorem]{Corollary}
\newtheorem{definition}[theorem]{Definition}

\newenvironment{proof}{{\flushleft \emph{Proof}:}}{\hfill\ding{110}}
\newenvironment{PROOF}[1]{{\flushleft {\bfseries Proof of #1}:}}{\hfill\ding{110}}

\newcommand{\M}{\calM}
\newcommand{\dM}{\partial\M}
\newcommand{\VF}{\frakX}
\newcommand{\N}{\calN}

\newcommand{\bra}{\langle}
\newcommand{\ket}{\rangle}

\newcommand{\nl}{\left\|}
\newcommand{\nr}{\right\|}

\newcommand{\gE}{\g_\mathbb{E}}
\newcommand{\E}{\mathbb{E}}

\newcommand{\jStar}{j^\star}
\newcommand{\jEpsStar}{j_\e^\star}

\newcommand{\Wsp}{W^{s,p}\Gamma(\E)}
\newcommand{\WspOk}{W^{s,p}\Omega^k(\M;\E)}

\newcommand{\g}{\frakg}
\newcommand{\G}{\frakG}
\newcommand{\GV}{\frakG_V}
\newcommand\gEps{{\g_\e}}
\newcommand{\h}{\frakh}

\renewcommand{\S}{\calS}

\newcommand{\II}{\operatorname{II}}
\newcommand{\Pt}{\bbP_\e^{\frakt}}
\newcommand{\Pn}{\bbP_\e^{\frakn}}
\newcommand{\Ptt}{\bbP_\e^{\frakt\frakt}}
\newcommand{\Ptn}{\bbP_\e^{\frakt\frakn}}
\newcommand{\Pnt}{\bbP_\e^{\frakn\frakt}}
\newcommand{\Pnn}{\bbP_\e^{\frakn\frakn}}

\newcommand{\ott}{\omega^{\frakt\frakt}}
\newcommand{\otn}{\omega^{\frakt\frakn}}
\newcommand{\ont}{\omega^{\frakn\frakt}}
\newcommand{\onn}{\omega^{\frakn\frakn}}
\newcommand{\ltt}{\lambda^{\frakt\frakt}}
\newcommand{\ltn}{\lambda^{\frakt\frakn}}
\newcommand{\lnt}{\lambda^{\frakn\frakt}}
\newcommand{\lnn}{\lambda^{\frakn\frakn}}

\newcommand{\ixi}{i_{\xi^{\sharp}}}
\newcommand{\ixiV}{i^{V}_{\xi^{\sharp}}}

\newcommand{\gD}{{\g_0}}
\newcommand{\PtD}{\bbP^{\frakt}}
\newcommand{\PnD}{\bbP^{\frakn}}
\newcommand{\PttD}{\bbP^{\frakt\frakt}}
\newcommand{\PtnD}{\bbP^{\frakt\frakn}}
\newcommand{\PntD}{\bbP^{\frakn\frakt}}
\newcommand{\PnnD}{\bbP^{\frakn\frakn}}

\newcommand{\nabg}{\nabla^\g}
\newcommand{\nabgEps}{\nabla^{\gEps}}
\newcommand{\nabgD}{\nabla^{\gD}}
\newcommand{\nabE}{\nabla^{\E}}

\newcommand{\dr}{\partial_r}
\newcommand{\idr}{i_{\partial_r}}
\newcommand{\al}{\alpha}
\newcommand{\be}{\beta}
\newcommand{\starG}{\star_\g}
\newcommand{\stargEps}{\star_\gEps}

\newcommand{\TT}{{\mathrm{TT}}}
\newcommand{\NT}{{\mathrm{NT}}}
\newcommand{\TN}{{\mathrm{TN}}}
\newcommand{\NN}{{\mathrm{NN}}}
\newcommand{\RR}{{\mathrm{R}}}

\newcommand{\Hg}{H_\g}
\newcommand{\bHg}{\mathbf{H}_\g}
\newcommand{\Fg}{F_\g}
\newcommand{\bFg}{\mathbf{F}_\g}
\newcommand{\Bg}{\mathbf{B}_\g}

\newcommand{\trhe}{\trace_{\frakh_\e}}
\newcommand{\ihe}{i_{\frakh_\e}}

\newcommand{\curl}{\operatorname{\nabla\times}}
\renewcommand{\trace}{\operatorname{tr}}
\renewcommand{\image}{\operatorname{Image}}
\newcommand{\Image}{\operatorname{Im}}
\newcommand{\Ric}{\operatorname{Ric}}
\newcommand{\Rm}{\operatorname{Rm}}

\renewcommand{\sgn}{\operatorname{sgn}}
\newcommand{\Hess}{\operatorname{Hess}}
\newcommand{\Graph}{\operatorname{Graph}}

\newcommand{\EE}{\mathcal{EE}}
\newcommand{\EC}{\mathcal{EC}}
\newcommand{\CE}{\mathcal{CE}}
\newcommand{\CC}{\mathcal{CC}}
\newcommand{\BH}{\mathcal{BH}}
\newcommand{\BHkm}{\BH^{k,m}(\M)}

\newcommand{\Riemann}{R}

\newcommand{\dg}{d^{\nabg}}
\newcommand{\deltag}{\delta^{\nabg}}
\newcommand{\dgEps}{d^{\nabgEps}}
\newcommand{\deltagEps}{\delta^{\nabgEps}}
\newcommand{\dgD}{d^{\nabgD}}
\newcommand{\deltagD}{\delta^{\nabgD}}

\newcommand{\dgV}{d^{\nabg}_V}
\newcommand{\deltagV}{\delta^{\nabg}_V}
\newcommand{\dgEpsV}{d^{\nabgEps}_V}
\newcommand{\deltagEpsV}{\delta^{\nabgEps}_V}
\newcommand{\dgDV}{d^{\nabgD}_V}
\newcommand{\deltagDV}{\delta^{\nabgD}_V}

\newcommand{\weak}{\rightharpoonup}

\newcommand{\EwR}[1]{\stackrel{\eqref{#1}}{=}}

\newcommand{\VectorFormsM}{\Omega^{*,*}(\M)}
\newcommand{\VectorFormsU}{\Omega^{*,*}(U)}
\newcommand{\VectorFormsPeps}{\Omega^{*,*}(\calP_\e)}
\newcommand{\VectorFormsdM}{\Omega^{*,*}(\dM)}
\newcommand{\VectorFormsKM}{\Omega^{k,m}(\M)}

\newcommand{\Volume}{\text{Vol}}
\newcommand{\VolumeG}{d\Volume_\g}

\newcommand{\VoldM}{d\Volume_{\jStar\g}}
\newcommand{\VolumeD}{d\Volume_{\gD}}

\makeatletter
\newcommand*\owedge{\mathpalette\@owedge\relax}
\newcommand*\@owedge[1]{%
  \mathbin{%
    \ooalign{%
      $#1\m@th\bigcirc$\cr
      \hidewidth$#1\m@th\wedge$\hidewidth\cr
    }%
  }%
}
\makeatother

\numberwithin{equation}{section}

\begin{document}

\title{Double forms: Regular elliptic bilaplacian operators
}

\author{
Raz Kupferman and
Roee Leder  \footnote{
raz@math.huji.ac.il,
roee.leder@mail.huji.ac.il
}
\\
\\
Institute of Mathematics \\
The Hebrew University \\
Jerusalem 9190401 Israel
}
\maketitle

\begin{abstract}
Double forms are sections of the vector bundles $\Lambda^{k}T^*\M\otimes \Lambda^{m}T^*\M$, where in this work $(\M,\g)$ is a compact Riemannian manifold with boundary. We study graded second-order differential operators on double forms, which are used in physical applications. A Combination of these operators yields a fourth-order operator, which we call a double bilaplacian. 
We establish the regular ellipticity of the double bilaplacian for several sets of boundary conditions. Under additional conditions, we obtain a Hodge-like decomposition for double forms, whose components are images of the second-order operators, along with a biharmonic element. This analysis lays foundations for resolving several topics in incompatible elasticity, most prominently the existence of stress potentials and Saint-Venant compatibility.
\end{abstract}

\tableofcontents
\section{Introduction}

This work is motivated by problems in the theory of elasticity, and specifically incompatible elasticity, in which the setting is that of finding an optimal immersion of one Riemannian manifold $(\M,\g)$ (the body) into another manifold $(\tilde{\M},\tilde{\g})$ (space). The corresponding Euler-Lagrange equations form a boundary-value problem for the stress tensor, which is a vector-valued form $\sigma\in\Omega^1(\M;T^*\M)$, supplemented with so-called constitutive relations, which can be related to the metric discrepancy between $\g$ and $\tilde{\g}$. One approach to solving such systems is the use of stress potentials, a method which goes back to the 19th century \cite{Air63,Bel92,Tru59,Gur72}, and  recently revived in the context of incompatible elasticity \cite{MSK14,MSK15}. The use of stress potentials brings forward second-order differential operators known as curl-curl and its dual div-div (Gurtin \cite{Gur72}); after linearization, the resulting partial differential equation is an equation for the bilaplacian of the stress potential.

This paper studies the analytical foundations of bilaplacian theory in incompatible elasticity. We study double forms \cite{deR84, Cal61, Gra70, Kul72}, which are vector-valued forms of type $\Omega^{k,m}(\M) = \Omega^k(\M;\Lambda^m T^*\M)$ on a $d$-dimensional manifold, generalizing the physical case of $k=m=1$ and $d=2,3$. For $k=m=1$ and when the metric $\g$ is locally-flat, the familiar curl-curl operator can be identified as an operator $\Omega^{1,1}(\M) \to \Omega^{2,2}(\M)$, which is the two-fold anti-symmetrization of the second covariant derivative,
\[
\begin{split}
\curl\curl\sigma(X_{1},X_{2};Y_{1},Y_{2}) 
&= \nabla^2_{X_1,Y_1}\sigma(X_2;Y_2) - \nabla^2_{X_2,Y_1}\sigma(X_1;Y_2) \\
&- \nabla^2_{X_1,Y_2}\sigma(X_2;Y_1) + \nabla^2_{X_2,Y_2}\sigma(X_1;Y_1).
\end{split}
\]
For arbitrary $(k,m)$ and generally non-flat metric $\g$, this operator generalizes into
\[
\begin{split}
& \curl\curl\psi(X_1,\dots,X_{k+1};Y_1,\dots,Y_{m+1}) \\
&\quad = \half\sum_{i=1}^{k+1} \sum_{j=1}^{m+1}
(-1)^{i+j} \nabla^2_{(X_i,Y_j)} \psi(X_1,\dots,\hat{X}_i,\dots,X_{k+1};Y_1,\dots,\hat{Y}_j,\dots,Y_{m+1}),
\end{split}
\]
where $\nabla^2_{(X,Y)}$ is the symmetrized second derivative. Noting that the space of double forms $\Omega^{k,m}(\M)$ has an exterior algebra structure both on its ``form part" and on its ``vector part", the curl-curl operator can be represented as a two-fold application of the covariant exterior derivative $\dg$,
\[
\Omega^{k,m}(\M) 
\stackrel{\dg}{\to} 
\Omega^{k+1,m}(\M)
\stackrel{T}{\to} 
\Omega^{m,k+1}(\M)  
\stackrel{\dg}{\to} 
\Omega^{m+1,k+1}(\M)
\stackrel{T}{\to} 
\Omega^{k+1,m+1}(\M),
\]
where $(\cdot)^T$ stands for the algebraic transposition, switching between the two exterior algebras. 

The graded space $\bigoplus \Omega^{k,m}(\M)$ is a graded algebra, endowed with a natural wedge-product, and graded differential operators $\dg: \Omega^{k,m}(\M)\to\Omega^{k+1,m}(\M)$ and their duals $\deltag: \Omega^{k,m}(\M)\to\Omega^{k-1,m}(\M)$. These operators have counterparts acting via transposition on the external algebra of their vector part \cite{Cal61,Kul72}, $\dgV: \Omega^{k,m}(\M)\to\Omega^{k,m+1}(\M)$ and $\deltagV: \Omega^{k,m}(\M)\to\Omega^{k,m-1}(\M)$. With these definitions, the curl-curl operator takes the more compact form 
\[
\Hg = \tfrac12\brk{\dg\dgV + \dgV\dg}.
\]
We further introduce its dual and ``semi-duals"
\[
\begin{aligned}
\Hg^* &= \tfrac12\brk{\deltagV\deltag + \deltag\deltagV} \\
\Fg^* &= \tfrac12\brk{\deltag\dgV + \dgV\deltag} \\
\Fg &= \tfrac12\brk{\deltagV\dg + \dg\deltagV},
\end{aligned}
\]
which are generalizations of the operators known as div-div, div-curl and curl-div \cite{Gur72}. These differential operators commute with algebraic operations pertinent to double forms \cite{Gra70, Kul72}, resulting in certain symmetry preservations.

In a non-Euclidean setting, these second-order operators are altered by graded tensorial operators 
$D_\g:\Lambda^{k,m}T^*\M\rightarrow \Lambda^{k+1,m+1}T^*\M$ and 
$S_\g:\Lambda^{k,m}T^*\M\rightarrow \Lambda^{k+1,m-1}T^*\M$,
%
\[
\begin{aligned}
&\bHg=\Hg + D_\g    &\qquad& \bFg=\Fg +  S_\g \\
&\bHg^*=\Hg^* +  D_\g^* &\qquad& \bFg^*=\Fg^* + S_\g^*
\end{aligned}
\]
where $D_\g^*,S_\g^*$ are the metric duals. 


In analogy with the Hodge Laplacian,
$\Delta_\g = \dg\deltag + \deltag\dg$,
we define the double bilaplacian as the fourth-order linear operator,
\[
\Bg = \bHg\bHg^* + \bHg^*\bHg + \bFg^*\bFg + \bFg\bFg^*.
\]
There is ample reason to name this operator a bilaplacian. Its principal symbol is $|\xi|^4$. For example, when $D_\g,S_\g=0$, for scalar functions $f\in\Omega^{0,0}(\M)$, $\bHg=\Hg$ coincides with the Hessian and 
\[
\Bg f = \Delta_\g^2 f - \delta\Ric_\g(df),
\]
where $\Ric_\g\in\End(T^*\M)$
is the Ricci operator. Finally, in a locally-flat manifold and arbitrary $(k,m)$,  $\Bg = \Delta_\g^2$.

The physical context imposes that this analysis be performed on manifolds with boundaries. Like in second-order elliptic theory, natural boundary conditions arise from integration by parts formulas. To this end, we introduce mixed projections of tangential and normal boundary components,
\beq
\begin{aligned}
&\PttD:\Omega^{k,m}(\M)\to \Omega^{k,m}(\dM) 
&\qquad 
&\PntD:\Omega^{k,m}(\M)\to \Omega^{k-1,m}(\dM) \\ 
&\PtnD:\Omega^{k,m}(\M)\to \Omega^{k,m-1}(\dM) 
&\qquad 
&\PnnD:\Omega^{k,m}(\M)\to \Omega^{k-1,m-1}(\dM).
\end{aligned}
\label{eq:intro_zero_order_operators}
\eeq
The first superscript in $\frakt\frakt$, $\frakt\frakn$, $\frakn\frakt$ and $\frakn\frakn$ refers to the projection of the form part, whereas the second superscript refers to the projection of the vector part, which also has the structure of an exterior algebra. These are complemented by a set of first-order linear differential boundary operators,
\beq
\begin{aligned}
&\frakT:\Omega^{k,m}(\M)\to \Omega^{k,m}(\dM) 
&\qquad 
&\frakF^*:\Omega^{k,m}(\M)\to \Omega^{k-1,m}(\dM) \\ 
&\frakF:\Omega^{k,m}(\M)\to \Omega^{k,m-1}(\dM) 
&\qquad 
&\frakT^*:\Omega^{k,m}(\M)\to \Omega^{k-1,m-1}(\dM),
\end{aligned}
\label{eq:intro_first_order_operators}
\eeq
such that
\beq
\begin{aligned}
\bra \bHg\psi, \eta\ket &=  \bra \psi, \bHg^*\eta\ket + 
\int_{\dM}\Brk{(\PttD\psi,\frakT^* \eta)_\gD -  (\frakT\psi,\PnnD \eta)_\gD}\VolumeD  \\
\bra \bFg\psi, \eta\ket &=  \bra \psi, \bFg^*\eta\ket + 
\int_{\dM}\Brk{(\PtnD\psi,\frakF^*\eta)_\gD - (\frakF\psi,\PntD\eta)_\gD}\VolumeD,
\end{aligned}
\label{eq:integration_by_parts_intro}
\eeq
where $\bra\cdot,\cdot\ket$ denotes the natural $L^2$-inner-product of double forms, $\gD$ is the induced metric on the boundary and $\VolumeD$ is the corresponding volume form.

It is well-known \cite{Tay11a} that the boundary-value problem for a fourth-order strongly elliptic operator supplemented by Dirichlet boundary conditions,
\[
\psi|_{\dM}=\phi 
\Textand
\nabg_{\N}\psi|_{\dM}=\rho,
\]
is regular elliptic \cite{Hor07,Tay11a}. In the present context, this implies that the boundary-value problem,
\[
\begin{gathered}
\Bg\psi = \chi  \\
(\PttD,\PntD,\PtnD,\PnnD)\psi = (\phi_1,\phi_2,\phi_3,\phi_4) 
\qquad
(\frakT,\frakF^*,\frakF,\frakT^*)\psi = (\rho_1,\rho_2,\rho_3,\rho_4)
\end{gathered}
\]
is regular elliptic. This choice of boundary conditions, however, is not appropriate to the physics applications, which usually only include the boundary operators associated with the normal components. Instead, we consider several sets of boundary conditions, reminiscent to relative and absolute boundary conditions for the Hodge laplacian \cite{Tay11a,Tay11b, Sch95b}. Our first result is:

\begin{theorem}
\label{thm:intro1}
The equation
\[
\begin{gathered}
\Bg\psi = \chi 
\end{gathered}
\]
forms a regular elliptic boundary-value problem for each of the following sets of boundary operators,
\[
\begin{aligned}
&(\PnnD,\PntD,\PtnD,\frakT^*,\frakF^*,\frakF,\PnnD\bHg,\frakT^*\bHg) \\
&(\PttD,\PntD,\PtnD,\frakT,\frakF^*,\frakF,\PttD\bHg^*,\frakT\bHg^*) \\
&(\PntD,\PttD,\PnnD,\frakF^*,\frakT,\frakT^*,\PntD\bFg,\frakF^*\bFg) \\
&(\PtnD,\PttD,\PnnD,\frakF,\frakT,\frakT^*,\PtnD\bFg^*,\frakF\bFg^*).
\end{aligned}
\]
\end{theorem}

For scalar functions $f\in\Omega^{0,0}(\M)$, where $\bHg f$ is the hessian of $f$, the only boundary operators that do not vanish trivially are $(\PttD,\frakT)$ and $(\PnnD \bHg,\frakT^*\bHg)$. Hence, the four sets of boundary operators reduce to two:
\[
(\PttD,\frakT)
\Textand 
(\PnnD\bHg ,\frakT^*\bHg ),
\]
which amount to determining Dirichlet boundary conditions,
\[
f|_{\dM}
\Textand 
\partial_{\calN} f|_{\dM},
\]
or Neumann boundary conditions,
\[
\partial_{\calN}^2f|_{\dM}
\Textand
\partial_{\calN}^3f|_{\dM}+\text{lower order terms}.
\]
The regular ellipticity of neither of these problems can be obtained from the elliptic theory of the laplacian, in contrast to other possible sets \cite{Wlo87}. 

\thmref{thm:intro1} has several implications: Firstly, we characterize the (finite-dimensional and smooth) kernels of the Fredholm operators provided by \thmref{thm:intro1},
\[
\begin{aligned}
\BH^{k,m}_\TT &= \ker{(\bHg,\bHg^*,\bFg^*,\bFg)}\cap \ker(\PttD,\frakT,\PntD,\frakF^*,\PtnD,\frakF) \\
\BH^{k,m}_\NN &= \ker{(\bHg,\bHg^*,\bFg^*,\bFg)}\cap \ker(\PnnD,\frakT^*,\PntD,\frakF^*,\PtnD,\frakF) \\
\BH^{k,m}_\TN &= \ker{(\bHg,\bHg^*,\bFg^*,\bFg)}\cap \ker(\PtnD,\frakF,\PttD,\frakT,\PnnD,\frakT^*) \\
\BH^{k,m}_\NT &= \ker{(\bHg,\bHg^*,\bFg^*,\bFg)}\cap \ker(\PntD,\frakF^*,\PttD,\frakT,\PnnD,\frakT^*),
\end{aligned}
\]
where for linear operators $T_1,T_2,\dots$, we denote by $\ker(T_1,T_2,\dots)$ the intersection of their kernels.
These in turn yield a characterization of the co-kernels, leading to existence and uniqueness results. 

Secondly, we obtain a Korn-like inequality,
\[
\begin{split}
\|\psi\|_{W^{2,2}(\M)} &\lesssim \|\bHg\psi\|_{L^2(\M)} + \|\bFg^*\psi\|_{L^2(\M)} + \|\bFg\psi\|_{L^2(\M)} + \|\bHg^*\psi\|_{L^2(\M)} \\
&+ \|\PnnD\psi\|_{H^{3/2}(\dM)} + \|\PntD\psi\|_{H^{3/2}(\dM)} + \|\PtnD\psi\|_{H^{3/2}(\dM)} \\
&+ \|\frakT^*\psi\|_{H^{1/2}(\dM)} + \|\frakF^*\psi\|_{H^{1/2}(\dM)} + \|\frakF\psi\|_{H^{1/2}(\dM)} \\
&+ \|\psi\|_{L^2(\M)},
\end{split}
\]
where $\lesssim$ denotes inequality up to a multiplicative constant. Three other such inequalities are obtained for the other choices of boundary conditions. 

For $k=m$, we consider the sub-module of $\Omega^{k,k}(\M)$ consisting of symmetric double forms; this is relevant to elasticity theory as the stress tensor is a symmetric element in $\Omega^{1,1}(\M)$ whereas the stress potential is a symmetric element in $\Omega^{2,2}(\M)$ satisfying the first Bianchi identity. Restricting the elliptic analysis to symmetric double forms, we obtain that symmetric data result in  symmetric solutions. 

Thus far, all the results apply to manifolds of arbitrary geometry and topology. In the final part of this work, we proceed under the assumption that for a specific $(k,m)$, the diagram 
\[
\begin{tikzcd}
&& {\Omega^{k+1,m+1}(\M)} \\
&& {} \\
{\Omega^{k-1,m+1}(\M)} && {\Omega^{k,m}(\M)} && {\Omega^{k+1,m-1}(\M)} \\
&& {} \\
&& {\Omega^{k-1,m-1}(\M)}
\arrow["{\bHg}"', curve={height=12pt}, from=5-3, to=3-3]
\arrow["{\bHg}"', curve={height=12pt}, from=3-3, to=1-3]
\arrow["{\bHg^*}", curve={height=-12pt}, tail reversed, no head, from=3-3, to=1-3]
\arrow["{\bHg^*}"', curve={height=12pt}, from=3-3, to=5-3]
\arrow["{\bFg}", curve={height=-12pt}, from=3-1, to=3-3]
\arrow["{\bFg^*}"', curve={height=12pt}, tail reversed, no head, from=3-1, to=3-3]
\arrow["{\bFg^*}", curve={height=-12pt}, from=3-5, to=3-3]
\arrow["{\bFg}", curve={height=-12pt}, from=3-3, to=3-5]
\end{tikzcd}
\]
satisfies the set of ``exactness" relations
\beq
\begin{gathered}
\bHg\bHg=0 \qquad \bFg^*\bFg^*=0 \qquad \bFg^*\bHg=0 \qquad \bHg\bFg^*=0 \\
\bHg^*\bHg^*=0 \qquad \bFg\bFg=0 \qquad \bFg\bHg^*=0 \qquad \bHg^*\bFg=0 \\
\bFg\bHg=0 \qquad \bHg^*\bFg^*=0 \qquad \bHg\bFg=0 \qquad \bFg\bHg=0.
\end{gathered}
\label{eq:exactness_into}
\eeq
In a locally-flat manifold, these relations hold for all $(k,m)$ with $S_\g,D_\g=0$; nontrivial geometries in which such a setting holds are addressed in the sequel to this work \cite{KL21b}. 
When combined with the integration by parts formulas, these relations imply the mutual orthogonality of the images of the  second-order operators with appropriate boundary conditions. 

We introduce the following subspaces of $\Omega^{k,m}(\M)$,
\[
\begin{aligned}
\EE^{k,m}(\M) &=  \BRK{\bHg\alpha ~:~ \alpha\in\ker(\PttD,\frakT)} \\
\CC^{k,m}(\M) &=  \BRK{\bHg^*\beta ~:~ \beta\in\ker(\PnnD,\frakT^*)} \\
\EC^{k,m}(\M) &= \BRK{\bFg\gamma ~:~ \gamma\in\ker(\PtnD,\frakF)} \\
\CE^{k,m}(\M) &= \BRK{\bFg^*\lambda ~:~ \lambda\in\ker(\PntD,\frakF^*)} \\
\BH^{k,m}(\M) &= \ker{(\bHg,\bFg^*,\bFg,\bHg^*)},
\end{aligned}
\]
where the symbol $\calE$ stands for ``exact" and the symbol $\calC$ stands for ``co-exact".
With a slight abuse of notations, we identify these spaces with their $L^2$-completions. 
Extending ideas from the Hodge decomposition \cite{Sch95b}, we prove: 

\begin{theorem}
\label{thm:decompostion_intro}
Let $(\M,\g)$ be a compact Riemannian manifold with boundary, and let $\Omega^{k,m}(\M)$ satisfy \eqref{eq:exactness_into}. Then $\Omega^{k,m}(\M)$ splits $L^2$-orthogonally into
\[
\Omega^{k,m}(\M)=\EE^{k,m}(\M)\oplus \CC^{k,m}(\M)\oplus\EC^{k,m}(\M)\oplus
\CE^{k,m}(\M)\oplus\BH^{k,m}(\M).
\]
Furthermore, the spaces in the decomposition are closed. An analogous decomposition holds in Sobolev regularity for $W^{s,p}\Omega^{k,m}(\M)$.
\end{theorem}

Further decompositions of the biharmonic module $\BH^{k,m}(\M)$ are studied, in analogy to the decompositions of the harmonic module in  Hodge theory. 
This analysis is the starting point for obtaining several results in elasticity (e.g., Saint-Venant compatibility and the existence of stress potentials) which are presented in the sequel article \cite{KL21b}.

The paper is organized as follows: 
In \secref{sec:preliminaries}, we introduce notations and state relevant facts pertinent to vector-valued forms. 
In \secref{sec:k,m_forms}, we present double forms and the second-order differential operators. 
In \secref{sec:boundary_construction}, we introduce the boundary operators \eqref{eq:intro_zero_order_operators} and \eqref{eq:intro_first_order_operators}, and study their properties, culminating in the integration by parts formulas \eqref{eq:integration_by_parts_intro}. 
In \secref{sec:elliptic_theory}, we establish the regular ellipticity of the double bilaplacian and derive some of its consequences.  
In \secref{sec:symmetric_forms}, we focus on symmetric forms. 
In \secref{sec:the_decompostion}, we prove  \thmref{thm:decompostion_intro}, 
and demonstrate how to solve related boundary-value problems. To facilitate readability, proofs of technical lemmas and propositions appearing in Sections~\ref{sec:k,m_forms} and \ref{sec:boundary_construction} and Subsection~\ref{subsec:extension} are collected in an appendix.

\paragraph{Acknowledgments}
We thank Matania Ben-Artzi, Cy Maor and Amitai Yuval for insights and advice.  We are grateful to Asaf Shachar for useful comments on the manuscript.
This research was partially supported by the Israel Science Foundation Grant No. 1035/17.

\section{Preliminaries}
\label{sec:preliminaries}
In this section we establish notations used throughout this paper, and review some definitions and facts in the context of sections of vector bundles.

\subsection{The covariant exterior derivative}

Let $(\M,\g)$ be a $d$-dimensional Riemannian manifold; we denote by $(\cdot,\cdot)_\g$ and $\nabg$ the inner-product and the Riemannian connection;
we denote by $\frakX(\M)$ and $\Omega^k(\M)$ the spaces of vector fields and $k$-forms over $\M$.
Let $\E\to\M$ be a vector bundle over $\M$ endowed with a connection $\nabE$.
We denote by $\Omega^k(\M;\E)$ the module of $k$-forms with values in $\E$; we denote by $\nabla^{\g,\E}$ the induced connection on $\Lambda^kT^*\M\otimes\E$.
The covariant exterior derivative 
\[
d^{\nabE}: \Omega^k(\M;\E) \to \Omega^{k+1}(\M;\E)
\]
is defined by
\[
\begin{split}
d^{\nabE}\sigma(X_1,\dots,X_{k+1}) &= 
\sum_{i=1}^{k+1} (-1)^{i+1} \nabla^{\g,\E}_{X_i} \sigma(X_1,\dots,\hat{X}_i,\dots,X_{k+1}).
\end{split}
\]

Let $f:\M\to\N$ be a smooth map between manifolds, and let $\E\to\N$ be a vector bundle; we denote by $f^\star\E$ the pullback bundle over $\M$, where for $p\in\M$, the fiber $(f^\star\E)_p$ is canonically identified with $\E_{f(p)}$. 
The pullback connection $\nabla^{f^\star\E}$ is defined such that for all $s\in\Omega^k(\N;\E)$,
\[
d^{\nabla^{f^\star\E}}f^\star s=f^\star d^{\nabE}s.
\]

The graded algebra $\Omega^*(\M)$ acts naturally on $\Omega^*(\M;\E)$ by a wedge product, both from the left and the right; it is defined by the linear extension of
\[
\eta\wedge(\omega\otimes V)=(\eta\wedge\omega)\otimes V,
\]
where $\eta,\omega\in \Omega^*(\M)$ and $V\in\Gamma(\E)$.
The exterior covariant derivative $d^{\nabE}$ is the unique graded operator on $\Omega^*(\M;\E)$ satisfying
\[
d^{\nabE}s=\nabE s
\]
for $s\in\Gamma(\E)$,
together with the Leibniz rule,
\beq
d^{\nabE}(\alpha\wedge\omega)=d\alpha\wedge\omega+(-1)^k\, \alpha\wedge d^{\nabE}\omega
\label{eq:lebnitz_wedge}
\eeq
for $\alpha\in\Omega^k(\M)$ and $\omega\in\Omega^*(\M;\E)$
\cite[p.~362]{Pet16}.

In this context, there is another natural wedge product,
\[
\wedge:\Omega^p(\M;\End(\E))\times\Omega^q(\M;\E)\to\Omega^{p+q}(\M;\E)
\]
also defined by linear extension: for $T\in\Gamma(\End(\E))$, $V\in\Gamma(\E)$ and $\omega,\eta \in \Omega^*(\M)$,
\beq
(\eta\otimes T)\wedge(\omega\otimes V)=(\eta\wedge\omega)\otimes T(V).
\label{eq:wedge_product_vector_valued_End_2}
\eeq

The connection $\nabE$ induces a natural connection on $\End(\E)$. The exterior covariant derivative $d^{\nabE}$ satisfies the following Leibniz rule with respect to the wedge product \eqref{eq:wedge_product_vector_valued_End_2},
\[
d^{\nabE}(T\wedge\omega)=d^{\nabE}T\wedge\omega+(-1)^p\, T\wedge d^{\nabE}\omega,
\]
for $T\in\Omega^p(\M;\End(\E))$ and $\omega\in\Omega^*(\M;\E)$.
With this notation,
\[
d^{\nabE}d^{\nabE}\omega=\Riemann^{\nabE}\wedge\omega,
\]
where $\Riemann^{\nabE}\in\Omega^2(\M;\End(\E))$ is the curvature endomorphism of the connection $\nabE$. 
In particular, $d^{\nabE} d^{\nabE}  = 0$ if and only if $\nabE$ has vanishing curvature.

\subsection{The Hodge-dual}

Let $(V,\g)$ be an oriented $d$-dimensional inner-product space. 
As usual, we denote by $\Lambda^k V$ the exterior algebra of $k$-vectors.
The Hodge-dual 
\[
\starG: \Lambda^k V \to \Lambda^{d-k}V,
\]
is defined by the relation
\[
\al\wedge \starG \be = (\al,\be)_\g\, \omega,
\]
for all $\al\in\Lambda^kV$,
where $\omega$ is the unit $d$-vector; that is, if $\{e_1,\dots,e_d\}$ is an oriented orthonormal basis for $V$, then
\[
\omega = e_1\wedge\dots\wedge e_d.
\]

In Riemannian geometry, it is more common to apply the Hodge-dual on covectors. The linear-algebraic definition extends naturally to vector bundles. For an oriented manifold $(\M,\g)$,
\[
\starG : \Omega^k(\M) \to \Omega^{d-k}(\M), 
\]
where 
\[
\al\wedge \starG \be = (\al,\be)_\g\, \VolumeG
\]
for all  $\alpha,\beta\in\Omega^k(\M)$, and $\VolumeG$ is the Riemannian volume form.
In particular, 
\[
\starG(1) = \VolumeG
\Textand
\starG(\VolumeG) = 1.
\]
The Hodge-dual is an isometry; it is its own inverse (modulo a sign), namely,
\beq
\starG^{-1} = (-1)^{dk+k} \starG.
\label{eq:Hodge_inverse}
\eeq

The definition of the Hodge-dual extends to vector-valued forms on oriented manifolds. Let $\E$ be a vector bundle over $\M$; then
\[
\starG : \Omega^k(\M;\E) \to \Omega^{d-k}(\M;\E)
\]
is defined by the linear extension of 
\[
\starG(\alpha\otimes\eta) = \starG\alpha\otimes\eta,
\] 
where $\alpha\in\Omega^*(\M)$ and $\eta\in\Gamma(\E)$.

\subsection{The covariant codifferential}
\label{sec:codifferential}

Let $\E\to\M$ be a vector bundle endowed with a metric $\g_\E$ and a metric connection $\nabla^\E$.
The covariant codifferential,
\[
\delta^{\nabla^\E}: \Omega^k(\M;\E) \to \Omega^{k-1}(\M;\E),
\]
is the $L^2$-adjoint of the covariant exterior derivative: for $\tau\in\Omega^k(\M;\E)$ and compactly supported $\sigma\in\Omega^{k-1}(\M;\E)$, 
\[
\int_\M  (\sigma,\delta^{\nabla^\E} \tau)_{\g,\g_\E} \,\VolumeG = \int_\M  (d^{\nabla^\E} \sigma,\tau)_{\g,\g_\E} \,\VolumeG.
\]
It has the following representation \cite[pp.~10--11]{EL83},
\beq
\delta^{\nabla^\E} \tau = (-1)^{dk+d+1} \starG d^{\nabla^\E} \starG \tau.
\label{eq:formula_delta_nabla}
\eeq
Moreover, let $(E_1,\dots,E_d)$ be an orthonormal frame for $T\M$. Then, for $\tau\in\Omega^k(\M;\E)$,
\[
\delta^{\nabE} \tau(X_1,\dots,X_{k-1}) = - \sum_{j=1}^d \nabla^{\g,\E}_{E_j}\tau (E_j,X_1,\dots,X_{k-1}).
\]
Unlike the exterior covariant derivative,  $\dg$, the exterior codifferential $\deltag$ does not  satisfy a simple rule with respect to wedge products.

\subsection{Sobolev sections of vectors bundles}
 
Let $\E\to\M$ be a vector bundle equipped with a fiber metric $\g_\E$. 
Denote by $\Gamma_{c}(\E)$ the space of smooth compactly-supported sections. It is equipped with a scalar product,
\[
\bra\sigma,\theta\ket = \int_{\M}(\sigma,\theta)_{\g_\E}\VolumeG.
\]
For $1\leq p < \infty$, the $L^p$-norm on $\Gamma_{c}(\E)$ is given by 
\[
\left\| \sigma \right\|_{L^p(\M)}^p=\int_{\M} (\sigma,\sigma)_{\g_\E}^{p/2}\VolumeG,
\]
and the space $L^p\Gamma(\E)$ is the completion of $\Gamma_{c}(\E)$ with respect to this norm.

To define Sobolev spaces of higher order, we fix a connection $\nabla^{\E}$ on $\E$. The Sobolev spaces $W^{k,p}\Gamma(\E)$ are defined by
\[
W^{k,p}\Gamma(\E) = \{\sigma\in L^p\Gamma(\E) ~:~ (\nabla^\E)^j\sigma\in  L^p\Gamma(\otimes^j T^*\M\otimes \E), \,\, j=0,\dots,k\},
\]
where the covariant derivatives are defined in a distributive sense \cite[Ch.~4]{Tay11a}.
The spaces $\Wsp$ are Banach spaces with respect to the norm
\[
\|\sigma\|^p_{W^{s,p}(\M)} = \sum_{j=0}^k \|(\nabla^\E)^j\sigma\|^p_{L^p}.
\]
The Hilbert spaces $W^{s,2}\Gamma(\E)$ are also denoted by $H^s\Gamma(\E)$. 

Sobolev sections of vector bundles satisfy the following properties 
\cite{Tay11c, Sch95b}:

\begin{enumerate}
\item {\bfseries Sobolev lemma}: For all integers $s\geq0$ and $t>0$, the embeddings
\beq
\begin{gathered}
W^{s+t,p}\Gamma(\E)\hookrightarrow W^{s,p}\Gamma(\E)
\end{gathered}
\label{eq:sobolev_lemma}
\eeq
are continuous and dense.
\item {\bfseries Rellich's embedding theorem}: The embeddings \eqref{eq:sobolev_lemma} are compact (here we rely on the compactness of $\M$). 
\end{enumerate}

We also have the following property of the trace operator \cite{Tay11c, Sch95b}: Let $s>0$, then the restriction $\sigma\mapsto\sigma|_{\dM}$ extends to a continuous linear surjection,
\[
W^{s+1/p,p}\Gamma(\E)\to W^{s,p}\Gamma(\E|_{\dM}).
\]

\subsection{Sobolev theory for vector-valued forms}

From here on we replace $\E$ with a vector bundle of the form $\Lambda^kT^*\M \otimes\E$, with fiber metric
$\g\oplus\gE$. Recall that
\[
\Omega^k(\M;\E)=  \Gamma(\Lambda^kT^*\M\otimes \E).
\]
Thus, the Sobolev theory of sections of general vectors bundles holds verbatim for $\WspOk$. In particular, the embedding and trace theorems remain valid. The differential operators $d^{\nabE}$ and $\delta^{\nabE}$ extend to vector-valued differential forms of  Sobolev classes, and define continuous mappings 
\[
\begin{gathered}
d^{\nabE}:W^{s+1,p}\Omega^k(\M;\E)\to W^{s,p}\Omega^{k+1}(\M;\E) \\
\delta^{\nabE}:W^{s+1,p}\Omega^k(\M;\E)\to W^{s,p}\Omega^{k-1}(\M;\E).
\end{gathered}
\] 

Furthermore, the Hodge dual $\starG:\WspOk\to W^{s,p}\Omega^{d-k}(\M;\E)$ is an isometry, 
\[
\nl\starG\sigma\nr_{W^{s,p}(\M)}=\nl\sigma\nr_{W^{s,p}(\M)}
\]
for all $\sigma\in\WspOk$.

The integral of a volume form of Sobolev class
is defined as the limit of the integrals of an approximating sequence in $\Omega^{d}_{c}(\M)$ \cite[pp.~59--60]{Sch95b}. With this definition, Green's theorem can be extended to vector-valued forms of Sobolev class:
Let $\omega\in W^{1,p}\Omega^{k-1}(\M;\E)$ and $\eta\in W^{1,q}\Omega^k(\M;\E)$, 
where $1/p+1/q=1$. Then,
\beq
\bra d^{\nabE}\omega,\eta\ket = \bra \omega,\delta^{\nabE}\eta\ket+\int_{\dM}(\jStar\omega,i_{\N}\eta)_{\jStar\g,\gE}\VoldM,
\label{eq:green's_formula}
\eeq
where $j:\dM\to \M$ is the natural embedding and $\N$ is the outer unit normal to the boundary.  
See the discussion building towards \cite[p.~178]{Tay11a} for the full formulation. 

As a final note, we use throughout this paper the notation $\lesssim$ to denote an inequality up to a constant multiplicative factor, which may vary from one line to another, however is independent of the expressions on both sides.

\section{Double forms and second-order differential operators}
\label{sec:k,m_forms}

\subsection{Double forms}

We hereafter assume that the vector bundle $\E\to\M$ is $\E=\Lambda^mT^*\M$ for some $m\in\bbN$, and that the fiber metric $\gE$ is the one naturally induced from $\g$ on multi-vectors. We denote the induced Levi-Civita connection on $\Lambda^kT^*\M\otimes\Lambda^mT^*\M$ by $\nabg$. We consider the spaces of sections,
\[
\VectorFormsKM = \Omega^k(\M;\Lambda^{m}T^*\M)= \Gamma(\Lambda^{k,m}T^*\M),
\]
where
\[
\Lambda^{k,m}T^*\M=\Lambda^kT^*\M\otimes\Lambda^{m}T^*\M.
\]
The graded algebra of double forms \cite{deR84,Cal61, Gra70,Kul72} is defined as.
\[
\VectorFormsM=\bigoplus_{k,m}\VectorFormsKM.
\]

Important examples of elements in $\VectorFormsM$ are scalar functions $f\in\Omega^{0,0}(\M)$, real-valued forms $\omega\in\Omega^{k,0}(\M)$, metrics and hessians of functions $\g,\Hess_\g f\in\Omega^{1,1}(\M)$, and the $(4,0)$-version of the Riemann curvature tensor $\Rm_\g\in\Omega^{2,2}(\M)$.

Double forms have a natural tensorial operation of transposition,
$(\cdot)^T:\VectorFormsKM\to \Omega^{m,k}(\M)$ defined by 
\[
\psi^T(Y_1,\dots, Y_{m};X_1,\dots ,X_k)=\psi(X_1,\dots ,X_k;Y_1,\dots ,Y_{m}),
\]
where the semicolon separates between the $k$ and $m$ parts. 
The transposition of double forms should not be confused with the musical isomorphisms $\#$ and $\flat$, which are metric operations; for example, if $\alpha\in\Omega^1(\M) \simeq \Omega^{1,0}(\M)$, then $\alpha^T\in\Omega^{0,1}(
M)\simeq \Gamma(T^*\M)$, whereas $\alpha^\#\in\VF(\M) =\Gamma(T\M)$.
A $(k,k)$-form $\psi$ satisfying $\psi^T=\psi$ is called symmetric; examples of symmetric forms are metrics, hessians, Ricci tensors and curvature tensors. 

Double forms are equipped with the $\R$-linear operator,  $\dg:\VectorFormsM\to\VectorFormsM$. Transposition commutes with $\nabg_X$ for $X\in\frakX(\M)$, however, transposition does not commute with $\nabg$, hence neither with $\dg$. To illustrate this point, consider a real-valued form $\omega\in\Omega^{k,0}(\M)\simeq\Omega^k(\M)$. Then, $(\dg\omega)^T=(d\omega)^T \in\Omega^{0,k+1}(\M)$, whereas $\dg\omega^T=\nabg\omega^T\in\Omega^{1,k}(\M)$.

The vector bundle $\Lambda^{k,m} T^*\M$ has a vector part which is also an exterior algebra, hence all the operators acting on vector-valued forms can be defined on the vector part through transposition. We introduce the Hodge dual and the interior product
\beq
\starG^V \psi = (\starG \psi^T)^T
\Textand
i_X^V\psi = (i_X\psi^T)^T.
\label{eq_HodgeV}
\eeq 
One readily sees that $\starG\starG^V = \starG^V\starG$ and $i_X^V i_Y = i_Y i_X^V$.
We also note that the metric contraction of the first couple of covariant indices, $\trace_\g:\VectorFormsKM\to \Omega^{k-1,m-1}(\M)$,  is given locally by
\[
\trace_\g\psi=\sum_i i_{E_i}^V i_{E_i}\psi,
\]
where $\{E_i\}$ is a local orthonormal frame. 

We further define the exterior covariant operations,
\[
\dgV\psi = (\dg\psi^T)^T
\Textand
\deltagV\psi = (\deltag\psi^T)^T.
\]
It is readily verified that $\deltagV$ is the $L^2$-adjoint of $\dgV$. Furthermore, for $\psi\in\VectorFormsKM$ (cf. with \eqref{eq:formula_delta_nabla}),
\[
\deltagV\psi = (-1)^{dm+d+1} \starG^V \dgV \starG^V \psi.
\]
The following additional formula are presented in \cite{Kul72},
\beq
\begin{aligned}
& \deltagV = -\trace_\g\dg-\dg\trace_\g \\
& \deltag = -\trace_\g\dgV-\dgV\trace_\g.
\end{aligned}
\label{eq:co-differential_trace}
\eeq

Like $\Lambda^*T^*\M$, the vector bundles $\Lambda^{*,*}T^*\M$ have a natural graded wedge product,
\[
\wedge:\Lambda^{*,*}T\M\times\Lambda^{*,*}T\M\to \Lambda^{*,*}T\M,
\] 
defined by the linear extension of
\beq
(\omega\otimes F)\wedge(\alpha\otimes Q)=(\omega\wedge\alpha)\otimes(F\wedge Q),
\label{eq:def_wedge}
\eeq
turning $\Lambda^{*,*}T\M$ (and hence $\VectorFormsM$) into a graded algebra \cite{deR84}. 

The exterior covariant derivative $\dg:\VectorFormsKM\to\Omega^{k+1,m}(\M)$ satisfies a Leibniz rule with respect to this wedge product: for $\eta\in\VectorFormsKM$ and $\omega\in\VectorFormsM$,
\beq
\begin{aligned}
\dg(\eta\wedge\omega) &= \dg\eta\wedge\omega+(-1)^k\eta\wedge \dg\omega \\
\dgV(\eta\wedge\omega) &= \dgV\eta\wedge\omega+(-1)^m\eta\wedge \dgV\omega.
\end{aligned}
\label{eq:leibnitz_rule_vector}
\eeq
Likewise, for $X\in\frakX(\M)$, $\omega\in\VectorFormsKM$ and $\eta\in\VectorFormsM$,
\[
\begin{aligned}
i_X(\omega\wedge\eta) &= (i_X\omega)\wedge\eta+(-1)^k\omega\wedge i_X\eta \\
i_X^V(\omega\wedge\eta) &= (i_X^V\omega)\wedge\eta+(-1)^m\omega\wedge i_X^V\eta.
\end{aligned}
\]
Finally, since $\nabg_X$ commutes with $\starG$ and $\starG^V$: 

\begin{proposition}
\label{prop:dg_commutes_starV}
$\dg$ commutes with $\starG^V$ and $\dgV$ commutes with $\starG$,
\beq
\dg\starG^V = \starG^V\dg
\Textand
\dgV\starG = \starG\dgV.
\label{eq:dg_commutes_starV}
\eeq
\end{proposition}

Another algebraic operation on double forms is the Bianchi sum \cite{Gra70,Kul72}, which is the bundle map $\G:\Lambda^{k,m}T^*\M\rightarrow  \Lambda^{k+1,m-1}T^*\M$ given by
\[
\G\psi(X_{1},...,X_{k+1};Y_{1},...,Y_{m-1})=\sum_{j=1}^{k+1}(-1)^{j+1}\psi(X_{1},...,\hat{X}_{j},...,X_{k+1};X_{j},Y_{1},...,Y_{m-1}).
\]
We also define $\G_V:\Lambda^{k,m}T^*\M\rightarrow  \Lambda^{k-1,m+1}T^*\M$, given by 
\[
\G_V\psi=(\G\psi^T)^T.
\]

The algebraic operators $\trace_\g$, $\G$  and $\G_V$ are closely related. We describe this relationship in a more general context, which will be used later on. 

Let $A\in\Lambda^{l,l}T^*\M$ be a symmetric tensor of rank $1\leq l\leq d$; it can be wedge-multiplied with an element in $\Lambda^{k,m} T^*\M$ to yield an operator
\[
(A \wedge) : \Lambda^{k,m}T^*\M \to \Lambda^{k+l,m+l}T^*\M. 
\]

\begin{definition}
\label{def:3.2}
The graded linear operators
\[
\begin{aligned}
\trace_A &: \Lambda^{k,m}T^*\M \to \Lambda^{k-l,m-l}T^*\M \\
i_A &: \Lambda^{k,m}T^*\M \to \Lambda^{k+l,m-l}T^*\M \\
i^*_A &: \Lambda^{k,m}T^*\M \to \Lambda^{k-l,m+l}T^*\M
\end{aligned}
\]
are defined by 
\beq
\begin{aligned}
\trace_A &= (-1)^{dk+dm} \starG\starG^V (A \wedge) \starG\starG^V  \\
i_A  &= (-1)^{dm + d} \starG^V (A \wedge)  \starG^V  \\
i^*_A &= (-1)^{dk+d} \starG (A \wedge) \starG.
\end{aligned}
\label{eq:formulas_traceA}
\eeq
\end{definition}

\begin{lemma}
\label{lem:3.2}
Let $\alpha\in\Omega^1(\M)$ and $\omega\in\Omega^k(\M)$. Then,
\beq
\starG(\alpha\wedge\omega) = (-1)^k i_{\alpha^{\sharp}} \starG\omega.
\label{eq:star_alpha_omega}
\eeq
\end{lemma}

\begin{lemma}
\label{lem:3.4}
The pairs of operators $(A \wedge)$, $\trace_A$ and $i_A$, $i^*_A$ are mutually dual with respect to the induced metric on $\Lambda^{*,*}T^*\M$, namely,
\[
(\trace_A \psi,\vp)_\g = (\psi,A\wedge \vp)_\g 
\Textand
(i_A^*\psi,\vp)_\g = (\psi,i_A\vp)_\g.
\]
\end{lemma}

\begin{lemma}
\label{lem:3.5}
For $\psi\in\VectorFormsKM$,
\[
(\trace_A \psi)^T = \trace_A\psi^T 
\Textand
(i_A \psi)^T = i_A^*\psi^T.
\]
\end{lemma}

The Bianchi sums $\G$ and $\GV$ are derived from these definitions by taking $A = \g$:

\begin{lemma}
\label{lem:bianchiG}
With $i_A$ and $i_A^*$ defined as above,
\[
i_\g =\G 
\Textand
i^*_\g = \G_V.
\]
\end{lemma}

Since $\dg \g=0$, it follows from \eqref{eq:leibnitz_rule_vector} that
\beq
\begin{aligned}
& \dg(\g\wedge\psi) = -\g\wedge\dg\psi \\
& \dgV(\g\wedge\psi) = -\g\wedge\dgV\psi.
\end{aligned}
\label{eq:g_wedge_commute_with_dg}
\eeq

\begin{lemma}
\label{lem:commute_d_G}
The following relations hold,
\beq
\begin{aligned}
& \G\dg=-\dg\G 
& \qquad 
& \G_V\dgV=-\dg_{V}\G_V \\
& \trace_\g\deltag=-\deltag\trace_\g 
& \qquad 
& \trace_\g\deltagV=-\deltagV\trace_\g.
\end{aligned}
\label{eq:commutations_dg_G}
\eeq
and
\beq
\begin{aligned} 
& \dg=\G\dgV+\dgV\G=-\deltagV(\g\wedge)-\g\wedge\deltagV \\
& \dgV=\G_V\dg+\dg\G_V=-\deltag(\g\wedge)-\g\wedge\deltag.
&
\end{aligned}
\label{eq:commutations_dg_bianchi}
\eeq
\end{lemma}


\subsection{Second-order differential operators}

We next consider the second-order differential operators $\dgV\dg,\dg\dgV:\VectorFormsKM\to\Omega^{k+1,m+1}(\M)$. 
For a scalar function $f\in \Omega^{0,0}(\M)$,
\[
\dg\dgV f=\dgV\dg f=\nabg (df)^T=\Hess_\g f\in \Omega^{1,1}(\M).
\]
For general values of $k$ and $m$, $\dg$ and $\dgV$ do not commute. For example, let $f\in \Omega^{0,0}(\M)$ be a scalar function, then
\[
\dgV\dg df = 0,
\]
whereas
\[
\dg\dgV df = R_\g^*\circ (df)^T,
\]
where $R^*_\g\in\Omega^2(\M;\End(T^*\M))$ is the curvature endomorphism of $T^*\M$.
Applying \eqref{eq:commutations_dg_bianchi} twice we find,
\beq
\begin{split}
&\Brk{\dg,\dgV} = \dg\dg\G_V-\G_V\dg\dg \\
&\Brk{\deltag,\dgV} = \dgV\dgV\trace_\g-\trace_\g\dgV\dgV.
\end{split}
\label{eq:dg_dgV_commutators}
\eeq

For $\psi\in\VectorFormsKM$ and $\vp\in\VectorFormsM$, 
\beq
\begin{split}
\dgV\dg(\psi\wedge\vp) &= \dgV\dg\psi\wedge\vp + (-1)^k \dgV\psi \wedge (\dg\vp)  \\
& + (-1)^m (\dg\psi)\wedge\dgV\vp +(-1)^{k+m} \psi\wedge \dgV\dg\vp.
\label{eq:leibenitz_rule_dvd}
\end{split}
\eeq
\beq
\begin{split}
\dg\dgV(\psi\wedge\vp) &= \dg\dgV\psi\wedge\vp + (-1)^{m} \dg\psi \wedge (\dgV\vp)  \\
& + (-1)^k (\dgV\psi)\wedge\dg\vp +(-1)^{k+m} \psi\wedge \dg\dgV\vp,
\label{eq:leibenitz_rule_ddv}
\end{split}
\eeq
as obtained by applying \eqref{eq:leibnitz_rule_vector} twice.

This leads to the following definitions:

\begin{definition} 
The graded, second-order differential operators
\[
\begin{aligned}
&\Hg : \VectorFormsKM\to\Omega^{k+1,m+1}(\M) 
&\qquad
&\Hg^* : \VectorFormsKM\to\Omega^{k-1,m-1}(\M) \\
&\Fg^* : \VectorFormsKM\to\Omega^{k-1,m+1}(\M) 
&\qquad
&\Fg : \VectorFormsKM\to\Omega^{k+1,m-1}(\M),
\end{aligned}
\]
are defined by 
\beq
\begin{aligned}
&\Hg = \smallhalf(\dgV\dg+\dg\dgV)                     
&\qquad
&\Hg^* = \smallhalf(\deltag\deltagV+\deltagV\deltag)   \\
&\Fg^* = \smallhalf(\dgV\deltag+\deltag\dgV)                 
&\qquad
&\Fg = \smallhalf(\dg\deltagV+\deltagV\dg).
\end{aligned}
\label{eq:HHFF}
\eeq
\end{definition}

In accordance with the physics terminology, 
we call $\Hg$, $\Hg^*$, $\Fg^*$ and $\Fg$ the curl-curl, div-div, curl-div and div-curl operators. 
It is easy to see that as the notation suggests, $\Hg$ and $\Hg^*$, and $\Fg^*$ and $\Fg$ are mutually dual with respect to the $L^2$-inner-product.
The operator $\Hg$ commutes with the transpose, i.e.,
\[
\Hg(\psi^T)=(\Hg\psi)^T,
\]
and by duality, so does $\Hg^*$. 
On the other hand,
\[
(\Fg^*\psi^T)^T = \Fg\psi.
\]

Combining \eqref{eq:leibenitz_rule_dvd} and \eqref{eq:leibenitz_rule_ddv},
\[
\begin{split}
\Hg(\psi\wedge\varphi)&=\Hg\psi\wedge \varphi+(-1)^k\dgV\psi\wedge\dg\varphi\\&+(-1)^{m}\dg\psi\wedge\dgV\varphi
+(-1)^{k+m}\psi\wedge \Hg\varphi. 
\end{split}
\]

\begin{proposition}
\label{prop:duals_stars}
For $\psi\in\VectorFormsKM$,
\beq
\begin{aligned}
\Hg^*\psi &= (-1)^{dk+dm}\starG \starG^{V} \Hg \starG \starG^{V}\psi \\
\Fg^*\psi &= (-1)^{dk+d+1} \starG \Hg \starG \psi \\
\Fg\psi &= (-1)^{dm+d+1} \starG^{V} \Hg \starG^{V}\psi.
\end{aligned}
\label{eq:duals_stars}
\eeq
\end{proposition}

\begin{proposition}
\label{prop:commute_H_G}
The following commutation relations hold:
\begin{enumerate}[itemsep=0pt,label=(\alph*)]
\item $\Hg$ commutes with $(\g\wedge)$, $\G$ and $\GV$.
\item $\Hg^*$ commutes with $\trace_\g$, $\G$ and $\GV$.
\item $\Fg^*$ commutes with $(\g\wedge)$, $\trace_\g$ and $\GV$.
\item $\Fg$ commutes with $(\g\wedge)$, $\trace_\g$ and $\G$.
\end{enumerate}
\end{proposition}

\section{Boundary constructions for double forms}
\label{sec:boundary_construction}
\subsection{Boundary constructions for vector-valued forms}
\label{section:boundary constructions and boundary operators} 

In treating boundary-value problems for double forms, 
we need the notion of a boundary normal neighborhood and constructions associated with it. 
The basic setting is quite standard and can be found e.g. in  \cite{Lee18} and \cite{Pet16}. The starting point is the following: let $(\M,\g)$ be a $d$-dimensional Riemannian manifold and let $\calP\hookrightarrow\M$ be an embedded submanifold.  There exists a normal neighborhood of $\calP$, $\calP\subseteq U\subseteq \M$, in which one can define  a distance function $r: U \to \R_{\geq0}$, measuring the metric distance of a point in $U$ from $\calP$; the function $r$ is smooth on $U\setminus\calP$ having smooth one-sided  derivatives on $\calP$. 
One can show that $\dr=(dr)^{\#}$ is a unit vector field orthogonal to the level sets of $r$; moreover, $\nabg_{\dr}\dr=0$. 

From here on, we restrict our attention to the case $\calP=\dM$, which is a submanifold of codimension one. 
For $\e\geq 0$, we  define  the level set $\calP_\e=r^{-1}(\lbrace\e\rbrace)$ and denote by $j_\e:\calP_\e \hookrightarrow\M$ the inclusion. In particular, $\calP_0=\dM$. 
We denote the pullback metric $j_\e^*\g$ on $\calP_\e$ be $\gEps$.
Since $r: U \to \R_{\geq0}$ has smooth one-sided directional derivatives on $\dM$, it follows that $\dr$ is a smooth vector field in $U$, and restricts on the boundary to its inward pointing normal, which we denote by $\N = -\dr|_{\dM}$. 


With that, we return to the setting of vector-valued differential forms. In a normal neighborhood $U$ of $\dM$, every vector $v\in TU$ decomposes orthogonally into
\[
v=v^\bot+v^\parallel,
\]
where
\[
v^\bot=dr(v) \,\dr 
\Textand
v^\parallel=v-v^\bot.
\]
As these projections are smooth, $TU$ decomposes orthogonally into smooth sub-bundles,
\[
TU=(TU)^\bot\oplus (TU)^\parallel.
\]
Since $\dr$ is perpendicular to the level sets of $r$, 
\[
(TU)^\parallel|_{\calP_\e} \simeq T\calP_\e,
\]
where $dj_\e: T\calP_\e\to (TU)^\parallel|_{\calP_\e}$ is the isomorphism, and
\[
(TU)^\bot|_{\calP_\e} = N\calP_\e.
\]
In particular, $(TU)^\parallel|_{\dM} \simeq T\dM$ and $(TU)^\bot|_{\dM}=N\dM$.

The module of vector fields near the boundary $\frakX(U)$ also decomposes smoothly and orthogonally into
\[
\frakX(U)=\frakX^\parallel(U)\oplus\frakX^\bot(U),
\]
where $\frakX^\parallel(U)$ is the space of vector fields tangent to the level sets of $r$, and $\frakX^\bot(U)\parallel \dr$; moreover, $\frakX^\parallel(U)|_{\calP_\e} \simeq \frakX(\calP_\e)$.  Gauss formula \cite[p.~228]{Lee18} asserts that for $X,Y\in\frakX(\calP_\e)$ extended arbitrarily to $X,Y\in\frakX(U)$,
\[
\nabla^\g_XY=\nabla^\gEps_XY  +\II(X,Y),
\]
where $\II:(TU)^\parallel\times(TU)^\parallel\to(TU)^\bot$ is the second fundamental form of the level sets of $r$, and is given by,
\[
\II(X,Y)=(\nabg_XY)^\bot.
\]

Rather than working with the second fundamental form, we work with the shape operator $S\in\Omega^1(U;TU)$ and the scalar second fundamental form, $\frakh\in\Omega^{1,1}(U)$ defined by
\[
S=\nabg\dr
\Textand
\frakh = \Hess_\g r=\dg(dr)^T.
\]
They satisfy the following properties
(see \cite[p.~236]{Lee18},  noting the sign difference as $\dr$ points in our case inward):
 
\begin{enumerate}
\item Their specification is equivalent to the specification of the second fundamental form,
\[
\II(X,Y) = -(S(X),Y)_\g\,\dr=-\frakh(X;Y)\,\dr.
\]
\item$\frakh$ is symmetric,
\[
\frakh^T=\frakh.
\]
\item $\frakh$ has no normal components,
\[
\idr \frakh = 0
\Textand
i^{V}_{\dr}\frakh = 0,
\]
so that on every level set $\calP_\e$, $\frakh$ can be identified with $\frakh_\e = \jEpsStar(\jEpsStar \frakh^T)^T\in  \Omega^{1,1}(\calP_\e)$, and $S$ can be identified with $S_\e\in \Omega^1(\calP_\e; T\calP_\e)$ given by $(S_\e(X),Y)_{\g_\e}=\frakh_\e(X;Y)$. 
\item The Gauss equation reads
\beq
\Rm_\gEps= \jEpsStar(\jEpsStar(\Rm_\g^T)^T) + \half\frakh_\e\wedge \frakh_\e,
\label{eq:Gauss_equation_Invariant}
\eeq
where $\Rm_\gEps\in\Omega^{2,2}(\calP_\e)$ is the curvature tensor of the level set $(\calP_\e,\gEps)$.  
\item The Codazzi equation reads
\beq
d^{\nabla^\gEps}\frakh_\e= -\jEpsStar(\jEpsStar(i_{\dr}^V \Rm_\g)^T)^T.
\label{eq:Codazzi_equation}
\eeq
\end{enumerate}

Thus, if $X,Y, \in\frakX(\calP_\e)$ are extended arbitrarily into vector fields on $U$, then
\beq
\label{eq:connections_with_S}
\nabg_XY|_{\mathcal{P}_\e}  = \nabla^{\gEps}_X Y - \frakh_\e(X;Y) \, \dr|_{\mathcal{P}_\e}.
\eeq

We further extend the decomposition to $\Lambda^kT^*U|_{U}$, obtaining an orthogonal decomposition
\[
\Lambda^kT^*U=(\frakt\Lambda^kT^*U)\oplus (\frakn\Lambda^kT^*U),
\]
where
\[
\frakt\omega(X_1,\dots ,X_k)=\omega(X_1^\parallel,\dots ,X_k^\parallel) 
\textand
\frakn\omega=\omega-\frakt\omega.
\]
From the alternation property,
\[
\begin{split}
&(\frakt\Lambda^kT^*U)|_{\calP_\e} \simeq \Lambda^kT^*\calP_\e.
\end{split} 
\]

\begin{lemma}
\label{lem:4.2}
Let $\omega\in\Omega^k(U)$. Then, on $\calP_\e$, 
\beq
\jStar_\e \frakt\omega=\jStar_\e\omega,
\label{eq:tangent_operator_and_pullback}
\eeq
and
\beq
\frakn\omega=dr\wedge \idr\omega.
\label{eq:normal_operator_and_interior}
\eeq
Moreover, $\frakt\omega$ is completely determined by $\jEpsStar\omega$, and
$\frakn\omega$ is completely determined by $\jEpsStar \idr\omega$, which are both forms on $\calP_\e$.
\end{lemma}

We introduce the following level-set projections,
\[
\Pt:\Omega^k(U)\to \Omega^k(\calP_\e)
\Textand
\Pn:\Omega^k(U)\to \Omega^{k-1}(\calP_\e)
\]
by
\[
\Pt\omega = \jEpsStar\omega
\Textand 
\Pn\omega = \jEpsStar\idr\omega.
\]
It follows from \lemref{lem:4.2} that $\omega$ is uniquely determined by $\Pt\omega$ and $\Pn\omega$. These projections  naturally extend to $\E$-valued forms. For $\alpha = \Omega^k(U)$ and $\beta = \Omega^*(U)$,
\beq
\begin{aligned}
\Pt(\alpha\wedge\beta) &= \Pt\alpha\wedge\Pt\beta \\
\Pn(\alpha\wedge\beta) &= \Pn\alpha\wedge\Pt\beta + (-1)^k \, \Pt\alpha\wedge\Pn\beta.
\end{aligned}
\label{eq:Pt_Pn_wedge}
\eeq

The level-set projection operators satisfy the following commutation relations with the Hodge-dual:

\begin{lemma}
\label{lem:PtPnStar}
For $\omega\in\Omega^k(\M)$,
\beq
\label{eq:Hodge_star_normal_tangentA} 
\begin{aligned}
\Pt  \starG\omega &= (-1)^{d+1}  \stargEps \Pn \omega \\
\Pn\starG\omega &= (-1)^{d+k+1}  \stargEps \Pt \omega.
\end{aligned}
\eeq
where $\starG$ and $\stargEps$ are the respective Hodge-dual operators of $\g$ and $\gEps$.
\end{lemma}

The next lemma is concerned with the commutation between the level-set projection operator $\Pt$ and covariant differentiation:

\begin{lemma}
\label{lem:4.3}
For all $\omega\in\Omega^*(U;\E|_U)$ and $X\in \frakX^\parallel(U)  \simeq \frakX(\calP_\e)$,
\beq
\Pt\nabE_X\omega = \nabE_X \Pt\omega + (i_X^{V}\frakh_\e)\wedge \Pn\omega.
\label{eq:commutator_Pt_nabla}
\eeq
\end{lemma}

The next two lemmas are concerned with the commutation between the level-set projection operators and the exterior differential and co-differential:

\begin{lemma}
\label{lem:4.4}
For all $\omega\in\Omega^*(U;\E|_U)$,
\[
\begin{aligned}
\Pt  d^{\nabE}\omega &=   d^{\nabE} \Pt \omega \\
\Pn \delta ^{\nabE} \omega &=  - \delta^{\nabE} \Pn\omega.
\end{aligned}
\]
\end{lemma}

\begin{lemma}
\label{lem:4.5}
For all $\omega\in\Omega^*(U;\E|_{U})$,
\beq
\begin{aligned}
\Pn d^{\nabE}\omega &= - d^{\nabE} \Pn\omega + \Pt \nabE_{\dr}\omega -  \calS_\e \Pt \omega \\
\Pt \delta^{\nabE} \omega &=   \delta^{\nabE}  \Pt \omega -
\Pn \nabE_{\dr} \omega  + \calS_\e^* \Pn  \omega,
\end{aligned}
\label{eq:exterior_covariant_derivative_and_normal} 
\eeq
where  $\calS_\e\in\End(\Lambda^kT^*\calP_\e\otimes\E|_{\calP_\e})$ is
given by
\beq
(\calS_\e \omega)(X_1,\dots,X_k) = \sum_{j=1}^k \omega(X_1,\dots,S_\e(X_j),\dots,X_k),
\label{eq:def_calSe}
\eeq
and  $\calS^*\in\End(\Lambda^kT^*\calP_\e\otimes\E|_{\calP_\e})$ is given by
\[
\calS_\e^*\omega =  (-1)^{dk} \stargEps \calS_\e \stargEps \omega.
\]
\end{lemma}

\subsection{Boundary constructions for double forms}

In the case of elements in $\VectorFormsKM$, the tangential and normal decomposition near the boundary can be repeated for the vector part as well.
Every $\psi\in\VectorFormsKM$ is uniquely determined near the boundary by
\[
\begin{aligned}
&\Pt(\Pt \psi^T)^T\in\Omega^{k,m}(\calP_\e)
&\qquad
&\Pn(\Pt \psi^T)^T\in\Omega^{k-1,m}(\calP_\e)
\\
&\Pt(\Pn\psi^T)^T\in\Omega^{k,m-1}(\calP_\e)
&\qquad
&\Pn(\Pn \psi^T)^T\in\Omega^{k-1,m-1}(\calP_\e).
\end{aligned}
\]

We define the level-set projection operators, $\Omega^{*,*}(U) \to \VectorFormsPeps$,
\beq
\begin{aligned}
\Ptt\psi &= \Pt(\Pt \psi^T)^T  = (\Pt(\Pt \psi)^T)^T \\
\Pnt\psi &= \Pn(\Pt \psi^T)^T  =  (\Pt(\Pn \psi)^T)^T  \\
\Ptn\psi &= \Pt(\Pn \psi^T)^T = (\Pn (\Pt \psi)^T)^T  \\
\Pnn\psi &=\Pn(\Pn \psi^T)^T = (\Pn(\Pn \psi)^T)^T.
\end{aligned}
\label{eq:Ptt_et_al}
\eeq
With a slight abuse of notation,
\[
\psi|_{\calP_\e} = \Ptt\psi + dr\wedge\Pnt\psi + (dr)^T\wedge\Ptn\psi + dr\wedge(dr)^T\wedge\Pnn\psi
\]
(strictly speaking, $\psi|_{\calP_\e} \in \VectorFormsM|_{\calP_\e}$ whereas e.g., $\Ptt\psi\in\VectorFormsPeps$).
We note the following immediate identities
\[
\begin{aligned}
&(\Ptt \psi)^T = \Ptt \psi^T
&\qquad 
&(\Pnn \psi)^T = \Pnn \psi^T \\
&(\Ptn \psi)^T = \Pnt \psi^T
&\qquad 
&(\Pnt \psi)^T = \Ptn \psi^T.
\end{aligned}
\]

As a direct consequence of \eqref{eq:Pt_Pn_wedge},
these projection operators combine with the wedge product in $\VectorFormsU$ as follows:
for $\psi\in \Omega^{k,m}(U)$ and $\xi\in\Omega^{*,*}(U)$ that
\beq
\begin{aligned}
\Ptt(\psi\wedge\xi) &= \Ptt\psi\wedge\Ptt\xi \\
\Ptn(\psi\wedge\xi) &= \Ptn\psi\wedge\Ptt\xi  + (-1)^m \, \Ptt\psi\wedge\Ptn\xi  \\
\Pnt(\psi\wedge\xi) &=   \Pnt\psi\wedge\Ptt\xi  + (-1)^k \, \Ptt\psi\wedge\Pnt\xi  \\
\Pnn(\psi\wedge\xi) &=   \Pnn\psi\wedge\Ptt\xi  + (-1)^k \, \Ptn\psi\wedge\Pnt\xi  \\
&\qquad+ (-1)^m \, \Pnt\psi\wedge\Ptn\xi  + (-1)^{k+m} \, \Ptt\psi\wedge\Pnn\xi .
\end{aligned}
\label{eq:Ptt_Ptn_Pnt_Pnn_wedge}
\eeq

In this context, $\frakh_\e\in\Omega^{1,1}(\calP_\e)$ is related to the scalar second fundamental form $\frakh\in\Omega^{1,1}(U)$ via
\[
\frakh_\e = \Ptt \frakh,
\]  
and the Gauss-Codazzi equations \eqref{eq:Gauss_equation_Invariant},\eqref{eq:Codazzi_equation} take the form
\beq
\Rm_\gEps= \Ptt\Rm_\g + \half\frakh_\e\wedge \frakh_\e
\Textand
d^{\nabla^\gEps}\frakh_\e= -\Ptn \Rm_\g.
\label{eq:Codazzi_equation2}
\eeq
For the sake of completion, we note that
\[
\Pnn \Rm_\g=\Ptt\nabg_{\dr}\frakh+\calS_\e\frakh_\e,
\]
where $\calS_\e$ is given by \eqref{eq:def_calSe} \cite[p.~97]{Pet16}. 
As an immediate consequence of \lemref{lem:PtPnStar}:

\begin{lemma}
For $\psi\in\VectorFormsKM$,
\beq
\label{eq:Hodge_star_normal_tangentAA} 
\begin{aligned}
&\Ptt  \starG\psi  =  (-1)^{d+1} \stargEps \Pnt \psi
&\qquad
&\Ptn  \starG\psi  = (-1)^{d+1}  \stargEps \Pnn \psi \\
&\Pnt\starG\psi  = (-1)^{d+k+1} \stargEps \Ptt \psi
&\qquad
&\Pnn\starG\psi = (-1)^{d+k+1}  \stargEps \Ptn \psi \\
&\Ptt  \starG^V\psi  =  (-1)^{d+1} \stargEps^V \Ptn \psi
&\qquad
&\Pnt  \starG^V \psi  = (-1)^{d+1}  \stargEps^V \Pnn \psi \\
&\Ptn\starG^V\psi  = (-1)^{d+m+1} \stargEps^V \Ptt \psi
&\qquad
&\Pnn\starG^V\psi = (-1)^{d+m+1}  \stargEps^V \Pnt \psi.
\end{aligned}
\eeq
\end{lemma}

The next definitions are pertinent to the scalar second fundamental form $\frakh_\e$. In the constructions of \defref{def:3.2}, we replace $(\M,\g)$ with $(\calP_{\e},\gEps)$, and set $A=\frakh_\e$,
to obtain the graded tensorial operators,
\[
\begin{split}
(\frakh_\e \wedge) &: \Lambda^{k,m}T^*\calP_\e \to \Lambda^{k+1,m+1}T^*\calP_\e \\
\trhe &: \Lambda^{k,m}T^*\calP_\e \to \Lambda^{k-1,m-1}T^*\calP_\e \\
\ihe &: \Lambda^{k,m}T^*\calP_\e \to \Lambda^{k+1,m-1}T^*\calP_\e \\
\ihe^* &: \Lambda^{k,m}T^*\calP_\e \to \Lambda^{k-1,m+1}T^*\calP_\e
\end{split}
\]
which satisfy as in \lemref{lem:3.4} and \lemref{lem:3.5},
\[
(\trhe \psi,\vp)_{\gEps} = (\psi,\frakh_\e\wedge \vp) _{\gEps}
\Textand
(\ihe^*\psi,\vp)_{\gEps} = (\psi,\ihe\vp)_{\gEps}
\]
and
\[
(\trhe \psi)^T = \trhe\psi^T 
\Textand
(\ihe \psi)^T = \ihe^*\psi^T.
\]

\subsection{Commutation between level-set projections and exterior differential operators}
\label{sec:commutation_d_delta}

In this section we derive commutation relations between the level-set projection operators and the exterior covariant differentials and co-differentials.

\begin{lemma}
\label{lem:4.11}
Let $V\in\Omega^{0,*}(U)$. Then,
\beq
\Ptt \nabg V = \nabgEps \Ptt V + \frakh_\e\wedge \Ptn V.
\label{eq:Ptt_nabla_V}
\eeq
\end{lemma}

\begin{lemma}
\label{lem:4.12}
For $\psi\in\VectorFormsU$, 
\beq
\begin{aligned}
\Ptt\dg\psi &= \dgEps \Ptt\psi + \frakh_\e \wedge \Ptn\psi \\
\Ptt\dgV\psi &= \dgEpsV \Ptt\psi + \frakh_\e \wedge \Pnt\psi \\
\Ptn \dg\psi &= \dgEps \Ptn\psi - \ihe \Ptt \psi  \\
\Pnt \dgV\psi &= \dgEpsV \Pnt\psi - \ihe^* \Ptt \psi.
\end{aligned}
\label{eq:dnabla_pure_tangent} 
\eeq
\end{lemma}

\begin{lemma}
\label{lem:4.13}
For $\psi\in\VectorFormsM$,
\beq
\begin{aligned}
\Pnt \dg\psi &= - \dgEps \Pnt\psi - \frakh_\e\wedge \Pnn\psi - \calS_\e \Ptt\psi + \Ptt \nabg_{\dr}\psi \\
\Ptn \dgV\psi &= - \dgEpsV \Ptn\psi - \frakh_\e\wedge \Pnn\psi - (\calS_\e (\Ptt\psi)^T)^T + \Ptt \nabg_{\dr}\psi \\
\Pnn\dg\psi &= -\dgEps\Pnn\psi + \ihe\Pnt \psi  - \calS_\e \Ptn\psi + \Ptn \nabg_{\dr}\psi \\
\Pnn\dgV\psi &= -\dgEpsV\Pnn\psi + \ihe^*\Ptn \psi  - (\calS_\e (\Pnt\psi)^T)^T + \Pnt \nabg_{\dr}\psi.
\end{aligned}
\label{eq:dnabla_tangent_normal} 
\eeq
\end{lemma}

The following commutation relations follow by substituting $\psi\mapsto \starG\psi$ or $\psi\mapsto \starG^V\psi$ in \eqref{eq:dnabla_pure_tangent} and \eqref{eq:dnabla_tangent_normal}:

\begin{lemma}
For $\psi\in\VectorFormsM$,
\beq
\begin{aligned}
\Pnn \deltag\psi &=-\deltagEps\Pnn\psi - \trhe \Pnt\psi \\
\Pnn \deltagV\psi &=-\deltagEpsV\Pnn\psi - \trhe \Ptn\psi \\
\Pnt \deltag\psi &= -\deltagEps\Pnt\psi +  \ihe^*  \Pnn \psi  \\
\Ptn \deltagV\psi &= -\deltagEpsV\Ptn\psi +  \ihe  \Pnn \psi \\
\Ptn\deltag\psi &= \deltagEps\Ptn\psi + \trhe\Ptt\psi + \calS_\e^* \Pnn \psi  - \Pnn\nabg_{\dr}\psi \\
\Pnt\deltagV\psi &= \deltagEpsV\Pnt\psi + \trhe\Ptt\psi + (\calS_\e^* (\Pnn \psi)^T)^T  - \Pnn\nabg_{\dr}\psi \\
\Ptt\deltag\psi &= \deltagEps\Ptt\psi - \ihe^* \Ptn\psi  + \calS^*_\e  \Pnt\psi -  \Pnt\nabg_{\dr}\psi \\
\Ptt\deltagV\psi &= \deltagEpsV\Ptt\psi - \ihe \Pnt\psi  + (\calS^*_\e  (\Ptn\psi)^T)^T -  \Ptn\nabg_{\dr}\psi.
\end{aligned}
\label{eq:delta_pure_normal} 
\eeq
\end{lemma}

\subsection{Commutation between level-set projections and second-order operators}
\label{sec:commutation relations}

We consider next the commutators of the second-order differential operators $\Hg,\Hg^*,\Fg^*,\Fg$, and the level-set projection operators $\Ptt,\Ptn,\Pnt,\Pnn$.

\begin{lemma}
\label{lem:4.15}
Let $\psi\in\VectorFormsU$. Then,
\beq
\begin{aligned}
& \Ptt \Hg\psi = H_{\gEps}\Ptt \psi + \frakh_\e \wedge\frakT_\e\psi  + \frakR_\g \psi \\
& \Pnn \Hg^*\psi = H_{\gEps}^*\Pnn \psi + \trhe \frakT_\e^*\psi
+ (-1)^{dk + dm} \stargEps\stargEps^V \frakR_\g \stargEps^V\stargEps \psi \\
& \Pnt \Fg^*\psi = -F_{\gEps}\Pnt \psi +  \ihe^* \frakF^*_\e \psi
+ (-1)^{dk} \stargEps \frakR_\g \stargEps \psi \\
& \Ptn \Fg\psi = -F_{\gEps}^*\Ptn \psi + \ihe \frakF_{\e}\psi
+ (-1)^{dm} \stargEps^V \frakR_\g \stargEps^V \psi,
\end{aligned}
\label{eq:hessian_and_tangent_parts} 
\eeq
where the commutators 
\[
\begin{aligned}
&\frakT_\e : \Omega^{k,m}(U)\to\Omega^{k,m}(\calP_\e) 
&\qquad
&\frakT_\e^* : \Omega^{k,m}(U)\to\Omega^{k-1,m-1}(\calP_\e) \\
&\frakF^*_\e : \Omega^{k,m}(U)\to\Omega^{k-1,m}(\calP_\e) 
&\qquad
&\frakF_{\e} : \Omega^{k,m}(U)\to\Omega^{k,m-1}(\calP_\e),
\end{aligned}
\]
are the first-order differential operators
\beq
\begin{aligned}
\frakT_\e\psi &= \smallhalf \brk{\Pnt \dg\psi-\dgEps \Pnt\psi}+\smallhalf \brk{\Ptn\dgV\psi- \dgEpsV \Ptn\psi} \\
\frakT_\e^*\psi &=  -\smallhalf\brk{ \Ptn\deltag\psi+\deltagEps \Ptn\psi} - \smallhalf\brk{\Pnt\deltagV\psi+\deltagEpsV\Pnt\psi} \\
\frakF^*_\e\psi &= \smallhalf\brk{\Pnn \dgV\psi- \dgEpsV \Pnn \psi} -\smallhalf\brk{\Ptt \deltag \psi+\deltagEps \Ptt  \psi} \\
\frakF_{\e}\psi &= \smallhalf\brk{\Pnn \dg\psi-  \dgEps \Pnn \psi} -\smallhalf\brk{\Ptt\deltagV \psi+ \deltagEpsV \Ptt  \psi},
\end{aligned}
\label{eq:T_commutator} 
\eeq
and 
\[
\frakR_\g \psi = \smallhalf\brk{\Pnt\Rm_\g \wedge \Ptn\psi + \Ptn\Rm_\g \wedge \Pnt\psi}.
\]
Moreover,
\beq
\begin{split}
& \frakT^*_\e= (-1)^{dk+dm + k +m} \stargEps\stargEps^{V}\frakT_\e\starG\starG^{V} \\
& \frakF^*_\e=(-1)^{dk+k} \stargEps\frakT_\e\starG\\
& \frakF_\e= (-1)^{dm + m} \stargEps^{V}\frakT_\e\starG^{V}.
\end{split}
\label{eq:second_order_boundary_operators_duality}
\eeq
\end{lemma}

\subsection{Boundary integrals}
\label{sec:integral_theroy} 

We complete the graded algebra $\Omega^{*,*}_{c}(\M)$ in Sobolev norms to obtain the graded Banach spaces $W^{s,p}\VectorFormsM$. The Sobolev theory for vector-valued forms holds verbatim for $W^{s,p}\VectorFormsM$, as $\Omega^{k,m}_{c}(\M)=\Gamma_{c}(\Lambda^{k,m}T^*\M)$. In particular, the embedding and trace theorems are valid. Thus, $\Hg,\Fg^*$ and their duals all extend naturally to second-order differential graded operators of type $W^{s,p}\VectorFormsM\to W^{s-2,p}\VectorFormsM$.

Henceforth, $\PttD$, $\PtnD$, $\PntD$ and $\PnnD$ and $\frakT$, $\frakT^*$, $\frakF^*$ and $\frakF$ denote the boundary operators corresponding to $\calP_0=\dM$. We denote the inclusion by $j = j_0:\dM\to\M$, and by $\gD = j^\star\g$ the induced metric at the boundary. In view of our choice of $\calN = -\partial_r|_{\dM}$,  Green's formula \eqref{eq:green's_formula} for double forms can be written as
\beq
\begin{split}
\bra\dg\psi,\eta\ket &= \bra\psi,\deltag\eta\ket - \int_{\dM} (\PtD\psi,\PnD\eta)_{\gD,\g}\,\VolumeD \\
&= \bra\psi,\deltag\eta\ket - \int_{\dM}\Brk{(\PttD\psi,\PntD\eta)_\gD + (\PtnD\psi,\PnnD\eta)_\gD}\,\VolumeD.
\end{split}
\label{eq:Green2}
\eeq
Replacing $\psi$ and $\eta$ by their transpose,
\beq
\bra\dgV\psi,\eta\ket= \bra\psi,\deltagV\eta\ket - \int_{\dM}\Brk{(\PttD\psi,\PtnD\eta)_\gD +
(\PntD\psi,\PnnD\eta)_\gD}\,\VolumeD.
\label{eq:Green3}
\eeq

\begin{lemma}
\label{lem:4.16}
Let  $\psi\in W^{2,p}\VectorFormsM$ and $\eta\in W^{2,q}\VectorFormsM$ with $1/p+1/q=1$. Then,
\beq
\begin{aligned}
\bra \Hg\psi, \eta\ket &= \bra \psi, \Hg^*\eta\ket +  \int_{\dM}\Brk{(\PttD\psi,\frakT^* \eta)_\gD - (\frakT\psi,\PnnD \eta)_\gD}\VolumeD \\
\bra \Fg\psi, \eta\ket &= \bra \psi, \Fg^*\eta\ket + \int_{\dM}\Brk{(\PtnD\psi,\frakF^*\eta)_\gD-(\frakF\psi,\PntD\eta)_\gD}\VolumeD .
\end{aligned}
\label{eq:H_general_integral}
\eeq
\end{lemma}

Assume graded tensorial operations $D_\g:\Lambda^{k,m}T^*\M\rightarrow \Lambda^{k+1,m+1}T^*\M$ and $S_\g:\Lambda^{k,m}T^*\M\rightarrow \Lambda^{k+1,m-1}T^*\M$ satisfying
\[
D_\g\psi^T=(D_\g\psi)^T 
\Textand 
S_\g\psi=(S_\g^*\psi^T)^T,
\]
where $D_\g^*:\Lambda^{k,m}T^*\M\rightarrow \Lambda^{k-1,m-1}T^*\M$ and $S_\g^*:\Lambda^{k,m}T^*\M\rightarrow \Lambda^{k-1,m+1}T^*\M$ are their metric duals. 

The introduction of the tensorial operators $D_\g$ and $S_\g$ allows us to analyze a larger family of second-order graded differential operators.  
We introduce the operators $\bHg:\Omega^{k,m}(\M)\rightarrow\Omega^{k+1,m+1}(\M)$ and  $\bFg:\Omega^{k,m}(\M)\rightarrow\Omega^{k+1,m-1}(\M)$, along with their $L^2$-duals,
\beq
\begin{gathered}
\bHg=\Hg+  D_\g \qquad\qquad \bHg^*=\Hg^*+  D_\g^* \\
\bFg^*=\Fg^*+  S_\g^* \qquad\qquad \bFg=\Fg+  S_\g.
\end{gathered}
\label{eq:extended_F_H}
\eeq

Note that $\bHg$ and $\bHg^*$ commute with transposition and $\bFg=(\bFg^*(\cdot)^T)^T$. 
As tensorial operators do not yield boundary terms in integration by parts, formulas \eqref{eq:H_general_integral} turn into
\beq
\begin{aligned}
\bra \bHg\psi,\eta\ket &=
\bra \psi,\bHg^*\eta\ket+\int_{\dM}\Brk{(\PttD\psi,\frakT^*\eta)_{\gD}-(\frakT\psi,\PnnD\eta)_{\gD}}\VolumeD \\
\bra \bFg\psi,\eta\ket &=
\bra \psi,\bFg^*\eta\ket+\int_{\dM}\Brk{(\PtnD\psi,\frakF^*\eta)_{\gD}-(\frakF\psi,\PntD\eta)_{\gD}}\VolumeD.
\end{aligned}
\label{eq:integration_by_parts_F_H_extended}
\eeq

To shorten notations, we adopt the following convention: let $\{T_1,\dots,T_n\}$ be a set of linear operators on a linear space; then
\[
\ker(T_1,\dots,T_n) = \bigcap_{i=1}^n \ker T_i.
\]

As a direct consequence of the integration by parts formula:

\begin{corollary}
\label{cor:integration_by_parts_second_order}
Let $\psi\in W^{2,p}\VectorFormsM$ and $\eta\in W^{2,q}\VectorFormsM$ (the type is implied in each context) with $1/p+1/q=1$. Then,
\[
\begin{aligned}
&\psi \in \ker(\PttD,\frakT)
\qquad &\text{implies}& \qquad 
&\bra \bHg\psi, \eta\ket= \bra\psi,\bHg^*\eta\ket \\
&\psi \in \ker(\PnnD,\frakT^*)
\qquad &\text{implies}& \qquad 
&\bra \bHg^*\psi, \eta\ket = \bra\psi,\bHg\eta\ket \\
&\psi \in \ker(\PtnD,\frakF)
\qquad &\text{implies}& \qquad 
&\bra \bFg\psi, \eta\ket = \bra\psi,\bFg^*\eta\ket \\
&\psi \in \ker(\PntD,\frakF^*)
\qquad &\text{implies}& \qquad 
&\bra \bFg^*\psi, \eta\ket= \bra\psi,\bFg\eta\ket.
\end{aligned}
\]
\end{corollary}

\section{Elliptic theory of double bilaplacians}
\label{sec:elliptic_theory}

We introduce the fourth-order linear differential operator
\[
\Bg = \bHg \bHg^* + \bHg^* \bHg + \bFg^* \bFg + \bFg \bFg^*,
\]
which we a double bilaplacian, in analogy with the Hodge laplacian $\Delta = d\delta + \delta d = d d^* + d^* d$.
In this section we show that the differential equation $\Bg\psi = \chi$ equipped with several sets of boundary conditions is regular elliptic.

\subsection{Extension of boundary data}
\label{subsec:extension}

On several occasions we will need to extend boundary data to a neighborhood of the boundary. 

\begin{lemma}
\label{lemma:prescribe_boundary_conditions}
Let $\phi_i,\rho_i\in \VectorFormsdM$, $i=1,\dots,4$ (the type is implied by the context). Then, there exists a vector-valued form $\lambda\in\VectorFormsM$ satisfying
\[
\begin{gathered}
(\PttD,\PntD,\PtnD,\PnnD)\lambda = (\phi_1,\phi_2,\phi_3,\phi_4) \\
(\frakT,\frakF^*,\frakF,\frakT^*)\lambda = (\rho_1,\rho_2,\rho_3,\rho_4).
\end{gathered}
\]
Moreover, if $\phi_i\in W^{s-1/p,p}\VectorFormsdM$ and $\rho_i\in W^{s-1-1/p,p}\VectorFormsdM$, 
then $\lambda$ can be chosen such that
\beq
\|\lambda\|_{W^{s,p}(\M)} \lesssim \sum_{i=1}^4 \brk{ \|\phi_i\|_{W^{s-1/p,p}(\dM)} +
\|\rho_i\|_{W^{s-1-1/p,p}(\dM)} }.
\label{eq:estimate_1st_bc}
\eeq
\end{lemma}

\begin{lemma}
\label{lemma:prescribed_boundary_data_second_order}
Let $\phi_i,\rho_i\in \VectorFormsdM$, $i=1,\dots,4$ (the type is implied by the context). Then, there exists a $\lambda\in \VectorFormsM$, satisfying
\beq
\begin{gathered}
\lambda|_{\dM}=0, \qquad \nabg_{\dr}\lambda|_{\dM}=0 \\
(\PttD\bHg^*,\PtnD\bFg^*,\PntD\bFg,\PnnD\bHg)\lambda = (\phi_1,\phi_2,\phi_3,\phi_4) \\
(\frakT\bHg^*,\frakF\bFg^*,\frakF^*\bFg,\frakT^*\bHg)\lambda = (\rho_1,\rho_2,\rho_3,\rho_4).
\end{gathered}
\label{eq:prescribed_boundary_data_second_order}
\eeq
Moreover, if $\phi_i\in W^{s-1/p,p}\VectorFormsdM$ and $\rho_i\in W^{s-1-1/p,p}\VectorFormsdM$, then $\lambda$ can be chosen such that,
\beq
\begin{split}
\Norm{\lambda}_{W^{s+2,p}(\M)}\lesssim \sum_{i=1}^4\brk{\Norm{\phi_i}_{W^{s-1/p,p}(\dM)}+\Norm{\rho_i}_{W^{s-1-1/p,p}(\dM)}}.
\end{split}
\label{eq:sobolev_estimate_boundary_data}
\eeq
\end{lemma}

\begin{corollary}
\label{cor:full_problem_boundary}
Let $\phi_i,\rho_i,\mu_i,\nu_i\in \VectorFormsdM$, $i=1,\dots,4$ (the type is implied by the context). Then, there exists a $\lambda\in \VectorFormsM$ satisfying
\[
\begin{gathered}
(\PttD,\PntD,\PtnD,\PnnD)\lambda = (\phi_1,\phi_2,\phi_3,\phi_4) \\
(\frakT,\frakF^*,\frakF,\frakT^*)\lambda = (\rho_1,\rho_2,\rho_3,\rho_4) \\
(\PttD\bHg^*,\PtnD\bFg^*,\PntD\bFg,\PnnD\bHg)\lambda = (\mu_1,\mu_2,\mu_3,\mu_4) \\
(\frakT\bHg^*,\frakF\bFg^*,\frakF^*\bFg,\frakT^*\bHg)\lambda = (\nu_1,\nu_2,\nu_3,\nu_4).
\end{gathered}
\]
Moreover, if $s\geq 4$ and,
\[
\begin{gathered}
\phi_i\in W^{s-1/p,p}\VectorFormsdM, \qquad \rho_i\in W^{s-1-1/p,p}\VectorFormsdM,
\\ \mu_i\in W^{s-2-1/p,p}\VectorFormsdM, \qquad \nu_i\in W^{s-3-1/p,p}\VectorFormsdM
\end{gathered}
\] 
then $\lambda$ can be chosen such that
\beq
\begin{split}
\Norm{\lambda}_{W^{s,p}(\M)}&\lesssim \sum_{i=1}^4\brk{\Norm{\phi_i}_{W^{s-1/p,p}(\dM)}+\Norm{\rho_i}_{W^{s-1-1/p,p}(\dM)}} \\
&+\sum_{i=1}^4\brk{\Norm{\mu_i}_{W^{s-2-1/p,p}(\dM)}+\Norm{\nu_i}_{W^{s-3-1/p,p}(\dM)}}.
\end{split}
\label{eq:sobolev_estimate_boundary_data_full}
\eeq
\end{corollary}

In the above construction, the specification of this particular choice of 16 boundary operators amounts to the specification at the boundary of the value of the form, along with its first three normal derivatives.

\subsection{Elliptic theory} 

There are several methods for establishing the ellipticity of boundary-value problems. The approach we use here is the so-called Lopatinskij-Shapiro criterion, which also provides a direct computational approach \cite[Sec.~20.1]{Hor07}. 

Let $\M$ be a compact manifold with boundary, let $\E_1\to\M$ and $\E_2\to\M$ be Riemannian vector bundles over $\M$, and let $\BRK{\mathbb{F}_j\to\dM}_{j=1}^{\l}$ be a collection of Riemannian vector bundles over $\dM$. 
Let $A:\Gamma(\E_1)\to\Gamma(\E_2)$ be a differential operator of order $m$, and let $B_j:\Gamma(\E_1)\to\Gamma(\mathbb{F}_j)$ be differential operators of order $m_j< m$. For data $\eta\in \Gamma(\E_2)$ and $\theta_j\in\Gamma(\mathbb{F}_j)$, we study a boundary-value problem for $\sigma\in \Gamma(\E_1)$,
\beq
\begin{aligned}
&A\sigma=\eta \qquad \text{on $\M$} \\ 
&B_j\sigma=\theta_j \qquad \text{on $\dM$} \qquad j=1,\dots,l. 
\end{aligned}
\label{eq:boundary_value_problem}
\eeq

The Lopatinskij-Shapiro criterion is local, stated by means of the principal symbols of the differential operators \cite[Sec.~1.6]{Sch95b}, \cite[Ch.~2, Sec.~9]{Tay11a}.
Principal symbols can be defined in a coordinate-free manner using the language of jet bundles; we follow the more common approach using trivializations and coordinates. 

Using trivializations for the vector bundles $\E_1,\E_2$, and 
local coordinates for $\M$, a differential operator $L:\Gamma(\E_1)\to\Gamma(\E_2)$ of order $m$ takes the form
\beq
L\sigma(x)=\sum_{|\alpha|\leq m} a_{\alpha}(x) D^{\alpha}\sigma(x),
\label{eq:L_local_form}
\eeq
where, $D^{\alpha}=D_1^{\alpha_1}...D_d^{\alpha_d}$ and $D_j=-\imath\, \partial_j$; the coefficients $a_{\alpha}(x)$ are of type $a_{\alpha}(x) \in\Hom(\R^{N_1},\R^{N_2})$, where $N_1,N_2$ are the ranks of $\E_1,\E_2$.

\begin{definition}
Let $L:\Gamma(\E_1)\to\Gamma(\E_2)$ have local form \eqref{eq:L_local_form}. 
For $(x,\xi)\in T^*\M$, the linear map $P_L(x,\xi):\E_1|_x\to \E_2|_x$, given in local trivializations by
\[
P_L(x,\xi)\sigma=\sum_{|\alpha|=m} a_{\alpha}(x)\xi^{\alpha}\sigma,
\]
is called the principal symbol of $L$ (on the right-hand side, $\xi\in\R^d$ is the coordinate representation of the covector, and $\xi^\alpha = \xi_1^{\alpha_1}\dots\xi_d^{\alpha_d}$). 
\end{definition}

It can be shown \cite[pp.~176--177]{Tay11a} that for all $(x,\xi)\in T^*\M$, $P_L(x,\xi):\E_1|_x\to \E_2|_x$ is well-defined and has a local characterization,
\beq
\lim_{\lambda\to\infty} \lambda^{-m}L(e^{i\lambda\phi} \sigma) = P_L(x,d\phi|_x)\sigma,
\label{eq:formula_symbol}
\eeq
where $\phi$ is a smooth real-valued function in a neighborhood of $x$, and $\sigma$ on the left-hand side has been extended in a neighborhood of $x$.
If $L_1:\Gamma(\E_1)\to\Gamma(\E_2)$ and $L_2:\Gamma(\E_2)\to\Gamma(\E_3)$ then,
\beq
P_{L_2\circ L_1}(x,\xi)=P_{L_2}(x,\xi)\circ P_{L_1}(x,\xi).
\label{eq:principel_homomo}
\eeq
Furthermore, if the vector bundles are endowed with metrics, then one can select a normal frame at $x$ and show that the symbol commutes with metric operations once localized. The following symbols are easy to calculate by means of \eqref{eq:formula_symbol} (see \cite[p.~181]{Tay11a} for the symbols of $d$ and $\delta$):
\[
\begin{aligned}
& -\imath\, P_{\dg}(x,\xi)\sigma=\xi\wedge\sigma 
& \qquad 
& -\imath\, P_{\deltag}(x,\xi)\sigma=- \ixi\sigma \\
& -\imath\, P_{\dgV}(x,\xi)\sigma=\xi^T\wedge\sigma 
& \qquad 
& -\imath\, P_{\deltagV}(x,\xi)\sigma=- \ixiV\sigma.
\end{aligned}
\]

Combining with \eqref{eq:principel_homomo}, and noting how zero-order terms do not affect the symbol, we obtain the principal symbols of the second-order operators \eqref{eq:extended_F_H},
\beq
\begin{aligned}
& P_{\bHg}(x,\xi)\sigma=-\xi^T\wedge\xi\wedge\sigma 
& \qquad 
& P_{\bHg^*}(x,\xi)\sigma= -\ixi\ixiV\sigma \\
& P_{\bFg^*}(x,\xi)\sigma=\xi^T\wedge \ixi\sigma 
& \qquad 
& P_{\bFg}(x,\xi)\sigma=\xi\wedge \ixiV\sigma.
\end{aligned}
\label{eq:symbols_of_H_H^*_F_F^*}
\eeq
The symbol of the associated double bilaplacian is
\[
P_{\Bg}(x,\xi)\sigma=|\xi|^4\sigma. 
\]

Back to the boundary-value problem \eqref{eq:boundary_value_problem}, we proceed with the following manipulation \cite[pp.~53--54]{Sch95b}. Let $U$ be a normal neighborhood of $\dM$ and let $x\in U$. Every $\xi\in T^*_x\M$ can be decomposed into
\[
\xi= \frakt\xi + \xi_d \, dr.
\] 
In the vicinity of the boundary, one can apply on $A$ a partial Fourier transform along the normal direction. Formally, this is done by replacing $\xi_d$ with $\imath\,\partial_s$, where $s\in\R$ is a smooth parameter.
This yields an ordinary differential operator on $\mathbb{C}\otimes\E_1|_x$-valued functions,
\[
P_A(x,\frakt\xi+\imath\, \partial_s\, dr):C^{\infty}(\R, \mathbb{C}\otimes\E_1|_x)\to C^{\infty}(\R, \mathbb{C}\otimes\E_2|_x).
\]
The differential equation
\[
P_A(x,\frakt\xi+\imath\,\partial_s\, dr)\sigma(s)=0
\]
is a linear ODE, hence has globally defined solutions. 
We denote its globally $\R_+$-bounded solutions by
\[
\calM^+_{x,\xi}=\BRK{\sigma\in C^{\infty}(\R, \mathbb{C}\otimes\E_1|_x) : \sup_{s\geq 0}|\sigma(s)|<\infty, \ \ P_A(x,\xi+\imath\,\partial_s \, dr)\sigma(s)=0}.
\]

The Lopatinskij-Shapiro condition \cite[p.~233]{Hor07} can at last be stated:

\begin{definition}
\label{def:regular_elliptic}
The boundary-value problem \eqref{eq:boundary_value_problem} is called regular elliptic (in the sense of Lopatinskij-Shapiro) if the following conditions are fulfilled:
\begin{enumerate}
\item For every $(x,\xi)\in T^*\M$, the principal symbol,
\[
P_A(x,\xi):\E_1|_x\to \E_2|_x,
\]
is a linear isomorphism if and only if $\xi\ne 0$.

\item For all $(x,\xi)\in\frakt T^*\M|_{\dM}\simeq T^*\dM$, $\xi\ne 0$, the map $\Xi_{x,\xi}:\calM^+_{x,\xi}\to \mathbb{C}\otimes(\oplus_j \mathbb{F}_j|_x)$ given by
\[
\sigma(s)\mapsto \brk{P_{B_1}(x,\xi+\imath\,\partial_s\, dr)\sigma(0),...,P_{B_l}(x,\xi+\imath\,\partial_s\,dr)\sigma(0)}
\]
is a linear isomorphism.
\end{enumerate}
\end{definition}

Regular elliptic problems satisfy the following key properties:

\begin{theorem}
\label{thm:regular_elliptic_estimates}
Let $A$ and $B_j$ be as in \eqref{eq:boundary_value_problem},  and associate with them the differential operator 
\[
(A,B_1,...,B_l):\Gamma(\E_1)\to \Gamma(\E_2)\times \oplus_j\Gamma(\mathbb{F}_j).
\]
If \eqref{eq:boundary_value_problem} is regular elliptic in the sense of \defref{def:regular_elliptic}, then
\begin{enumerate}
\item The operator
\beq
(A,B_1,...,B_l):W^{s,p}\Gamma(\E_1)\to W^{s-m,p}\Gamma(\E_2)\times \bigoplus_jW^{s-m_j-1/p,p}\Gamma(\mathbb{F}_j)
\label{eq:Fredholm_ellipticity}
\eeq
is Fredholm (i.e., has finite-dimensional kernel and co-kernel) for all $s\geq m$ and $1\leq p < \infty$.

\item There is an elliptic a-priori estimate for all $\sigma \in W^{s,p}\Gamma(\E_1)$,
\beq
\Norm{\sigma}_{W^{s,p}(\M)}\lesssim \Norm{A\sigma}_{W^{s-m,p}(\M)}+\sum_j \Norm{B_j\sigma}_{W^{s-m_j-1/p,p}(\dM)}+\Norm{\sigma}_{L^q(\M)},
\label{eq:elliptic_estimate}
\eeq
where $1\leq q\leq p$ can be chosen arbitrarily. 

\item In particular, if $\sigma$ is a solution to the boundary-value problem \eqref{eq:boundary_value_problem} and the prescribed data $\eta,\theta_j$ is smooth, then $\sigma$ is smooth.

\item The (finite-dimensional) kernel and co-kernel of the Fredholm operator \eqref{eq:Fredholm_ellipticity} are independent of $s$ and $p$. 
\end{enumerate}
\end{theorem}

A proof is given in \cite[Ch.~20]{Hor07}. 

\subsection{Ellipticity of the double bilaplacian}
\label{sec:proof_of_ellipticiity}

In this section we apply the machinery of regular elliptic operators to the double bilaplacian:

\begin{theorem}
\label{thm:elliptic_regularity}
The boundary-value problems defined by each of the operators
\beq
\begin{aligned}
\frakB_\TT &= (\Bg,\PttD,\PtnD,\PntD,\frakT,\frakF^*,\frakF,\PttD\bHg^*,\frakT\bHg^*) \\
\frakB_\NN &= (\Bg,\PnnD,\PtnD,\PntD,\frakT^*,\frakF^*,\frakF,\PnnD\bHg,\frakT^*\bHg) \\
\frakB_\NT &= (\Bg,\PntD,\PnnD,\PttD,\frakF^*,\frakT,\frakT^*,\PntD\bFg,\frakF^*\bFg) \\
\frakB_\TN &= (\Bg,\PtnD,\PnnD,\PttD,\frakF,\frakT,\frakT^*,\PtnD\bFg^*,\frakF\bFg^*)
\end{aligned}
\label{eq:boundary_data_TT}
\eeq
are regular elliptic in the sense of \defref{def:regular_elliptic}.
\end{theorem}

\begin{proof}
The proof is technical. It requires calculating  the symbols of all the operators and verifying the satisfaction of \defref{def:regular_elliptic}. We prove for the first set of operators; the other three follow by replacing $\psi\mapsto\starG\psi$, $\psi\mapsto\starG^{V}\psi$, $\psi\mapsto\starG\starG^{V}\psi$ and recalling  \eqref{eq:duals_stars},\eqref{eq:Hodge_star_normal_tangentAA}, and \eqref{eq:second_order_boundary_operators_duality}
($\bHg$ and $\bFg$ can be replaced interchangeably with $\Hg$ and $\Fg$ as they have the same symbols).

In the context of \defref{def:regular_elliptic}, $\E_1 = \E_2 = \Lambda^{k,m}T^*\M$, i.e.,
\[
\dim_\R\E_1 = \dim_\R\E_2 = \binom{d}{k} \binom{d}{m}.
\]
The fact that $P_{\Bg}(x,\xi) = |\xi|^4$ is an isomorphism if and only if $\xi\ne 0$ is immediate. We therefore  only need to verify the second condition in \defref{def:regular_elliptic}. 
We start by noting that for all $(x,\xi)\in T^*\M|_{\dM}$,
\[
\begin{split}
& P_{\PttD}(x,\xi):\Lambda^{k,m}T^*_x\M\to \Lambda^{k,m}T^*_x\dM  \equiv \bbF_1|_x \\
& P_{\frakT}(x,\xi):\Lambda^{k,m}T^*_x\M\to \Lambda^{k,m}T^*_x\dM  \equiv \bbF_2|_x \\ 
& P_{\PntD}(x,\xi):\Lambda^{k,m}T^*_x\M\to \Lambda^{k-1,m}T^*_x\dM  \equiv \bbF_3|_x \\
& P_{\frakF^*}(x,\xi):\Lambda^{k,m}T^*_x\M\to \Lambda^{k-1,m}T^*_x\dM  \equiv \bbF_4|_x \\
& P_{\PtnD}(x,\xi):\Lambda^{k,m}T^*_x\M\to \Lambda^{k,m-1}T^*_x\dM  \equiv \bbF_5|_x \\
& P_{\frakF}(x,\xi):\Lambda^{k,m}T^*_x\M\to \Lambda^{k,m-1}T^*_x\dM  \equiv \bbF_6|_x \\
& P_{\PttD H^*_\g}(x,\xi):\Lambda^{k,m}T^*_x\M\to \Lambda^{k-1,m-1}T^*_x\dM  \equiv \bbF_7|_x \\
& P_{\frakT H^*_\g}(x,\xi):\Lambda^{k,m}T^*_x\M\to \Lambda^{k-1,m-1}T^*_x\dM  \equiv \bbF_8|_x. 
\end{split}
\]
Using Pascal's formula, 
\[
\begin{aligned}
\dim_\R(\bbF_1 \oplus \bbF_3) &= \dim_\R(\bbF_2 \oplus \bbF_4) = \binom{d}{k} \binom{d-1}{m} \\
\dim_\R(\bbF_5 \oplus \bbF_7) &= \dim_\R(\bbF_6 \oplus \bbF_8) = \binom{d}{k} \binom{d-1}{m-1}, 
\end{aligned}
\]
so that
\[
\dim_\R \oplus_{j=1}^8 \bbF_j = \dim_\bbC \oplus_{j=1}^8 (\bbC\otimes \bbF_j) = 2 \binom{d}{k} \binom{d}{m}.
\]

We proceed to calculate the set $\calM^+_{x,\xi}$ for $x\in\dM$ and $\xi\in T_x^*\dM$.
Noting that
\[
P_{\Bg}(x,\xi + \imath\,\partial_s \, dr) = (|\xi|^2 - \partial^2_s)^2 = \partial^4_s - 2|\xi|^2 \partial^2_s + |\xi|^4,
\]
we obtain that
\[
\calM^+_{x,\xi}=\BRK{\omega_0\exp{(-|\xi|s)}+\lambda_0s\exp{(-|\xi|s)} ~:~ \omega_0,\lambda_0\in\mathbb{C}\otimes\Lambda^{k,m}T^*_x\M}.
\]
In particular,
\[
\dim_{\bbC} \calM^+_{x,\xi} = 2 \binom{d}{k} \binom{d}{m}.
\]
Thus, the map $\Xi_{x,\xi}:\calM^+_{x,\xi} \to \mathbb{C}\otimes(\oplus_{j=1}^8 \bbF_j|_x)$ is a linear map between complex vector spaces of the same dimension. It is therefore bijective if and only if it is injective, namely, if and only if it has a trivial kernel. The proof of \thmref{thm:elliptic_regularity} will be completed by showing that for $\sigma\in \calM^+_{x,\xi}$, 
\[
\sigma(s) = \omega_0\exp{(-|\xi|s)}+\lambda_0s\exp{(-|\xi|s)},
\]
it holds that
\[
\Xi_{x,\xi} \sigma = 0
\qquad\text{if and only if}\qquad 
\omega_0 = \lambda_0 = 0. 
\]
One direction is obvious, leaving the ``only if" part to be proved. 

To this end we need to calculate the symbols of the operators 
$\PttD$, $\PtnD$, $\PntD$, $\frakT$, $\frakF^*$, $\frakF$, 
$\PttD \bHg^*$ and $\frakT \bHg^*$.
We start with the symbols of the zeroth-order operators,
\[
\begin{split}
P_{\PttD}(x, \xi + \xi_d\, dr) \sigma &=  \PttD\sigma \\
P_{\PtnD}(x, \xi + \xi_d\, dr) \sigma &=  \PtnD\sigma \\
P_{\PntD}(x, \xi + \xi_d\, dr) \sigma &=  \PntD\sigma \\
P_{\PnnD}(x, \xi + \xi_d\, dr) \sigma &=  \PnnD\sigma,
\end{split}
\]
along with the symbols of the three relevant first-order operators,
\beq
\begin{split}
-\imath\, P_{\frakT}(x,\ \xi + \xi_d\, dr) \sigma &=  \PttD\xi_d\sigma -  \xi\wedge \PntD\sigma  -  \xi^T\wedge \PtnD\sigma \\
-\imath\, P_{\frakF^*}(x,\ \xi + \xi_d\, dr) \sigma &=   \PntD\xi_d\sigma -  \xi^T\wedge \PnnD\sigma  +  \ixi \PttD\sigma \\
-\imath\, P_{\frakF}(x,\ \xi + \xi_d\, dr) \sigma &=   \PtnD\xi_d\sigma -  \xi \wedge \PnnD\sigma   +  \ixiV \PttD\sigma.
\end{split}
\label{eq:first_order_boundary_symbols}
\eeq
The symbols of the only relevant second-order operator is
\[
\begin{split}
P_{\PttD \bHg^*}(x,\ \xi + \xi_d\, dr) \sigma &=  -\PnnD\xi_d^2\sigma -  \ixi \PtnD\xi_d\sigma -
\ixiV \PntD\xi_d\sigma - \ixi \ixiV \PttD\sigma.
\end{split}
\]

For the symbol of the third-order operator $\frakT \bHg^*$, we combine \eqref{eq:first_order_boundary_symbols}, \eqref{eq:symbols_of_H_H^*_F_F^*} and \eqref{eq:principel_homomo} to obtain
\[
\begin{split}
-\imath\, P_{\frakT \bHg^*}(x,\xi+\xi_d\,dr)\sigma&=\xi\wedge\PntD(\ixi\ixiV\sigma+i_{\dr}\ixiV\xi_d\sigma+\ixi i_{\dr}^V\xi_d\sigma+i_{\dr}i^V_{\dr}\xi_d^2\sigma) \\
&+\xi^T\wedge\PtnD(\ixi\ixiV\sigma+i_{\dr}\ixiV\xi_d\sigma+\ixi i_{\dr}^V\xi_d\sigma+i_{\dr}i^V_{\dr}\xi_d^2\sigma) \\
&-\xi_d\, \PttD(\ixi\ixiV\sigma+i_{\dr}\ixiV\xi_d\sigma+\ixi i_{\dr}^V\xi_d\sigma
+i_{\dr}i^V_{\dr}\xi_d^2\sigma).
\end{split}
\]
Noting that
\[
\begin{aligned}
&\PntD \ixi\ixiV\sigma=-\ixi\ixiV\PntD\sigma
&\qquad &\PntD i_{\dr}\ixiV\xi_d\sigma=0
\\&\PntD \ixi i_{\dr}^V\xi_d\sigma=-\ixi\PnnD\xi_d\sigma 
&\qquad &\PntD i_{\dr}i_{\dr}^V\xi_d^2\sigma=0
\\
\\&\PtnD \ixi\ixiV\sigma=-\ixi\ixiV\PtnD\sigma
&\qquad &\PtnD i_{\dr}\ixiV\xi_d\sigma=-\ixiV\PnnD\xi_d\sigma
\\&\PtnD \ixi i_{\dr}^V\xi_d\sigma=0
&\qquad &\PtnD i_{\dr}i_{\dr}^V \xi_d^2\sigma=0
\\ 
\\&\PttD \ixi\ixiV\sigma=\ixi\ixiV\PttD\sigma
&\qquad &\PttD i_{\dr}\ixiV\xi_d\sigma=\ixiV\PntD\xi_d\sigma
\\&\PttD \ixi i_{\dr}^V\xi_d\sigma=\ixi\PtnD\xi_d\sigma
&\qquad &\PttD i_{\dr}i_{\dr}^V\xi_d^2\sigma=\PnnD\xi_d^2\sigma,
\end{aligned}
\] 
we obtain
\[
\begin{split}
-\imath\, P_{\frakT \bHg^*}(x,\xi+\xi_d\,dr)\sigma&=-\xi\wedge \ixi\ixiV\PntD\sigma
-\xi\wedge \ixi\PnnD\xi_d\sigma 
\\&-\xi^T\wedge \ixi\ixiV\PtnD\sigma-\xi^T\wedge \ixiV\PnnD\xi_d\sigma
\\&-\ixi\ixiV\PttD\xi_d\sigma
-\ixiV\PntD\xi_d^2\sigma
-\ixi\PtnD\xi_d^2\sigma
-\PnnD\xi_d^3\sigma.
\end{split}
\]

Let $\sigma\in \calM^+_{x,\xi}$ be given by
\[
\sigma(s) = \omega_0\exp{(-|\xi|s)}+\lambda_0s\exp{(-|\xi|s)}.
\]
We further note that we may restrict the analysis to $\xi\ne0$ having unit norm. Then,
\[
\begin{split}
\sigma(0) &= \omega_0 \\
\dot{\sigma}(0) &= -\omega_0+\lambda_0 \\
\ddot{\sigma}(0) &= \omega_0-2\lambda_0 \\
\dddot{\sigma}(0) &= -\omega_0+3\lambda_0.
\end{split}
\]
To calculate $\Xi_{x,\xi}\sigma$, we write the symbols of $\PttD$, $\PtnD$, $\PntD$, $\frakT$, $\frakF^*$, $\frakF$,  $\PttD \bHg^*$ and $\frakT \bHg^*$, and substitute,
\[
\begin{split}
\sigma &\mapsto \omega_0 \\
\xi_d \sigma &\mapsto -\imath\,\omega_0+\imath\,\lambda_0 \\
\xi_d^2 \sigma &\mapsto -\omega_0+2\,\lambda_0\\
\xi_d^3 \sigma &\mapsto \imath\,\omega_0-3\imath\,\lambda_0.
\end{split}
\]
To simply notations, set
\[
\begin{gathered}
\onn=\PnnD\omega_0, \qquad \ott=\PttD\omega_0, \qquad \ont = \imath\,\PntD\omega_0
\qquad
\otn = \imath\,\PtnD\omega_0 \\
\lnn=\PnnD\lambda_0, \qquad \ltt=\PttD\lambda_0, \qquad \lnt = \imath\,\PntD\lambda_0
\qquad
\ltn = \imath\,\PtnD\lambda_0.
\end{gathered}
\]
The zeroth-order equations yield 
\[
\ott = 0 \qquad \otn=0 \qquad \ont=0.
\]
The first-order equations (taking into account the zeroth-order equations) yield
\[
\ltt=0 \qquad \ltn=\xi\wedge\onn \qquad \lnt=\xi^T\wedge\onn,
\]
The second-order equation yields 
\[
\onn-2\lnn-\ixi\ltn-\ixiV\lnt=0.
\]
Finally, the third-order equation yields
\[
(\xi\wedge\ixi + \xi^T\wedge\ixiV)(\lnn - \onn) - 2\ixiV\lnt - 2\ixi\ltn + \onn - 3\lnn=0.
\]

The real and complex components decouple, hence we may restrict the analysis to sections of real vector bundles. By eliminating $\ltn,\lnt$ we get,
\[
\begin{gathered}
\onn-2\lnn- (\ixi\circ \xi\wedge + \ixiV\circ \xi^T\wedge)\onn=0
\\
(\xi\wedge\ixi+ \xi^T\wedge\ixiV)(\lnn - \onn)  -2(\ixiV\circ \xi^T\wedge+ \ixi\circ \xi\wedge)\onn + \onn - 3\lnn=0.
\end{gathered}
\]
The proof is complete if we prove that $\onn=\lnn=0$.
Noting that $\xi\wedge\ixi = \id - \ixi\circ \xi\wedge$ and $\xi^T\wedge\ixiV = \id - \ixiV\circ \xi^T\wedge$, the second equation takes the form
\[
(\xi\wedge\ixi+ \xi^T\wedge\ixiV)\lnn  - (\ixiV\circ \xi^T\wedge+ \ixi\circ \xi\wedge)\onn - \onn - 3\lnn=0,
\]
which combining with the first equation reduces the system into
\[
\begin{aligned}
2\lnn + (\id - \xi\wedge \ixi  - \xi^T\wedge \ixiV)\onn &= 0 \\
2\onn + (\id - \xi\wedge\ixi - \xi^T\wedge\ixiV)\lnn &= 0.
\end{aligned}
\]
Eliminating $\lnn$ we obtain
\[
4\onn - (\id - \xi\wedge\ixi - \xi^T\wedge\ixiV)(\id - \xi\wedge \ixi  - \xi^T\wedge \ixiV)\onn = 0,
\]
and after further simplification 
\[
\brk{4\,\id  + (\ixi \circ \xi\wedge )(\ixiV \circ \xi^T\wedge)  + (\xi\wedge\ixi)(\xi^T\wedge \ixiV)} \onn = 0.
\] 
Since $(\ixi \circ \xi\wedge)$, $(\ixiV \circ \xi^T\wedge)$, $(\xi\wedge\ixi)$ and $(\xi^T\wedge\ixiV)$ are projections, i.e., have operator norm one, it follows that $\onn=0$ and $\lnn=0$.
\end{proof}

\subsection{Bilaplacian analysis}
\label{subsec:biharmonic_analysis}
In this section we use the elliptic regularity of the double bilaplacian.
We work with the Hilbert space 
\[
\scrH = L^2\VectorFormsKM.
\]
We define its (dense) subspaces,
\[
\begin{split}
H^2_\TT\VectorFormsKM &=H^2\VectorFormsKM \cap\ker(\PttD,\frakT,\PntD,\frakF^*,\PtnD,\frakF) \\
H^4_\TT\VectorFormsKM &= H^4\VectorFormsKM\cap \ker(\PttD,\frakT,\PntD,\frakF^*,\PtnD,\frakF,\PttD\bHg^*,\frakT\bHg^*) \\
H^2_\NN\VectorFormsKM &=H^2\VectorFormsKM\cap\ker(\PnnD,\frakT^*,\PntD,\frakF^*,\PtnD,\frakF) \\
H^4_\NN\VectorFormsKM &= H^4\VectorFormsKM\cap \ker(\PnnD,\frakT^*,\PntD,\frakF^*,\PtnD,\frakF,\PnnD\bHg,\frakT^*\bHg) \\
H^2_\NT\VectorFormsKM &=H^2\VectorFormsKM\cap\ker(\PntD,\frakF^*,\PttD,\frakT,\PnnD,\frakT^*) \\
H^4_\NT\VectorFormsKM &= H^4\VectorFormsKM\cap \ker(\PntD,\frakF^*,\PttD,\frakT,\PnnD,\frakT^*,\PntD\bFg,\frakF^*\bFg) \\
H^2_\TN\VectorFormsKM &=H^2\VectorFormsKM\cap\ker(\PtnD,\frakF,\PttD,\frakT,\PnnD,\frakT^*) \\
H^4_\TN\VectorFormsKM &= H^4\VectorFormsKM\cap \ker(\PtnD,\frakF,\PttD,\frakT,\PnnD,\frakT^*,\PtnD\bFg^*,\frakF\bFg^*),
\end{split}
\]
and the vector bundle,
\[
\E=\Lambda^{k+1,m+1}T^*\M\oplus\Lambda^{k-1,m-1}T^*\M\oplus\Lambda^{k-1,m+1}T^*\M\oplus\Lambda^{k+1,m-1}T^*\M.
\]
For every $\RR\in\BRK{\TT,\TN,\NT,\NN}$, we define the unbounded operator $D_\RR:\scrH\to L^2\Gamma(\E)$,
\[
\begin{gathered}
\scrD(D_\RR) = H^2_\RR\VectorFormsKM 
\qquad
D_\RR\psi = (\bHg\psi, \bHg^*\psi, \bFg^*\psi, \bFg\psi),
\end{gathered}
\]
where $\scrD(\cdot)$ denotes the domain of an unbounded operator; 
we denote its closure (which exists because its dual is defined on a dense subset) by $\bar{D}_\RR$, and define the operator $B_\RR:\scrH\to\scrH$,
\[
\begin{gathered}
\scrD(B_\RR) = H^4_\RR\VectorFormsKM
\qquad 
B_\RR\psi = \Bg \psi.
\end{gathered}
\]

\begin{proposition}
\label{prop:boundary_data_vanish}
Let $\RR\in\BRK{\TT,\TN,\NT,\NN}$ and let $\psi\in H^2_\RR\VectorFormsKM$ be such that
\[
\Bg\psi\in \scrH,
\]
where the left-hand side is defined in the distributive sense. Then, 
\[
\bra \Bg\psi, \eta\ket=\bra D_\RR\psi, D_\RR\eta\ket 
\qquad\text{for all $\eta\in H^2_\RR\VectorFormsKM$}
\]
if and only if $\psi\in H^4_\RR\VectorFormsKM$, in which  case $\Bg\psi = B_\RR\psi$.
\end{proposition}

\begin{proof}
The proof is  based on the regular ellipticity of the boundary data $\RR$; by duality,  we may consider only the case of  $\RR=\TT$. Let $\psi\in H^2_\TT\VectorFormsKM$ satisfy $\Bg\psi\in \scrH$. 
Since $\Bg$ is an elliptic operator, if follows from \cite[p.~459]{Tay11a} that $\psi$ has well-defined boundary data, namely $\PttD \bHg^*\psi,\frakT \bHg^*\psi$ are defined.
By definition,
\[
\bra D_\TT\psi, D_\TT\eta\ket = \bra \bHg\psi, \bHg\eta\ket +\bra \bHg^*\psi, \bHg^*\eta\ket +\bra \bFg^*\psi, \bFg^*\eta\ket+\bra \bFg\psi, \bFg\eta\ket.
\]
Using \eqref{eq:H_general_integral}  along with $\eta\in H^2_\TT\VectorFormsKM$,  
\[
\bra \Bg\psi, \eta\ket - \bra D_\TT\psi, D_\TT\eta\ket =
\int_{\dM}\Brk{(\PttD \bHg^*\psi,\frakT^* \eta)_\gD - (\frakT \bHg^*\psi,\PnnD \eta)_\gD}\VolumeD.
\]
Clearly, if $\psi\in H^4_\TT\VectorFormsKM$, then the right-hand side vanishes for all $\eta\in H^2_\TT\VectorFormsKM$.
As for the other direction, suppose that the left-hand side vanishes for all $\eta\in H^2_\TT\VectorFormsKM$.
From \lemref{lemma:prescribe_boundary_conditions}, we can produce $\eta\in H^2_\TT\VectorFormsKM$ such that
$\PnnD\eta$ and $\frakT^*\eta$ are prescribed independently, hence $\PttD \bHg^*\psi=0$ and $\frakT \bHg^*\psi=0$, and by the elliptic estimate \eqref{eq:elliptic_estimate}, $\psi\in H^4_\TT\VectorFormsKM$.
\end{proof}

From here on we follow \cite[pp.~462--463]{Tay11b}, using the elliptic regularity established in \thmref{thm:elliptic_regularity}. 

\begin{proposition}
\label{prop:self_adjoint}
$B_\RR$ is self-adjoint for all $\RR\in\BRK{\TT,\TN,\NT,\NN}$.
\end{proposition}

\begin{proof}
Every closed unbounded linear operator $\scrH\to\scrH$ defined on a dense domain has a Hilbert space adjoint. 
Let $\bar{D}_\RR^*$ be the adjoint of $\bar{D}_\RR$; 
By von Neumann's theorem \cite[Thm.~13.13]{Rud91},  $\bar{D}_\RR^*\bar{D}_\RR:\scrH\to\scrH$ is an unbounded, self-adjoint operator, having dense domain. It remains to prove that $\bar{D}_\RR^*\bar{D}_\RR = B_\RR$. The domain of $\bar{D}_\RR$ is
\[
\scrD(\bar{D}_\RR)=\BRK{\psi\in\scrH~:~ D_{\RR}\psi\in L^2\Gamma(\E), \quad \psi\in\ker(\PttD,\PtnD,\PntD,\frakT,\frakF^*,\frakF)},
\]
where the operators in the defining expression are to be understood distributively. 
Let $\psi\in\scrD(\bar{D}_\RR^*\bar{D}_\RR)$.  
By definition of the Hilbert space dual, for all $\eta \in H^2_\RR\VectorFormsKM  = \scrD(D_\RR) \subset \scrD(\bar{D}_\RR)$, 
\[
\bra \bar{D}_\RR^*\bar{D}_\RR\psi, \eta\ket = \bra \bar{D}_\RR\psi, \bar{D}_\RR\eta\ket  = \bra D_\RR\psi, D_\RR\eta\ket,
\] 
where in the last passage we used the fact that $\bar{D}_\RR$ and $D_\RR$ coincide on $\scrD(\bar{D}_\RR)$. 
In particular, for all $\eta$ smooth and compactly-supported, and since $D_{\RR}^*$ is always closed with $\bar{D}_\RR^*=D_{\RR}^*$, by integration by parts we have $D_\RR^*D_\RR\eta=\Bg\eta$ and
\[
\bra \bar{D}_\RR^*\bar{D}_\RR\psi, \eta\ket =\bra \psi, D_\RR ^*D_\RR\eta\ket=\bra \psi,\Bg\eta\ket= \bra \Bg \psi, \eta\ket,
\]

where $\Bg \psi$ has now to be interpreted in a distributional sense. Since smooth, compactly-supported sections are dense in $\scrH$, we conclude that $\Bg \psi = \bar{D}_\RR^*\bar{D}_\RR\psi \in \scrH$, and by \propref{prop:boundary_data_vanish}, $\psi \in H^4_\RR\VectorFormsKM$. This proves that $\scrD(\bar{D}_\RR^*\bar{D}_\RR) \subset H^4_\RR\VectorFormsKM = \scrD(B_\RR)$, and that $\bar{D}_\RR^*\bar{D}_\RR$ coincides with $B_\RR$ on its domain.

Thus, $B_\RR$ is a symmetric extension of $\bar{D}_\RR^*\bar{D}_\RR$, however a self-adjoint operator does not have proper symmetric extensions \cite[Thm.~13.15]{Rud91}, hence $B_\RR = \bar{D}_\RR^*\bar{D}_\RR$.
\end{proof}

Since $B_\RR$  is a Fredholm operator, it follows that $\ker B_\RR$ is a finite-dimensional subspace of $\VectorFormsKM$; in particular, its elements are smooth. 
This observation motivates the following analog of the harmonic modules in the Hodge decomposition:

\begin{definition}
For every $\RR\in\BRK{\TT,\TN,\NT,\NN}$, the biharmonic $\RR$-module is the finite-dimensional subspace of $\VectorFormsKM$,
\[
\BH^{k,m}_\RR(\M)=\ker B_\RR  .
\]
\end{definition}

\begin{proposition}
\label{prop:5.11}
The biharmonic $\RR$-modules can be characterized as follows,
\[
\BH^{k,m}_\RR(\M) = \BHkm \cap H^2_\RR\VectorFormsKM,
\]
where 
\beq
\BHkm = \VectorFormsKM \cap\ker( \bHg , \bHg^* , \bFg^* , \bFg).
\label{eq:def_BHkm}
\eeq
\end{proposition}

\begin{proof}
Let $\psi\in \BH^{k,m}_\RR(\M)$, namely $B_\RR\psi=0$, and in particular $\psi\in\VectorFormsKM\cap H^4_\RR\VectorFormsKM$.
Since $H^4_\RR\VectorFormsKM\subset H^2_\RR\VectorFormsKM$, and since $B_\RR =\bar{D}_\RR^*\bar{D}_\RR = D_\RR^* D_\RR$ on $H^2_\RR\VectorFormsKM$, it follows that 
\[
0=\bra \Bg\psi ,\psi\ket = \bra D_\RR\psi,D_\RR\psi\ket =\Norm{\bHg\psi}_{L^2}^2+\Norm{\bHg^*\psi}_{L^2}^2+\Norm{\bFg^*\psi}_{L^2}^2+\Norm{\bFg\psi}_{L^2}^2,
\]
proving that $\psi \in \BHkm \cap H^2_\RR\VectorFormsKM$, namely,
\[
\BH^{k,m}_\RR(\M) \subset  \BHkm \cap H^2_\RR\VectorFormsKM.
\]
In the other direction, let $\psi\in \BHkm \cap H^2_\RR\VectorFormsKM$. 
By the definition of $\Bg$ it follows that $\Bg\psi=0$. Moreover, since $\psi\in\ker(\bHg,\bHg^*,\bFg^*,\bFg)$, then $\psi\in\ker(\PttD\bHg^*,\frakT\bHg^*)$, i.e., $\psi\in H^4_\RR\VectorFormsKM$, implying that $B_\RR\psi=0$, namely, 
\[
\BHkm \cap H^2_\RR\VectorFormsKM \subset \BH^{k,m}_\RR(\M).
\]

\end{proof}

Furthermore, we obtain the following estimate on the $H^2$-norm:

\begin{proposition}
Let $\RR\in\BRK{\TT,\TN,\NT,\NN}$. Then for every $\psi\in H^2_\RR\VectorFormsKM$, 
\beq
\Norm{\psi}_{H^2}^2\lesssim \Norm{D_\RR\psi}^2_{L^2}+\Norm{\psi}^2_{L^2}.
\label{eq:generlized_korn_inequality}
\eeq
\end{proposition}

\begin{proof}
By the definition of a closed operator, $\Graph(\bar{D}_\RR)$ is a closed subspace of $\scrH\oplus L^2\Gamma(\E)$, which is a Hilbert space with inner-product
\[
((u,\bar{D}_\RR u),(v,\bar{D}_\RR v)) = (u,v)_{L^2} + (\bar{D}_\RR u,\bar{D}_\RR v)_{L^2}.
\]
The inequality follows from the fact that $\Graph(\bar{D}_\RR)$ is continuously embedded in $H^2\VectorFormsKM$, namely,
\[
\Norm{\psi}_{H^2}^2\lesssim \|(\psi,\bar{D}_\RR\psi)\|_{\Graph(\bar{D}_\RR)}^2 = \Norm{D_\RR\psi}^2_{L^2}+\Norm{\psi}^2_{L^2}.
\]

Let $J: \Graph(\bar{D}_\RR)\to\scrH$ be defined by the projection
\[
J(u,\bar{D}_\RR u) = u.
\]
Its dual $J^*:\scrH\to \Graph(\bar{D}_\RR)$ is defined by
\[
(J^*u,(v,\bar{D}_\RR v))_{\Graph(\bar{D}_\RR)} = \bra u,J(v,\bar{D}_\RR v)\ket = \bra u,v\ket
\]
for all $v\in\scrH$. Let $J^*u = (w,\bar{D}_\RR w)$, then
\[
\bra w,v\ket + \bra\bar{D}_\RR w,\bar{D}_\RR v\ket = \bra u,v\ket
\]
for all $v\in\scrH$. Clearly, $w\in H^4_\RR\VectorFormsKM$ which is a solution to the elliptic problem
\[
(\id + B_\RR)w = u
\]
satisfies this condition. Thus, the range of the self-adjoint operator $JJ^*: \scrH\to\scrH$ defined by
\[
JJ^* = (\id + B_\RR)^{-1},
\]
is $H^4_\RR\VectorFormsKM$, and $(JJ^*)^{-1}: H^4_\RR\VectorFormsKM\to \scrH$ 
is given by $\id + B_\RR$. 

From \cite[p.~99]{Tay11b},
\[
\scrD((\id + B_\RR)^{1/2}) = J\Graph(\bar{D}_\RR) = \scrD(\bar{D}_\RR).
\]
From \cite[p.~323]{Tay11a},
\[
\scrD((\id + B_\RR)^{1/2}) = [\scrH,H^4_\RR\VectorFormsKM]_{1/2},
\] 
where the right-hand side is an interpolation space \cite[p.~322]{Tay11a}.
Since $H^4_\RR\VectorFormsKM\hookrightarrow H^4\VectorFormsKM$ continuously,  from \cite[p.~323]{Tay11a} we obtain that the following inclusion is also continuous,
\[
\scrD((\id + B_\RR)^{1/2}) \hookrightarrow [\scrH,H^4\VectorFormsKM]_{1/2}=H^2\VectorFormsKM,
\]
where the latter equality is due to \cite[p.~324]{Tay11a}. All in all, we find that $\scrD(\bar{D}_\RR)\hookrightarrow  H^2\VectorFormsKM$ continuously (by definition, $\scrD(\bar{D}_\RR)$ is isometric to $\Graph(\bar{D}_\RR)$), which provides   \eqref{eq:generlized_korn_inequality}.
\end{proof}

By using the extension \lemref{lemma:prescribe_boundary_conditions} , we can generalize \eqref{eq:generlized_korn_inequality} for general boundary conditions: if we extend the definition of $D_\RR$ into an operator,
\[
D_\RR^\partial: H^2\VectorFormsKM\to L^2\Gamma(\E)\oplus (\oplus_{j=1}^3H^{3/2}\VectorFormsdM)\oplus(\oplus_{j=1}^3 H^{1/2}\VectorFormsdM),
\]
given by
\[
\begin{aligned}
D_\TT^\partial &=(\bHg,\bHg^*,\bFg^*,\bFg,\PttD,\PtnD,\PntD,\frakT,\frakF^*,\frakF)
\\ D_\NT^\partial &=(\bHg,\bHg^*,\bFg^*,\bFg,\PntD,\PnnD,\PttD,\frakF^*,\frakT^*,\frakT)
\\
D_\TN^\partial &= (\bHg,\bHg^*,\bFg^*,\bFg,\PtnD,\PnnD,\PttD,\frakF,\frakT^*,\frakT)
\\
D_\NN^\partial &=(\bHg,\bHg^*,\bFg^*,\bFg,\PnnD,\PtnD,\PntD,\frakT^*,\frakF^*,\frakF),
\end{aligned}
\]
then by setting the induced norm on 
\[
L^2\Gamma(\E)\oplus(\oplus_{j=1}^3H^{3/2}\VectorFormsdM)\oplus(\oplus_{j=1}^3 H^{1/2}\VectorFormsdM)
\]
to be
\[
\Norm{\cdot}^2_\RR=\Norm{\cdot}^2_{L^2(\M)}+\sum_{j=1}^3\Norm{\cdot}^2_{H^{3/2}(\dM)}+\sum_{l=1}^3\Norm{\cdot}^2_{H^{1/2}(\dM)},
\]
we can rewrite \eqref{eq:generlized_korn_inequality} as an estimate valid for all $\psi\in H^2\VectorFormsKM$, 
\beq
\Norm{\psi}_{H^2}^2\lesssim \Norm{D_\RR^\partial\psi}^2_\RR+{\Norm{\psi}^2_{L^2}}.
\label{eq:general_estimate}
\eeq

We further note that since the biharmonic modules are finite-dimensional, the projections 
\[
P_\RR:\scrH\to \BH^{k,m}_\RR(\M)
\]
are well-defined. The right-hand side of \eqref{eq:general_estimate} can be sharpened even further:

\begin{theorem}
For every $\RR\in\BRK{\TT,\TN,\NT,\NN}$ the linear map $D_\RR^\partial$
has a closed range, and
\beq
\Norm{\psi}_{H^2}^2\lesssim \Norm{D_\RR^\partial \psi}^2_\RR+{\Norm{P_\RR\psi}^2_{L^2}}.
\label{eq:general_estimate_final}
\eeq
\end{theorem}

\begin{proof}
By \propref{prop:5.11},  
$\BH^{k,m}_\RR(\M) =\BHkm \cap H^2_\RR(\M) =\ker D_\RR^\partial$; since it is a finite-dimensional subspace of $\scrH$, there exists an orthogonal decomposition
\[
\scrH = \BH^{k,m}_\RR(\M)\oplus(\BH^{k,m}_\RR(\M))^\bot.
\]
Consider the space
\[
\bbH=(\BH^{k,m}_\RR(\M))^\bot\cap H^2\VectorFormsKM,
\]
which is a Hilbert subspace of $H^2\VectorFormsKM$, and the quadratic form on $\bbH$
\[
\psi\mapsto \Norm{D_\RR^\partial\psi}^2_\RR.
\] 
If we prove that
\beq
\Norm{\psi}^2_{H^2}\lesssim \Norm{D_\RR^\partial\psi}^2_\RR 
\label{eq:needs_to_be_proved}
\eeq
in $\bbH$, the proof will be complete: 
specifically, given $\psi\in H^2\VectorFormsKM$, we write
\[
\|\psi\|_{H^2}  \le \|P_\RR\psi\|_{H^2} + \|(\id- P_\RR)\psi\|_{H^2}.
\]
For the first term on the right-hand side, \eqref{eq:general_estimate} yields,
\[
\|P_\RR\psi\|_{H^2}^2 \lesssim \|D_\RR P_\RR\psi\|_{L^2}^2 + \|P_\RR\psi\|_{L^2}^2 =  \|P_\RR\psi\|_{L^2}^2,
\]
where we used the fact that $P_\RR\psi\in\ker D_\RR$.
As for the second term, since $(\id - P_\RR)\psi \in \bbH$, it follows by \eqref{eq:needs_to_be_proved} that
\[
\|(\id- P_\RR)\psi\|_{H^2}^2 \lesssim \Norm{D_\RR^\partial(\id-P_\RR) \psi}^2_\RR = \Norm{D_\RR^\partial \psi}^2_\RR,
\]
where we used the fact that $P_\RR\psi\in\ker D^\partial_\RR$.
As $P_\RR$ is a compact linear map, it follows from \cite[p.~583]{Tay11a} that $D_\RR^\partial$ has closed range.

It remains to prove \eqref{eq:needs_to_be_proved}.
Consider the $L^2$-unit sphere in $\bbH$, i.e.,
\[
S_1=\BRK{\psi\in\bbH ~:~ \Norm{\psi}_{L^2}=1}
\]
Let  $\eta_n\in\bbH$ be a minimizing sequence
\[
\Norm{D_\RR^\partial \eta_n}^2_\RR \longrightarrow \inf_{\omega\in S_1} \Norm{D_\RR^\partial \omega}^2_\RR.
\]
By \eqref{eq:general_estimate}, the sequence $\eta_n$ is $H^2$-bounded. Since $\bbH$ is a Hilbert space, we can apply the Banach-Alaoglu theorem and move to a (not relabeled) weakly convergent subsequence $\eta_n \weak \eta_0$. By Rellich's embedding theorem, $\eta_n$ strongly converges in $L^2$, implying that $\eta_0\in S^1$. 
Since the quadratic form is lower-semicontinuous,
\[
\Norm{D_\RR^\partial \eta_0}^2_\RR = \inf_{\omega\in S_1} \Norm{D_\RR^\partial \omega}^2_\RR.
\]
Since $\Norm{\eta_0}_{L^2}=1$ and $\eta_0\notin\BH^{k,m}_\RR(\M)=\ker D_\RR^\partial$ by construction, we must have $\Norm{D_\RR^\partial \eta_0}^2_\RR > 0$. 
This implies that for all $\psi\in\bbH$,
\[
\Norm{D_\RR^\partial \psi}^2_\RR \geq\Norm{D_\RR^\partial \eta_0}^2_\RR \, \Norm{\psi}^2_{L^2}.
\]
By a final reference to \eqref{eq:general_estimate}, we obtain for all $\psi\in\bbH$,
\[
\Norm{\psi}_{H^2}^2\lesssim \brk{1+\frac{1}{\Norm{D_\RR^\partial \eta_0}^2_\RR}}\Norm{D_\RR^\partial \psi}^2_\RR,
\]
which is the required result. 
\end{proof}

We return to the operators $B_\RR$ and establish their spectral properties:

\begin{proposition}
\label{prop:spectral_1}
For every $\RR\in\BRK{\TT,\TN,\NT,\NN}$ there exists a smooth $L^2$-orthonormal basis $\BRK{\psi_j}_{j=1}^{\infty} \subset H^4_\RR\VectorFormsKM$ for $\scrH$, and sequences of non-negative numbers $(\rho_j)_{j=1}^{\infty}$ with $\rho_j\to \infty$ as $j\to\infty$, such that,
\[
B_\RR\psi_j = \rho_j\psi_j.
\]
\end{proposition}

\begin{proof}
By von Neumann's theorem \cite[Thm.~13.13]{Rud91}, 
there exists a bounded, positive-definite, self-adjoint operator,
\[
K_\RR: \scrH\to \scrH,
\]
such that
\[
(\id + B_\RR) K_\RR = \id.
\]
Moreover, $\|K_\RR\| \le 1$. Note how the image of $K_\RR$ is contained $\scrD(B_{\RR})\cap\scrD(\id)= H^4_\RR\VectorFormsKM$, thus $K_{\RR}:\scrH\rightarrow H^4_\RR\VectorFormsKM$ is an unbounded operator whose domain is the whole of $\scrH$. By the elliptic estimate \eqref{eq:elliptic_estimate} (noting that $K_R\psi$ is in the kernel of the boundary operators), for $s=0, p=2,q=2$, 
\[
\begin{split}
\Norm{K_{\RR}\psi}_{H^{4}}&\lesssim \Norm{\Bg K_{\RR}\psi}_{L^2}+\Norm{K_{\RR}\psi}_{L^2} \\
&\lesssim   \Norm{(\id+B_{\RR}) K_{\RR}\psi}_{L^2}+\Norm{K_{\RR}\psi}_{L^2} \\
&\lesssim\Norm{\psi}_{L^2},
\end{split}
\]
where we used the continuity of $K_{\RR}:\scrH\rightarrow\scrH$. This inequality implies that $K_{\RR}:\scrH\rightarrow H^4_\RR\VectorFormsKM$ is continuous. By Rellich's embedding theorem, $H^4_\RR\VectorFormsKM$ is compactly embedded in $\scrH$, and since the composition of a continuous operator and a compact operator is compact, it follows that $K_\RR:\scrH\rightarrow\scrH$ is a compact, self-adjoint operator. By the spectral theorem for compact operators, there exists an orthonormal basis $(\psi_j)_{j=1}^{\infty}\subset \scrH$ and numbers $(\mu_j)_{j=1}^{\infty}\subset (0,1]$ such that,
\[
\begin{split}
K_\RR\psi_j=\mu_j\psi_j.
\end{split}
\]
We order the eigenvectors such that $\mu_j\to 0$ as $j\to \infty$.
Also, $\psi_j \in\image K_\RR \subset  H^4_\RR\VectorFormsKM$.

Thus,
\[
\begin{split}
(\id + B_\RR)\psi_j &= \mu_j^{-1} (\id + B_\RR) K_\RR \psi_j = \mu_j^{-1} \psi_j,
\end{split}
\]
namely,
\[
B_\RR\psi_j = (\mu_j^{-1}-1) \psi_j.
\]
The eigenvalues $\rho_j = \mu_j^{-1}-1$  satisfy the required properties. 
It remains to prove that the basis functions $\psi_j$ are in fact smooth, which follows by iterating the elliptic estimate  \eqref{eq:elliptic_estimate}. 
\end{proof}

This last proposition provides us with the biharmonic Green operators:

\begin{proposition}
\label{prop:spectral_2}
For every $\RR\in\BRK{\TT,\TN,\NT,\NN}$, $s\in\bbN$ and $p\geq 2$, there exists a continuous linear mapping,
\[
G_\RR: W^{s,p}\VectorFormsKM\to W^{s+4,p}\VectorFormsKM\cap H^4_\RR \VectorFormsKM,
\] 
such that $G_\RR$ annihilates $\BH^{k,m}_\RR(\M)$ and inverts $B_\RR$ on the orthogonal complement of $\BH^{k,m}_\RR(\M)$. I.e., for all $\psi\in W^{s,p}\VectorFormsKM$,
\beq
B_\RR G_\RR\psi=(\id-P_\RR)\psi.
\label{eq:green_operator}
\eeq
\end{proposition}

\begin{proof}
 We define the Green operator $G_\RR:\scrH\to \scrH$ by its actions on the orthonormal basis $(\psi_j)$, 
\[
G_\RR\psi_j=\begin{cases}
0 & \rho_j=0 \\
\rho_j^{-1}\psi_j & \rho_j>0. 
\end{cases}
\]
Since $\rho_{j}^{-1}$ is bounded by \propref{prop:spectral_1}, $G_{\RR}:\scrH\to \scrH$ is continuous and satisfies
\[
B_\RR G_\RR\psi_j=\begin{cases}
0 & \rho_j=0 \\
\psi_j & \rho_j>0, 
\end{cases}
\]
namely, $B_\RR G_\RR\psi_j=(\id-P_\RR)\psi_j$. By continuity,
\[
B_\RR G_\RR\psi=(\id-P_\RR)\psi
\qquad\text{for all $\psi\in\scrH$},
\]
and in particular, $G_\RR\psi\in H^4_\RR\VectorFormsKM$.
Since $G_\RR\psi$ satisfies the boundary-data associated with $\RR$, and since $B_\RR G_\RR\psi$ is at least as regular as $\psi$, we can apply the elliptic estimate \eqref{eq:elliptic_estimate} to obtain that $G_\RR: W^{s,p}\VectorFormsKM\to W^{s+4,p}\VectorFormsKM$ continuously.  More explicitly, for $q=2$ \eqref{eq:elliptic_estimate} reads,
\[
\begin{split}
\Norm{G_\RR\psi}_{W^{s+4,p}}&\lesssim \Norm{\Bg G_\RR\psi}_{W^{s,p}}+\Norm{G_\RR\psi}_{L^2} \\
& \lesssim \Norm{ \psi-P_{\RR}\psi}_{W^{s,p}}+\Norm{G_\RR\psi}_{L^2} \\
& \lesssim \Norm{ \psi}_{W^{s,p}}+\Norm{P_{\RR}\psi}_{W^{s,p}}+\Norm{G_\RR\psi}_{L^2} \\
& \lesssim \Norm{ \psi}_{W^{s,p}}+\Norm{\psi}_{L^2} \\
&\lesssim \Norm{ \psi}_{W^{s,p}},
\end{split}
\]
where in the last step we used the fact that $P_{\RR}: W^{s,p}\Omega^{k,m}(\M)\rightarrow W^{s,p}\Omega^{k,m}(\M)$ continuously (as a projection onto a finite-dimensional space) and $G_{\RR}:\scrH\to \scrH$ is continuous as already established. 
\end{proof}

We next state an existence and uniqueness result for boundary-value problems associated with $\frakB_\RR$. We state it for $\RR = \TT$, as the other statements are analogous. 

\begin{theorem}
\label{thm:bilaplacian_full_problem}
For every $s\in \bbN$ and $p\geq 2$, given data,
\[
\begin{gathered}
\upsilon\in W^{s,p}\VectorFormsKM \\
\phi_i\in  W^{s+7/p,p}\VectorFormsdM 
\qquad 
\rho_i\in  W^{s+5/p,p}\VectorFormsdM \qquad i=1,2,3 \\
\theta\in W^{s+3/p,p}\VectorFormsdM
\qquad 
\tau\in W^{s+1/p,p}\VectorFormsdM
\end{gathered}
\]
(the type is implied by the context), there exists a solution $\psi\in W^{s+4,p}\VectorFormsKM$ to the boundary-value problem
\beq
\frakB_\TT\psi = (\upsilon,\phi_1,\phi_2,\phi_3,\rho_1,\rho_2,\rho_3,\theta,\tau) 
\label{eq:BTT_BVP}
\eeq
if and only if for all $\zeta\in \BH_\TT^{k,m}(\M)$,
\beq
\bra \upsilon,\zeta \ket = \int_{\dM}\Brk{(\theta,\frakT^*\zeta)_{\gD}-(\tau,\PnnD\zeta)_{\gD}}\VolumeD.
\label{eq:integrability_condition_biharmonic_equation}
\eeq
The solution is unique up to an element in $\BH_\TT^{k,m}(\M)$. 
\end{theorem}

\begin{proof}
We first prove the theorem for $s=0$ and $p=2$, namely,
\[
\begin{split}
{\frakB_\TT}:H^4\VectorFormsKM &\to L^2\VectorFormsKM \oplus_{i=1}^3H^{7/2}\VectorFormsdM \\
&\qquad\oplus_{i=1}^3 H^{5/2}\VectorFormsdM \oplus H^{3/2}\VectorFormsdM
\oplus H^{1/2}\VectorFormsdM.
\end{split}
\]

The necessity of the condition \eqref{eq:integrability_condition_biharmonic_equation} is immediate: given a solution $\psi$ to \eqref{eq:BTT_BVP}, we have for all $\zeta\in \BH_\TT^{k,m}(\M)$,
\[
\begin{split}
\bra\upsilon,\zeta\ket&=\bra \Bg\psi,\zeta \ket =\bra \bHg\bHg^*\psi+\bFg^*\bFg\psi+\bFg\bFg^*\psi+\bHg^*\bHg\psi,\zeta \ket
\\&= \int_{\dM}\Brk{(\theta,\frakT^*\zeta)_{\gD}-(\tau,\PnnD\zeta)_{\gD}}\VolumeD,
\end{split}
\]
where the passage to the second line follows from integration by parts and the fact that $\zeta\in\BH_\TT^{k,m}(\M)$.

We next prove that condition \eqref{eq:integrability_condition_biharmonic_equation} is also sufficient.
Let 
\[
V = (\upsilon,\phi_1,\phi_2,\phi_3,\rho_1,\rho_2,\rho_3,\theta,\tau).
\]
By the Fredholm property of $\frakB_\TT$, there exists $\psi\in H^4\VectorFormsKM$ such that
\[
{\frakB_\TT}\psi - V\in\operatorname{coker}{\frakB_\TT}.
\]
That is, for all $\lambda\in H^4\VectorFormsKM$,
\beq
({\frakB_\TT}\psi - V,\frakB_\TT\lambda)_\TT=0,
\label{eq:cokerenl_conditions}
\eeq
where $(\cdot,\cdot)_\TT$ is the inner-product on
\[
L^2\VectorFormsKM \oplus_{i=1}^3H^{7/2}\VectorFormsdM\oplus_{i=1}^3 H^{5/2}\VectorFormsdM) \oplus H^{3/2}\VectorFormsdM
\oplus H^{1/2}\VectorFormsdM.
\]

Setting $\lambda = G_\TT\omega$ for arbitrary  $\omega\in \VectorFormsKM$, all the boundary terms in \eqref{eq:cokerenl_conditions} vanish and we are left with
\[
0=\bra \Bg\psi-\upsilon,\Bg G_\TT\omega\ket = \bra \Bg\psi-\upsilon,(\id-P_\TT)\omega\ket.
\]
Since this holds for all $\omega$, it follows that $\Bg\psi-\upsilon\in \BH_\TT^{k,m}(\M)$. 

We next use \corrref{cor:full_problem_boundary} to construct an $\omega\in H^4\VectorFormsKM$ satisfying the boundary conditions
\[
(\PttD,\PntD,\PtnD,\frakT,\frakF^*,\frakF,\PttD\bHg^*,\frakT\bHg^*)\omega = (\phi_1,\phi_2,\phi_3,\rho_1,\rho_2,\rho_3,\theta,\tau),
\]
and set  $\lambda\in H^4\VectorFormsKM$,
\[
\lambda=\psi-G_\TT\upsilon-\omega.
\]
A direct calculation gives
\[
\begin{split}
\frakB_\TT\lambda &= \frakB_\TT\psi - ((\id - P_\TT)\upsilon, 0,0,0,0,0,0,0,0) - (\Bg\omega,\phi_1,\phi_2,\phi_3,\rho_1,\rho_2,\rho_3,\theta,\tau) \\
&= \frakB_\TT\psi - V + (P_\TT \upsilon - \Bg\omega, 0,0,0,0,0,0,0,0).
\end{split}
\]
Substituting into \eqref{eq:cokerenl_conditions}, we obtain
\[
0 = \Norm{\frakB_\TT\psi - V}^2_\TT + \bra \Bg\psi-\upsilon, P_\TT\upsilon - \Bg\omega \ket.
\]
If we show that the second term on the right-hand side vanishes, then $\frakB_\TT\psi = V$ as required.

Write $\zeta = \Bg\psi-\upsilon\in\BH_\TT$. Since $P_\TT$ is an orthogonal projection and $P_\TT\zeta = \zeta$, we can replace $P_\TT\upsilon$ by $\upsilon$, i.e., it suffices to show that
\[
\bra \zeta,\upsilon \ket  =  \bra \zeta, \Bg\omega \ket.
\]
Integrating by parts the right hand-side and substituting the boundary conditions for $\omega$, we obtain
\[
\bra \zeta, \Bg\omega \ket = 
\int_{\dM}\Brk{(\theta,\frakT^*\zeta)_{\gD}-(\tau,\PnnD\zeta)_{\gD}}\VolumeD
\EwR{eq:integrability_condition_biharmonic_equation} \bra \zeta,\upsilon \ket, 
\]
which completes the first part of the proof.

The uniqueness clause follows directly from the fact that $\frakB_\TT\psi' = V$ if and only if 
\[
\psi-\psi'\in\ker \frakB_\TT=\BH_\TT^{k,m}(\M).
\]
The case of general $s,p$ is recovered by solving for $s=0,p=2$ and observing that by the elliptic estimate \eqref{eq:elliptic_estimate}, the solution inherits the appropriate $W^{s,p}$-regularity. 
\end{proof}

\section{Symmetric forms}
\label{sec:symmetric_forms}

\subsection{The module of symmetric forms}

Since the transposition of double forms is an isomorphism of vector bundles, we can define the subbundle of symmetric $k$-vectors,
\[
\Upsilon^k T^*\M = \BRK{\psi \in \Lambda^{k,k}T^*\M ~:~ \psi^T=\psi},
\]
and the corresponding $C^{\infty}$-module of symmetric $k$-forms,
\[
\Theta^k(\M) = \Gamma(\Upsilon^k T^*\M) \subseteq \Omega^{k,k}(\M).
\]
The wedge product commutes with transposition, hence it restricts to a map $\Theta^k(\M)\times\Theta^{m}(\M)\to\Theta^{k+m}(\M)$. This turns $\Theta^*(\M)=\oplus_{k}\Theta^k(\M)$ into a graded algebra. Moreover, $\starG\starG^{V}:\Omega^{k,k}(\M)\to \Omega^{d-k,d-k}(\M)$ commutes with transposition as well, resulting in an isomorphism,
\[
\starG\starG^{V}:\Theta^k(\M)\to \Theta^{d-k}(\M).
\]
The Sobolev theory of double forms holds verbatim for $W^{s,p}\Theta^k(\M)$. 

Consider now the restrictions of the second-order operators to symmetric forms.
Since
\beq 
\begin{aligned}
& (\bHg\psi^T)^T = \bHg\psi
& \qquad
& (\bHg^*\psi^T)^T = \bHg^*\psi
\\
& (\bFg^*\psi^T)^T = \bFg\psi
& \qquad
& (\bFg\psi^T)^T = \bFg^*\psi \\
&& (\Bg\psi^T)^T = \Bg\psi,
\end{aligned}
\label{eq:symmetries_preserved_by_H_F}
\eeq
it follows that 
\[
\begin{aligned}
& \bHg:\Theta^k(\M)\to\Theta^{k+1}(\M) \\
& \bHg^*:\Theta^k(\M)\to\Theta^{k-1}(\M)  \\
& \Bg:\Theta^k(\M)\to\Theta^k(\M).
\end{aligned}
\]

The boundary operators and integral formulas apply to elements in $\Theta^k(\M)$. In particular, 
\beq
\begin{aligned}
& (\PttD\psi^T)^T  = \PttD\psi
& \qquad
& (\PnnD\psi^T)^T  = \PnnD\psi \\
& (\PntD\psi^T)^T  = \PtnD\psi
& \qquad
& (\PtnD\psi^T)^T  = \PntD\psi \\
& (\frakT\psi^T)^T  = \frakT\psi
& \qquad
& (\frakT^*\psi^T)^T  = \frakT^*\psi \\
& (\frakF^*\psi^T)^T  = \frakF\psi
& \qquad
& (\frakF\psi^T)^T  = \frakF^*\psi,
\end{aligned}
\label{eq:boundary_T}
\eeq
so that
\[
\begin{aligned}
& \PttD:\Theta^k(\M)\to \Theta^k(\dM) 
& \qquad 
& \frakT:\Theta^k(\M)\to \Theta^k(\dM) \\
& \PnnD:\Theta^k(\M)\to \Theta^{k-1}(\dM) 
& \qquad 
& \frakT^*:\Theta^k(\M)\to \Theta^{k-1}(\dM).
\end{aligned}
\]

\subsection{Bilaplacian analysis}

Since symmetry is preserved by the differential and boundary operators, the elliptic theory of the bilaplacian adapts naturally to the context of symmetric forms. 

\begin{lemma}
\label{lem:BH_T}
The following relations hold:
\[
\begin{aligned}
& \text{If}\quad \psi\in\BH_\TT^{k,m}(\M) \qquad \text{then}\qquad \psi^T\in\BH_\TT^{m,k}(\M) \\
& \text{If}\quad \psi\in\BH_\NN^{k,m}(\M) \qquad \text{then}\qquad \psi^T\in\BH_\NN^{m,k}(\M) \\
& \text{If}\quad \psi\in\BH_\TN^{k,m}(\M) \qquad \text{then}\qquad \psi^T\in\BH_\NT^{m,k}(\M) \\
& \text{If}\quad \psi\in\BH_\NT^{k,m}(\M) \qquad \text{then}\qquad \psi^T\in\BH_\TN^{m,k}(\M).
\end{aligned}
\]
\end{lemma}

\begin{proof}
This is an immediate consequence of \eqref{eq:symmetries_preserved_by_H_F} and \eqref{eq:boundary_T}.
\end{proof}

\begin{proposition}
\label{prop:green_commute_with_transpose}
Let $\psi\in W^{s,p}\Omega^{k,m}(\M)$. Then,
\[
\begin{aligned}
& (G_\TT\psi^T)^T = G_\TT\psi
&\qquad
& (G_\NN\psi^T)^T = G_\NN\psi \\
& (G_\NT\psi^T)^T = G_\TN\psi
&\qquad
& (G_\TN\psi^T)^T = G_\NT\psi,
\end{aligned}
\]
and
\[
\begin{aligned}
& (P_\TT\psi^T)^T = P_\TT\psi
&\qquad
& (P_\NN\psi^T)^T = P_\NN\psi \\
& (P_\NT\psi^T)^T = P_\TN\psi
&\qquad
& (P_\TN\psi^T)^T = P_\NT\psi.
\end{aligned}
\]
\end{proposition}

\begin{proof}
Take for example $\RR = \TT$. Since
\[
\psi = \Bg G_{\TT}\psi + P_{\TT}\psi,
\]
if follows from \eqref{eq:symmetries_preserved_by_H_F} that 
\[
\psi = \Bg (G_{\TT}\psi^T)^T+(P_{\TT}\psi^T)^T,
\]
hence 
\[
\Bg ((G_{\TT}\psi^T)^T-G_{\TT}\psi)=P_{\TT}\psi -(P_{\TT}\psi^T)^T.
\]
By \lemref{lem:BH_T},
\[
P_{\TT}\psi -(P_{\TT}\psi^T)^T\in\BH_{\TT}^{k,m}(\M).
\]
Integrating by parts, using the fact that both terms satisfy $\TT$-boundary conditions,
\[
\bra \Bg ((G_{\TT}\psi^T)^T-G_{\TT}\psi), P_{\TT}\psi -(P_{\TT}\psi^T)^T \ket =0,
\]
proving that 
\[
P_{\TT}\psi^T = (P_{\TT}\psi)^T
\Textand 
(G_{\TT}\psi^T)^T - G_{\TT}\psi\in \BH_{\TT}^{k,m}(\M). 
\]
However, by the very construction of the biharmonic Green operators \eqref{eq:green_operator}, 
\[
\image{G_{\TT}}\,\bot\, \image{P_{\TT}},
\]
which combined with \lemref{lem:BH_T} yields that
\[
(G_{\TT}\psi^T)^T - G_{\TT}\psi = 0.
\]
The same holds for $\RR = \NN$. For the case of, say, $\RR = \NT$, we obtain that
\[
\Bg ((G_{\NT}\psi^T)^T-G_{\TN}\psi)=P_{\TN}\psi -(P_{\NT}\psi^T)^T,
\]
from which we proceed as above.
\end{proof}

We conclude that two of the Green operators restrict to symmetric forms,
\[
\begin{aligned}
& G_\TT:W^{s,p}\Theta^k(\M)\to W^{s+4,p}\Theta^k(\M) \\
& G_\NN:W^{s,p}\Theta^k(\M)\to W^{s+4,p}\Theta^k(\M).
\end{aligned}
\]
Their kernels are denoted
\[
\begin{aligned}
\calS\BH^k_\TT(\M) &= \ker\frakB_\TT|_{\Theta^k(\M)} \\ 
\calS\BH^k_\NN(\M) &= \ker\frakB_\NN|_{\Theta^k(\M)} 
\end{aligned}
\]

We then formulate regular elliptic boundary-value problems in $\Theta^k(\M)$; in this section we define
\[
\begin{aligned}
\frakB_\TT &= (\Bg,\PttD,\PntD,\frakT,\frakF^*,\PttD\bHg^*,\frakT\bHg^*) \\
\frakB_\NN &= (\Bg,\PnnD,\PntD,\frakT^*,\frakF^*,\PnnD\bHg,\frakT^*\bHg),
\end{aligned}
\]
which are both symmetric with respect to transposition. The following theorem is the symmetric version of \thmref{thm:elliptic_regularity}.

\begin{theorem}
\label{thm:regular_ellipticity_symmetric_1}
The boundary-value problems for $\psi\in \Theta^k(\M)$, defined by the operators $\frakB_\TT$ and $\frakB_\NN$
are regular elliptic in the sense of \defref{def:regular_elliptic}.
\end{theorem}

\begin{proof}
The proof is an adaptation of the proof of \thmref{thm:elliptic_regularity}; note however that not all the boundary operators map symmetric forms to symmetric form.

In the context of \defref{def:regular_elliptic}, $\E_1=\E_2=\Upsilon^kT^*\M$, i.e.,
\[
\dim_\R\E_1=\dim_\R\E_2=\half\binom{d}{k}\Brk{\binom{d}{k}+1}.
\]
The fact that $P_{\Bg}(x,\xi) = |\xi|^4$ is an isomorphism if and only if $\xi\ne 0$ is immediate. We therefore  only need to verify the second condition in \defref{def:regular_elliptic}. 
In view of \eqref{eq:boundary_T}, for $(x,\xi)\in T^*\M|_{\dM}$,
\[
\begin{split}
& P_{\PttD}(x,\xi):\Upsilon^kT^*_{x}\M\to \Upsilon^kT_{x}^*\dM  \equiv \bbF_1|_x \\
& P_{\frakT}(x,\xi):\Upsilon^kT_{x}^*\M\to \Upsilon^kT_{x}^*\dM \equiv \bbF_2|_x \\ 
& P_{\PntD}(x,\xi):\Upsilon^kT_{x}^*\M\to \Lambda^{k-1,k}T^*_x\dM  \equiv \bbF_3|_x \\
& P_{\frakF^*}(x,\xi):\Upsilon^kT_{x}^*\M\to \Lambda^{k-1,k}T_{x}^*\dM  \equiv \bbF_4|_x \\
& P_{\PttD \bHg^*}(x,\xi):\Upsilon^kT_{x}^*\M\to \Upsilon^{k-1}T_{x}^*\dM \equiv \bbF_5|_x \\
& P_{\frakT \bHg^*}(x,\xi):\Upsilon^kT_{x}^*\M\to \Upsilon^{k-1}T_{x}^*\dM \equiv \bbF_6|_x. 
\end{split}
\]
A straightforward calculation yields
\[
\dim_\R \oplus_{j=1}^6 \bbF_j = \binom{d}{k} \Brk{\binom{d}{k}+ 1}.
\]

The calculation of $\calM^+_{x,\xi}$ for $x\in\dM$ and $\xi\in\frakt T_x^*\M$ is the same as in \ref{sec:proof_of_ellipticiity}, 
\[
\calM^+_{x,\xi}=\BRK{\omega_0\exp{(-|\xi|s)}+\lambda_0s\exp{(-|\xi|s)} ~:~ \omega_0,\lambda_0\in\mathbb{C}\otimes\Upsilon^kT^*_x\M}.
\]
In particular,
\[
\dim_\bbC \calM^+_{x,\xi} = \binom{d}{k} \Brk{\binom{d}{k}+ 1}.
\]
Thus, the map $\Xi_{x,\xi}:\calM^+_{x,\xi} \to \mathbb{C}\otimes(\oplus_{j=1}^6 \bbF_j|_x)$ is a linear map between vector spaces of the same dimension. As in \secref{sec:proof_of_ellipticiity}, the proof will be completed by showing that for $\sigma\in \calM^+_{x,\xi}$, 
\[
\Xi_{x,\xi} \sigma = 0
\qquad\text{if and only if}\qquad 
\omega_0 = \lambda_0 = 0. 
\]
One direction is again obvious, leaving the ``only if" part to be proved. By \eqref{eq:boundary_T}, for $\psi\in\Theta^k(\M)$,
\[
\PtnD \psi=0 \text{ if and only if }\PntD \psi=0, \qquad \frakF\psi=0 \text{ if and only if } \frakF^*\psi=0,
\]
which implies that if $\Xi_{x,\xi} \sigma = 0$, then
\[
P_{\PtnD}(x,\xi+\imath\,\partial_{s}dr)\sigma(0)=0
\Textand 
P_{\frakF}(x,\xi+\imath\,\partial_{s}dr)\sigma(0)=0.
\]   
Since $\Upsilon^kT^*_{x}\M\subseteq \Lambda^{k,k}T^*_{x}\M$, if follows that $\sigma$ is in the kernel of the solution map associated with the boundary-value problem \eqref{eq:boundary_data_TT}. As a result, the proof in \secref{sec:proof_of_ellipticiity} applies here too.  
\end{proof}

We can now repeat almost verbatim the proof of \thmref{thm:bilaplacian_full_problem}, to obtain:

\begin{theorem}
\label{thm:symmetric_bilaplacian_full_problem}
For every $s\in \bbN$ and $p\geq 2$, given data
\[
\upsilon\in W^{s,p}\Theta^k(\M),
\]
and
\[
\begin{aligned}
& \phi_1\in W^{s+7/2,p}\Theta^k(\dM) 
& \qquad 
& \rho_1\in W^{s+5/2,p}\Theta^k(\dM) \\
& \phi_2\in W^{s+7/2,p}\Omega^{k-1,k+1}(\dM) 
& \qquad 
& \rho_2\in W^{s+5/2,p}\Omega^{k-1,k+1}(\dM) \\
& \theta\in W^{s+3/2,p}\Theta^{k-1}(\dM)
& \qquad 
& \tau\in W^{s+1/2,p}\Theta^{k-1}(\dM)
\end{aligned}
\]
(the type is implied by the context), there exists a solution $\psi\in W^{s+4,p}\Theta^k(\M)$ to the problem
\[
\frakB_\TT\psi = (\upsilon,\phi_1,\phi_2,\rho_1,\rho_2,\theta,\tau) 
\]
if and only if for all $\zeta\in \S\BH_\TT^k(\M)$,
\[
\bra \upsilon,\zeta \ket = \int_{\dM}\Brk{(\theta,\frakT^*\zeta)_{\gD}-(\tau,\PnnD\zeta)_{\gD}}\VolumeD.
\]
The solution is unique up to an element in $\S\BH_\TT^k(\M)$. 
\end{theorem}

\section{Bilaplacian decomposition of double forms}
\label{sec:the_decompostion}
\subsection{An exact diagram}

Fix $(k,m)$ and suppose that for the diagram

\[
\begin{tikzcd}
&& {\Omega^{k+1,m+1}(\M)} \\
&& {} \\
{\Omega^{k-1,m+1}(\M)} && {\Omega^{k,m}(\M)} && {\Omega^{k+1,m-1}(\M)} \\
&& {} \\
&& {\Omega^{k-1,m-1}(\M)}
\arrow["{\bHg}"', curve={height=12pt}, from=5-3, to=3-3]
\arrow["{\bHg}"', curve={height=12pt}, from=3-3, to=1-3]
\arrow["{\bHg^*}", curve={height=-12pt}, tail reversed, no head, from=3-3, to=1-3]
\arrow["{\bHg^*}"', curve={height=12pt}, from=3-3, to=5-3]
\arrow["{\bFg}", curve={height=-12pt}, from=3-1, to=3-3]
\arrow["{\bFg^*}"', curve={height=12pt}, tail reversed, no head, from=3-1, to=3-3]
\arrow["{\bFg^*}", curve={height=-12pt}, from=3-5, to=3-3]
\arrow["{\bFg}", curve={height=-12pt}, from=3-3, to=3-5]
\end{tikzcd}
\]
the following conditions hold,
\beq
\begin{gathered}
\bHg\bHg=0 \qquad \bFg^*\bFg^*=0 \qquad \bFg^*\bHg=0 \qquad \bHg\bFg^*=0
\\ \bHg^*\bHg^*=0 \qquad \bFg\bFg=0 \qquad \bFg\bHg^*=0 \qquad \bHg^*\bFg=0
\\ \bFg\bHg=0 \qquad \bHg^*\bFg^*=0 \qquad \bHg\bFg=0 \qquad \bFg\bHg=0.
\end{gathered}
\label{eq:exact_relations}
\eeq

This condition holds for all $(k,m)$ when $(\M,\g)$ is locally-flat,  setting $\bHg=\Hg$ and $\bFg=\Fg$. 
This is an immediate consequence of the exactness relations  $\dg\dg=0$, $\dgV\dgV=0$, $\deltag\deltag=0$ and $\deltagV\deltagV=0$ and the commutation relations \eqref{eq:dg_dgV_commutators}. 
In the sequel \cite{KL21b}, we will construct operators satisfying these condition for $k=m=1$ and $\g$ having constant sectional curvature.


\begin{proposition}
\label{prop:boundary_conditions_orthogonality}
Let $\psi\in W^{2,p}\VectorFormsM$, let $\eta\in W^{4,q}\VectorFormsM$ and let $\kappa,\lambda,\mu\in W^{2,q}\VectorFormsM$ (the types are implied in each context) with $1/p+1/q=1$. Then,
\beq
\text{if} \qquad 
\psi \in \ker(\PttD,\frakT)
\qquad \text{then} \qquad 
\begin{cases}  \bra \bHg\psi, \bHg\eta\ket= \bra \psi, \bHg^*\bHg\eta\ket \\   \bra \bHg\psi, \bHg^*\kappa\ket=0 \\ \bra \bHg \psi, \bFg^*\lambda\ket=0 \\ \bra \bHg\psi, \bFg\mu\ket=0 
\end{cases}
\label{eq:H_image_orthogonality}
\eeq
\[
\text{if} \qquad 
\psi \in \ker(\PnnD,\frakT^*)
\qquad \text{then} \qquad 
\begin{cases} \bra \bHg^*\psi, \bHg^*\eta\ket=\bra \psi, \bHg \bHg^*\eta\ket \\ \bra \bHg^*\psi, \bHg\kappa\ket=0 \\ \bra \bHg^* \psi, \bFg^*\lambda\ket=0 \\ \bra \bHg^*\psi, \bFg\mu\ket=0 
\end{cases}
\]
\[
\text{if} \qquad 
\psi \in \ker(\PntD,\frakF^*)
\qquad \text{then} \qquad 
\begin{cases} \bra \bFg^*\psi, \bFg^*\eta\ket=\bra \psi, \bFg\bFg^*\eta\ket \\ \bra \bFg^*\psi, \bHg\kappa\ket=0 \\ \bra \bFg^* \psi, \bHg^*\lambda\ket=0 \\ \bra \bFg^*\psi, \bFg\mu\ket=0 
\end{cases}
\]
\[
\text{if} \qquad 
\psi \in \ker(\PtnD,\frakF)
\qquad \text{then} \qquad 
\begin{cases} \bra \bFg\psi, \bFg\eta\ket=\bra \psi, \bFg^* \bFg\eta\ket \\ \bra \bFg\psi, \bHg\kappa\ket=0 \\ \bra \bFg \psi, \bHg^*\lambda\ket=0 \\ \bra \bFg\psi, \bFg^*\mu\ket=0 .
\end{cases}
\]
\end{proposition}

\begin{proof}
The first assertion follows from substituting sequentially  $\eta\mapsto \bHg\eta$ $\eta\mapsto \bHg^*\eta$, $\eta\mapsto \bFg\eta$ and $\eta\mapsto \bFg^*\eta$ in \corrref{cor:integration_by_parts_second_order}
using \eqref{eq:exact_relations}. The second assertion is obtained similarly, switching the roles of $\psi$ and $\eta$. The third and fourth assertion are obtained similarly.
\end{proof} 

\begin{corollary}
\label{corr:7.2}
Let $\psi\in W^{3,p}\VectorFormsM$. Then,
\beq
\begin{aligned}
&\text{if} 
&\quad 
&\psi \in \ker(\PttD,\frakT)
&\qquad \text{then} 
&\qquad &
& \bHg\psi\in \ker(\PttD,\frakT,\PntD,\frakF^*,\PtnD,\frakF) \\
&\text{if} 
&\quad 
&\psi \in \ker(\PnnD, \frakT^*)
&\qquad \text{then} 
&\qquad &
&\bHg^*\psi\in\ker(\PnnD, \frakT^*,\PntD,\frakF^*,\PtnD,\frakF) \\
&\text{if} 
&\quad 
&\psi \in \ker(\PntD, \frakF^*)
&\qquad \text{then} 
&\qquad &
&\bFg^*\psi\in\ker(\PntD, \frakF^*,\PttD,\frakT,\PnnD,\frakT^*) \\
&\text{if} 
&\quad 
&\psi \in \ker(\PtnD, \frakF)
&\qquad \text{then} 
&\qquad &
&\bFg\psi\in\ker(\PtnD, \frakF,\PttD,\frakT,\PnnD,\frakT^*).
\end{aligned}
\label{eq:operators_inherit_conditions}
\eeq
\end{corollary}

\begin{proof}
It suffices to prove the statement for smooth forms. 

Applying twice the integration by parts formula \eqref{eq:H_general_integral}, using the fact that $\bHg\bHg=0$ and $\bHg^*\bHg^*=0$, we obtain for all $\eta\in\VectorFormsM$,
\[
\begin{split}
0 &= \int_{\dM}\Brk{(\PttD\bHg\psi,\frakT^*\eta)_{\gD}-(\frakT\bHg\psi,\PnnD\eta)_{\gD}}\VolumeD \\
&+ \int_{\dM}\Brk{(\PttD\psi,\frakT^*\bHg^*\eta)_{\gD}-(\frakT\psi,\PnnD\bHg^*\eta)_{\gD}}\VolumeD.
\end{split}
\]
For $\psi\in\ker(\PttD, \frakT)$,
\[
0 = \int_{\dM}\Brk{(\PttD\bHg\psi,\frakT^*\eta)_{\gD}-(\frakT\bHg\psi,\PnnD\eta)_{\gD}}\VolumeD.
\]
Since by \lemref{lemma:prescribe_boundary_conditions}, $\PnnD\eta$ and $\frakT^*\eta$ can be set arbitrarily, we conclude that
\[
\bHg\psi \in \ker(\PttD,\frakT).
\]

%

Similarly, combining \eqref{eq:H_general_integral} and \eqref{eq:H_image_orthogonality},
\[
\begin{split}
0 &= \int_{\dM}\Brk{(\PtnD\bHg\psi,\frakF^*\eta)_{\gD}-(\frakF\bHg\psi,\PntD\eta)_{\gD}} \VolumeD \\
&+ \int_{\dM}\Brk{(\PttD\psi,\frakT^*\bFg^*\eta)_{\gD}-(\frakT\psi,\PnnD\bFg^*\eta)_{\gD}}\VolumeD.
\end{split}
\]
For $\psi\in\ker(\PttD, \frakT)$, using the fact that $\PntD\lambda$ and $\frakF^*\lambda$ can be prescribed arbitrarily, we obtain that
\[
\bHg\psi \in \ker(\PtnD,\frakF).
\]
We prove that $\bHg\psi \in \ker(\PntD,\frakF^*)$ in a similar way, thus completing the first assertion.
The other assertions are proved in the same way.
\end{proof}

\subsection{Orthogonal decomposition in manifolds with boundary} 

We now use the elliptic regularity of \thmref{thm:elliptic_regularity} to derive a decomposition for $\Omega^{k,m}(\M)$, assuming the exactness relations \eqref{eq:exact_relations}.

\begin{theorem}
\label{thm:decompostion_1}
Consider the following subspaces of $\VectorFormsKM$, 
\beq
\begin{split}
\EE^{k,m}(\M) &=  \BRK{\bHg\alpha ~:~ \alpha\in\ker(\PttD,\frakT)} \\
\CC^{k,m}(\M) &=  \BRK{\bHg^*\beta ~:~ \beta\in\ker(\PnnD,\frakT^*)} \\
\EC^{k,m}(\M) &= \BRK{\bFg\gamma ~:~ \gamma\in\ker(\PtnD,\frakF)} \\
\CE^{k,m}(\M) &= \BRK{\bFg^*\lambda ~:~ \lambda\in\ker(\PntD,\frakF^*)} .
\end{split}
\label{eq:subspaces_decomposition}
\eeq
(The symbols $\calE$ and $\calC$ stands for ``exact" and ``co-exact" in both form and vector parts.)
Let $W^{s,p}\EE^{k,m}(\M)$ etc. stand for the Sobolev versions, i.e.,   
$\alpha,\beta,\lambda,\delta\in W^{s+2,p}\VectorFormsM$ (the type is implied by the context). Then, for all $s\in\bbN$ and $p\ge2$ there exists a unique $L^2$-orthogonal decomposition,
\beq
\begin{split}
W^{s,p}\VectorFormsKM &= W^{s,p}\EE^{k,m}(\M)\oplus W^{s,p}\CC^{k,m}(\M)\\&
\oplus W^{s,p}\CE^{k,m}(\M)\oplus W^{s,p}\EC^{k,m}(\M)\oplus W^{s,p}\BHkm.
\end{split}
\label{eq:second_order_decompostion_initial}
\eeq
The splitting is both algebraic and topological in the $W^{s,p}$-topology. Furthermore, each of the subspaces in this decomposition is closed in the $W^{s,p}$-topology.
\end{theorem}

\begin{proof}
Let $\psi\in W^{s,p}\VectorFormsKM$. From \eqref{eq:green_operator}, for every $\RR\in\BRK{\TT,\TN,\NT,\NN}$,
\beq
\begin{split}
\psi &= \bHg \bHg^*G_\RR\psi + \bHg^*\bHg G_\RR\psi + \bFg \bFg^* G_\RR\psi + \bFg^* \bFg G_\RR\psi  + P_\RR\psi.
\end{split}
\label{eq:potentials}
\eeq
From the construction of the Green operators, 
\[
G_\RR\psi\in W^{s+4,p}\VectorFormsKM \cap H^4_\RR\VectorFormsKM.
\]
Setting
\[
\alpha_\psi = \bHg^*G_\TT\psi
\qquad \beta_\psi = \bHg G_\NN\psi \qquad
\gamma_\psi = \bFg^* G_\TN\psi
\qquad \lambda_\psi = \bFg G_\NT\psi,
\]
it follows from the boundary properties of the image of $G_\RR$ that
\[
\begin{aligned}
& \bHg\alpha_\psi\in W^{s,p}\EE^{k,m}(\M) 
& \qquad 
& \bHg^*\beta_\psi\in W^{s,p}\CC^{k,m}(\M) \\
& \bFg\gamma_\psi\in W^{s,p}\EC^{k,m}(\M)
& \qquad 
& \bFg^*\lambda_\psi\in W^{s,p}\CE^{k,m}(\M).
\end{aligned}
\]
The fact that these elements are mutually orthogonal in $\scrH$ is due to \eqref{eq:H_image_orthogonality} and its dual versions. 

We then consider
\[
\kappa_\psi = \psi-\bHg\alpha_\psi-\bHg^*\beta_\psi-\bFg\gamma_\psi-\bFg^*\lambda_\psi,
\]
and prove that it is orthogonal to each of the spaces in \eqref{eq:subspaces_decomposition}. 
Decomposing $\psi$ as in \eqref{eq:potentials} for $\RR=\TT$,
\[
\kappa_\psi =  \bHg^* \bHg(G_\TT-G_\NN)\psi + \bFg \bFg^* (G_\TT - G_\TN)\psi + \bFg^* \bFg(G_\TT - G_\NT)\psi + P_\TT\psi.
\]

In view of \eqref{eq:H_image_orthogonality}, every $\bHg\alpha\in W^{s,p}\EE^{k,m}(\M)$
is orthogonal to each of the four terms on the right-hand side, namely
$\kappa_\psi\perp  W^{s,p}\EE^{k,m}(\M)$; the other orthogonal properties are proved similarly.
Thus,
\[
\kappa_\psi\in\brk{W^{s,p}\EE^{k,m}(\M)\oplus W^{s,p}\CC^{k,m}(\M)\oplus W^{s,p}\CE^{k,m}(\M)\oplus W^{s,p}\EC^{k,m}(\M)}^\perp.
\]

We next show that this orthogonal complement coincides with the biharmonic module $W^{s,p}\BHkm$. 
One side of the inclusion is immediate: biharmonic functions are orthogonal to the spaces  \eqref{eq:subspaces_decomposition}. For the other direction,
let 
\[
\kappa\in\brk{W^{s,p}\EE^{k,m}(\M)\oplus W^{s,p}\CC^{k,m}(\M)\oplus W^{s,p}\CE^{k,m}(\M)\oplus W^{s,p}\EC^{k,m}(\M)}^\perp.
\]
By definition, for every $\bHg\alpha\in W^{s,p}\EE^{k,m}(\M)$, $\bHg^*\beta\in W^{s,p}\CC^{k,m}(\M)$, $\bFg\gamma\in W^{s,p}\CE^{k,m}(\M)$ and $\bFg^*\lambda\in W^{s,p}\EC^{k,m}(\M)$,
\[
\begin{split}
0 &= \bra \kappa ,\bHg\alpha\ket \EwR{eq:H_general_integral} \bra \bHg^*\kappa,\alpha\ket \\
0 &= \bra \kappa ,\bHg^*\beta\ket \EwR{eq:H_general_integral} \bra \bHg\kappa,\beta\ket \\
0 &= \bra \kappa ,\bFg\gamma\ket \EwR{eq:H_general_integral} \bra \bFg^*\kappa,\gamma\ket \\
0 &= \bra \kappa ,\bFg^*\lambda\ket \EwR{eq:H_general_integral} \bra \bFg\kappa,\lambda\ket. 
\end{split}
\]
Since the spaces $\ker(\PttD,\frakT)$, $\ker(\PnnD,\frakT^*)$, $\ker(\PntD,\frakF^*)$ and $\ker(\PtnD,\frakF)$, are  dense in $\scrH$, we conclude that  $\kappa\in W^{s,p}\BHkm$, hence the decomposition 
\[
\psi = \bHg  \alpha_\psi + \bHg^* \beta_\psi   + \bFg^*\lambda_\psi + \bFg\gamma_\psi + \kappa_\psi .
\]
satisfies \eqref{eq:second_order_decompostion_initial}.

It remains to show the $W^{s,p}$-closeness of the spaces in the decomposition. We do it for $W^{s,p}\EE^{k,m}(\M)$; the others follow in a similar way. 

Let $\bHg\alpha_j\to \eta$ in $W^{s,p}\VectorFormsKM$. 
Decomposing the latter,
\[
\eta= \bHg\alpha_{\eta}+\bHg^*\beta_{\eta}+\bFg^*\lambda_{\eta}+\bFg\gamma_{\eta}+\kappa_{\eta},
\] 
by Pythagoras' law and the orthogonality result,
\[
\Norm{\eta-\bHg\alpha_j}_{L^2}^2=\Norm{\bHg\alpha_{\eta}-\bHg\alpha_j}_{L^2}^2+\Norm{\bHg^*\beta_{\eta}}_{L^2}^2+\Norm{\bFg^*\lambda_{\eta}}_{L^2}^2+\Norm{\bFg\gamma_{\eta}}_{L^2}^2+\Norm{\kappa_{\eta}}_{L^2}^2.
\]
Since, in particular, $\bHg\alpha_j\longrightarrow\eta$ in $L^2$, we obtain
\[
\begin{gathered}
\Norm{\bHg\alpha_{\eta}-\bHg\alpha_j}_{L^2}^2\longrightarrow 0 \\
\bHg^*\beta_{\eta} = \bFg^*\lambda_{\eta} = \bFg\gamma_{\eta} = \kappa_{\eta} =0,
\end{gathered}
\]
i.e., $\bHg\alpha_{\eta}  = \eta \in W^{s,p}\VectorFormsKM$, which proves that $W^{s,p}\EE^{k,m}(\M)$ is indeed closed. 

\end{proof}

The forms $\alpha_\psi$, $\beta_\psi$, $\lambda_\psi$  and $\gamma_\psi$ in the decomposition,
\[
\psi=\bHg\alpha_\psi+\bHg^*\beta_\psi+\bFg^*\lambda_\psi+\bFg\gamma_\psi+\kappa_{\psi}
\]
are referred to as potentials of $\psi$. Note the considerable  gauge freedom in the choice of potentials: we can alter $\alpha_\psi$  by elements in $\ker (\bHg,\PttD,\frakT)$, $\beta_\psi$ by elements in $\ker (\bHg^*,\PnnD,\frakT^*)$, etc. The following proposition shows how the gauge freedom can be exploited to control the norms of the potentials.

\begin{proposition}
\label{prop:gauge_freedom_1}
In the decomposition of $\psi$, the potential,
\[
\alpha_{\psi}=\bHg^*G_\TT\psi
\]
satisfies
\[
\alpha_{\psi}\in\Image{\bHg^*}
\]
and
\[
\Norm{\alpha_\psi}_{W^{s+2,p}(\M)}\lesssim \Norm{\psi}_{W^{s,p}(\M)}.
\]
Similar estimates apply to the other potentials of $\psi$ as well. 
\end{proposition}

\begin{proof}
The first assertion is immediate. The bound on the norm follows from the continuity of $G_\TT:W^{s,p}\VectorFormsM\to W^{s+4,p}\VectorFormsM$ and the continuity of $\bHg^*:W^{s+4,p}\VectorFormsM\to W^{s+2,p}\VectorFormsM$.
\end{proof}

\subsection{Orthogonal decomposition of symmetric forms}

For $k=m$, we examine the decomposition of $\Theta^{k}(\M)$. Consider the diagram
\[
\begin{tikzcd}
&& {\Theta^{k+1}(\M)} \\
&& {} \\
{}&& {\Theta^{k}(\M)} && {\Omega^{k+1,k-1}(\M)} \\
&& {} \\
&& {\Theta^{k-1}(\M)}
\arrow["{\bHg}"', curve={height=12pt}, from=5-3, to=3-3]
\arrow["{\bHg}"', curve={height=12pt}, from=3-3, to=1-3]
\arrow["{\bHg^*}", curve={height=-12pt}, tail reversed, no head, from=3-3, to=1-3]
\arrow["{\bHg^*}"', curve={height=12pt}, from=3-3, to=5-3]
\arrow["{\tfrac12(\bFg^* + (\bFg^*(\cdot))^T)}", curve={height=-12pt}, from=3-5, to=3-3]
\arrow["{\bFg}", curve={height=-12pt}, from=3-3, to=3-5]
\end{tikzcd}
\]
along with the exactness conditions
\[
\begin{gathered}
\bHg\bHg=0 \qquad \bFg\bHg=0 \qquad \bHg(\bFg^*+(\bFg^*(\cdot))^T)=0 \\
\bHg^*\bHg^*=0 \qquad \bFg\bHg^*=0 \qquad \bHg^*(\bFg^*+\bFg^*(\cdot)^T)=0. 
\end{gathered}
\]

The list of exactness conditions is shorter than in \eqref{eq:exact_relations}, since for symmetric forms
$\bFg$ can be expressed in terms of $\bFg^*$. 
The symmetric version of the decomposition theorem, whose proof is essentially a repetition of the proof of \thmref{thm:decompostion_1}, reads:

\begin{theorem}
\label{thm:decompostion_symmetric}
Consider the following subspaces of $\Theta^k(\M)$, 
\[
\begin{split}
&\calS\EE^k(\M) = \BRK{\bHg\alpha ~:~ \alpha\in\ker({\PttD},{\frakT})\cap\Theta^{k-1}(\M)} \\
&\calS\CC^k(\M) = \BRK{\bHg^*\beta ~:~ \beta\in\ker({\PnnD},{\frakT}^*)\cap\Theta^{k+1}(\M)} \\
&\calS\EC^k(\M) = \BRK{(\bFg^*\lambda)^T+\bFg^*\lambda ~:~ \lambda\in\ker({\PntD},{\frakF^*})\cap\Omega^{k+1,k-1}(\M)} \\
&\calS\BH^k(\M) = \BH^{k,k}(\M)\cap\Theta^k(\M).
\end{split}
\]
Let $W^{s,p}\calS\EE^k(\M)$ etc. stand for their Sobolev versions. Then, for all $s\in\bbN$ and $p\ge2$ there exists an $L^2$-orthogonal decomposition,
\[
\begin{split}
W^{s,p}\Theta^k(\M) &= W^{s,p}\calS
\EE^k(\M)\oplus W^{s,p}\calS\CC^k(\M)
\oplus W^{s,p}\calS\EC^k(\M) 
\oplus W^{s,p}\calS\BH^k(\M).
\end{split}
\]
The splitting is both algebraic and topological in the $W^{s,p}$-topology. Furthermore, each of the subspaces in this decomposition is closed in the $W^{s,p}$-topology. 
\end{theorem}

\subsection{Decompositions of the biharmonic module}

One of the components of the decomposition \eqref{eq:second_order_decompostion_initial} is the biharmonic module $\BH^{*,*}(\M)$, defined by \eqref{eq:def_BHkm}. Like the harmonic module in Hodge theory, the biharmonic module decomposes further, and its decomposition is needed in applications, notably in the solution of boundary-value problems.

To formulate our results, we introduce the following submodules of $W^{s,p}\BH^{*,*}(\M)$, for $s,p\ge 2$:
\[
\begin{aligned}
W^{s,p}\BH^{*,*}_{\bHg}(\M) &= W^{s,p}\BH^{*,*}(\M) \cap  \bHg(W^{s+2,p} \VectorFormsM) \\
W^{s,p}\BH^{*,*}_{\bHg^*}(\M) &= W^{s,p}\BH^{*,*}(\M) \cap  \bHg^*(W^{s+2,p} \VectorFormsM) \\ 
W^{s,p}\BH^{*,*}_{\bFg^*}(\M) &= W^{s,p}\BH^{*,*}(\M) \cap  \bFg^*(W^{s+2,p} \VectorFormsM)  \\
W^{s,p}\BH^{*,*}_{\bFg}(\M) &= W^{s,p}\BH^{*,*}(\M) \cap  \bFg(W^{s+2,p} \VectorFormsM) \\ 
W^{s,p}\BH^{*,*}_{\bHg+\bHg^*}(\M) &= W^{s,p}\BH^{*,*}_{\bHg}(\M) + W^{s,p}\BH^{*,*}_{\bHg^*}(\M)  \\ 
\cdots &= \cdots \\
W^{s,p}\BH^{*,*}_{tt}(\M) &= W^{s,p}\BH^{*,*}(\M) \cap \ker(\PttD,\frakT) \\
W^{s,p}\BH^{*,*}_{nn}(\M) &= W^{s,p}\BH^{*,*}(\M) \cap \ker(\PnnD,\frakT^*) \\
W^{s,p}\BH^{*,*}_{nt}(\M) &= W^{s,p}\BH^{*,*}(\M) \cap \ker(\PntD,\frakF^*) \\
W^{s,p}\BH^{*,*}_{tn}(\M) &= W^{s,p} \BH^{*,*}(\M) \cap \ker(\PtnD,\frakF).
\end{aligned}
\]
Recall that we already have the finite-dimensional modules, $\BH_\RR^{*,*}(\M)$, $\RR\in\{\TT,\NN,\NT,\TN\}$, which consist of smooth sections. 

For $s=0$, it is not a priori clear why $\psi\in L^2\BH^{*,*}(\M)$ has well-defined boundary projections $\frakT\psi$, $\frakT^*\psi$, $\frakF^*\psi$ and $\frakF\psi$. The space $L^2\BH^{*,*}(\M)$ is the completion of $\BH^{*,*}(\M)$, in which $\Bg=0$ identically, hence $\psi\in L^2\BH^{*,*}(\M)$ satisfies $\Bg\psi=0$ distributively. By \cite[p.~459]{Tay11a}, $\psi$ has well-defined boundary data up to order four. We conclude that the spaces $L^2 \BH^{*,*}_{tt}(\M)$, $L^2 \BH^{*,*}_{tn}(\M)$, etc. are well-defined.

We start by establishing the closedness of those spaces:

\begin{lemma}
\label{lemma:images_are_closed}
The spaces
\[
L^2 \BH^{*,*}_{\bHg}(\M)
\quad
L^2 \BH^{*,*}_{\bHg^*}(\M)
\quad
L^2 \BH^{*,*}_{\bFg^*}(\M)\
\textand
L^2 \BH^{*,*}_{\bFg}(\M)
\]
are closed in  $L^2\VectorFormsM$.
\end{lemma}

\begin{proof}
We prove the first item. 
Since $L^2 \BH^{*,*}(\M)$ is closed in $L^2\VectorFormsM$, 
it suffices to prove that $\bHg(W^{2,2} \VectorFormsM)$ is closed in $L^2\VectorFormsM$.
By \thmref{thm:decompostion_1} and the exactness relations \eqref{eq:exact_relations},
\[
\bHg(W^{2,2}\VectorFormsM) = \bHg(W^{2,2}\CC^{*,*}(\M)).
\]
It suffices to prove that the image of the restriction
\[
\bHg:W^{2,2}\CC^{*,*}(\M) \to L^2\Omega^{*,*}(\M)
\]
is closed.
Given $\bHg^*\beta\in W^{2,2}\CC^{*,*}(\M)$, we have by \corrref{corr:7.2},
\[
\bHg^*\beta\in\ker{(\PnnD,\frakT^*,\PntD,\frakF^*,\PtnD,\frakF)},
\]
and by \eqref{eq:general_estimate_final} for $\RR=\NN$, noting that $P_\NN \bHg^*\beta =0$, 
\[
\Norm{\bHg^*\beta}_{W^{2,2}}^2\lesssim  \Norm{\bHg\bHg^*\beta}_{L^2}^2.
\]
By \cite[p.~583]{Tay11a}, $\bHg(W^{2,2}\CC^{*,*}(\M))$ is closed in $L^2\VectorFormsM$. 
\end{proof}

\begin{lemma}
\label{lem:7.13}
For every $s,p\ge 2$,
\[
\begin{aligned}
W^{s,p}\BH^{*,*}_{\bHg}(\M) &= W^{s,p}\BH^{*,*}(\M) \cap  \bHg(W^{2,2} \VectorFormsM) \\
W^{s,p}\BH^{*,*}_{\bHg^*}(\M) &= W^{s,p}\BH^{*,*}(\M) \cap  \bHg^*(W^{2,2} \VectorFormsM) \\ 
W^{s,p}\BH^{*,*}_{\bFg^*}(\M) &= W^{s,p}\BH^{*,*}(\M) \cap  \bFg^*(W^{2,2} \VectorFormsM)  \\
W^{s,p}\BH^{*,*}_{\bFg}(\M) &= W^{s,p}\BH^{*,*}(\M) \cap  \bFg(W^{2,2} \VectorFormsM).
\end{aligned}
\]
That is, if for example
\[
\psi = \bHg\kappa \in W^{s,p}\BH^{*,*}(\M)
\]
for some $\kappa\in W^{2,2} \VectorFormsM$, then $\psi = \bHg\vp$ for some $\vp\in W^{s+2,p} \VectorFormsM$.
\end{lemma}

\begin{proof}
We prove the first item. Let $\vp\in \VectorFormsM$ be the solution of the boundary-value problem,
\[
(\Bg,\PnnD,\PntD,\PtnD,\frakT^*,\frakF^*,\frakF,\PnnD\bHg,\frakT^*\bHg)\vp = (0,0,0,0,0,0,0,\PnnD\bHg\kappa.\frakT^*\bHg\kappa),
\]
By \thmref{thm:bilaplacian_full_problem} (for $\NN$ rather than $\TT$), this system is solvable if
\[
\int_{\dM}\Brk{(\PnnD\bHg\kappa,\frakT\zeta)_{\gD} -(\frakT^*\bHg\kappa,\PttD\zeta)_{\gD} }\, \VolumeD = 0
\]
for all $\zeta\in\BH_\NN^{k,m}(\M)$, which is indeed satisfied since 
\[
\int_{\dM}\Brk{(\PnnD\bHg\kappa,\frakT\zeta)_{\gD} -(\frakT^*\bHg\kappa,\PttD\zeta)_{\gD} }\, \VolumeD = 
\bra \bHg^*\bHg\kappa,\zeta\ket - \bra\bHg\kappa,\bHg \zeta\ket ,
\]
and $\bHg\kappa$ and $\zeta$ are biharmonic. The solution $\vp$ is in $W^{s+2,p}\VectorFormsM$.

Since $\bHg\kappa$ is biharmonic,
\[
\bHg^*(\bHg\vp -  \bHg\kappa) + \bHg^*\bHg\vp + \bFg^*\bFg\vp + \bFg\bFg^*\vp = 0.
\]
Since, by construction,  $\bHg\vp -  \bHg\kappa\in \ker(\PnnD,\frakT^*)$, the first term is $L^2$-orthogonal to the other three, which implies that $\bHg^*(\bHg\vp -  \bHg\kappa)=0$. Finally,
\[
0 = \bra\bHg^*(\bHg\vp -  \bHg\kappa), \vp - \kappa\ket = \bra\bHg\vp -  \bHg\kappa, \bHg\vp - \bHg\kappa\ket,
\]
i.e., $\bHg\kappa = \bHg\vp$, which completes the proof.
\end{proof}

\begin{corollary}
The spaces
\[
W^{s,p} \BH^{*,*}_{\bHg}(\M)
\quad
W^{s,p} \BH^{*,*}_{\bHg^*}(\M)
\quad
W^{s,p} \BH^{*,*}_{\bFg^*}(\M)\
\textand
W^{s,p} \BH^{*,*}_{\bFg}(\M)
\]
are closed in  $W^{s,p}\VectorFormsM$.
\end{corollary}

\begin{proof}
We prove the first item. Let $\psi_n \in W^{s,p} \BH^{*,*}_{\bHg}(\M)$ be a converging sequence in $W^{s,p}\VectorFormsM$; denote the limit by $\eta$. Since it is also convergent in $L^2 \VectorFormsM$, it follows from  \lemref{lemma:images_are_closed} that $\eta\in L^2 \BH^{*,*}_{\bHg}(\M)$. By \lemref{lem:7.13}, $\eta\in W^{s,p} \BH^{*,*}_{\bHg}(\M)$, proving that $W^{s,p} \BH^{*,*}_{\bHg}(\M)$ is closed.
\end{proof}

A first family of decompositions is given by the following theorem:

\begin{theorem}
\label{thm:BH_decompose1}
The biharmonic modules $W^{s,p}\BH^{*,*}(\M)$ decompose in four different ways,
\beq
\begin{split}
W^{s,p}\BH^{*,*}(\M) &= \BH^{*,*}_\TT(\M) \oplus W^{s,p}\BH^{*,*}_{\bHg^*+\bFg^*+\bFg}(\M)   \\
&= \BH^{*,*}_\NN(\M) \oplus W^{s,p}\BH^{*,*}_{\bHg+\bFg^*+\bFg}(\M)     \\
&= \BH^{*,*}_\NT(\M) \oplus W^{s,p}\BH^{*,*}_{\bFg+\bHg+\bHg^*}(\M)    \\
&= \BH^{*,*}_\TN(\M) \oplus W^{s,p}\BH^{*,*}_{\bFg^*+\bHg+\bHg^*}(\M).
\end{split}
\label{eq:BH_decompose1}
\eeq
The splittings are algebraic, topological and $L^2$-orthogonal. 
\end{theorem}

\begin{proof}
It suffices to prove the first item. Since $\BH^{k,m}_\TT(\M)$ is finite-dimensional, 
\[
L^2 \BHkm = \BH^{k,m}_\TT(\M)  \oplus (\BH^{k,m}_\TT(\M))^\bot.
\]
We need to show that 
\[
W^{s,p} \VectorFormsKM \cap (\BH^{k,m}_\TT(\M))^\bot = W^{s,p}\BH^{k,m}_{\bHg^*+\bFg^*+\bFg}(\M).
\]
One direction is trivial, as  $\BH^{k,m}_\TT(\M)$ is orthogonal to the images of $\bHg^*$, $\bFg^*$ and $\bFg$. For the other direction, let $\psi\in W^{s,p} \VectorFormsKM\cap (\BH^{k,m}_\TT(\M))^\bot$, namely, $\psi = (I-P_\TT)\psi$. Then,
\[
\psi = \Bg G_\TT\psi = \bHg\bHg^* G_\TT\psi + \bHg^*\bHg G_\TT\psi + \bFg^*\bFg G_\TT\psi + \bFg\bFg^* G_\TT\psi.
\]
The first term on the right-hand side is in $W^{s,p}\EE^{k,m}(\M)$, hence $L^2$-orthogonal to all other terms, hence vanishes. Setting 
\[
\alpha = \bHg G_\TT\psi
\qquad
\beta = \bFg G_\TT\psi
\Textand
\gamma = \bFg^* G_\TT\psi,
\]
which, by the continuity of $G_\TT$ and the second-order differential operators are all in $W^{s+2,p}\VectorFormsM$, we obtain that
\[
\psi = \bHg^*\alpha + \bFg^*\beta + \bFg\gamma,
\]
which concludes the proof.
\end{proof}

A second family of decompositions is given by the following theorem:

\begin{theorem}
\label{thm:BH_decompose2}
The biharmonic module $\BHkm$ decomposes in four different ways,
\[
\begin{split}
L^2 \BH^{*,*}(\M) &= L^2 \BH^{*,*}_{\bHg}(\M)  \oplus L^2 \BH^{*,*}_{nn}(\M)  \\
&= L^2 \BH^{*,*}_{\bHg^*}(\M) \oplus L^2 \BH^{*,*}_{tt}(\M)  \\
&= L^2 \BH^{*,*}_{\bFg^*}(\M) \oplus L^2 \BH^{*,*}_{tn}(\M)  \\
&= L^2 \BH^{*,*}_{\bFg}(\M) \oplus L^2 \BH^{*,*}_{nt}(\M) .
\end{split}
\]
The splittings are algebraic, topological and $L^2$-orthogonal; in particular, all the spaces are closed.
\end{theorem}
\begin{proof}
We prove the first item. By \lemref{lemma:images_are_closed}, $L^2 \BH^{*,*}_{\bHg}(\M)$ is closed in $L^2\VectorFormsM$, hence also in $L^2 \BH^{*,*}(\M)$, from which we deduce that
\[
L^2 \BH^{*,*}(\M) = L^2 \BH^{*,*}_{\bHg}(\M) \oplus (L^2 \BH^{*,*}_{\bHg}(\M))^\bot,
\]
where the orthogonal complement here is in $L^2 \BH^{*,*}(\M)$. We thus need to show that
\[
(L^2 \BH^{*,*}_{\bHg}(\M))^\bot = L^2 \BH^{*,*}_{nn}(\M).
\]

Let $\psi\in L^2\BH_{nn}^{*,*}(\M)$; then $\bHg^*\psi=0$ in the distributive sense, hence for all $\bHg\eta\in L^2 \BH^{*,*}_{\bHg}(\M)$,
\[
\bra \bHg\eta,\psi\ket =0,
\]  
from which we deduce that $L^2\BH_{nn}^{*,*}(\M) \subset (L^2 \BH^{*,*}_{\bHg}(\M))^\bot$.

Conversely, let $\psi\in (L^2\BH_{\bHg}^{*,*}(\M))^\bot$. Since $\psi$ is biharmonic, 
it is also $L^2$-orthogonal to $L^2\EE^{*,*}(\M)$, from which we conclude that it is $L^2$-orthogonal to $\bHg(W^{2,2}\VectorFormsM)$, i.e., to the image of $\bHg$ in $L^2\VectorFormsM$ (by \thmref{thm:decompostion_1} $\image \bHg = \EE^{*,*}(\M) \oplus \BH^{*,*}_{\bHg}(\M)$).
Thus, for every $\alpha\in W^{2,2}\VectorFormsM$,
\[
\bra \psi,\bHg\alpha\ket = 0.
\]
Since $\bHg^*\psi=0$ in the distributive sense, 
\[
0=\bra \psi, \bHg\alpha\ket \EwR{eq:H_general_integral}\int_{\dM}\Brk{(\PttD\alpha,\frakT^*\psi)_{\gD}-(\frakT\alpha,\PnnD\psi)_{\gD}}\,\VolumeD.
\]
By \lemref{lemma:prescribe_boundary_conditions},  $\PttD\alpha$ and $\frakT\alpha$ can be prescribed arbitrarily, from which follow that $\PnnD\psi=0$ and $\frakT^*\psi=0$. 
Thus, $(L^2\BH_{\bHg}^{*,*}(\M))^\bot\subset L^2\BH_{nn}^{*,*}(\M)$. 
\end{proof}

Finally, we observe that above results can be restated for $\calS\BH^k(\M)$. 

\subsection{Boundary-value problems}

In this subsection we demonstrate how the orthogonal decomposition of double forms (\thmref{thm:decompostion_1}) and the decompositions of the biharmonic module (\thmref{thm:BH_decompose1}) can be used to solve boundary-value problems. As throughout the present section, we assume a specific $(k,m)$, with  $\bHg,\bFg$ satisfying the exactness relations \eqref{eq:exact_relations}. In this setting, the technique we employ follows the same lines as \cite{Sch95b}, with some adjustments due to the fact that the operators here are second-order.

\begin{theorem}
\label{thm:basic_boundary_value_problem}
Given $\chi\in\Omega^{k,m}(\M)$ and $\phi,\mu\in \Omega^{k,m}(\partial\M)$, the boundary-value problem
\beq
\begin{gathered}
\bHg^*\psi=\chi
\qquad 
(\PnnD,\frakT^*) \psi = (\phi,\mu)
\end{gathered}
\label{eq:simple_boundary_problem}
\eeq
is solvable for $\psi\in\Omega^{k+1,m+1}(\M)$ if and only if
\beq
\begin{gathered}
\chi\ \bot \ \EE^{k,m}(\M) \qquad \chi\ \bot \ \CE^{k,m}(\M) \qquad \chi\ \bot \ \EC^{k,m}(\M) \\
\bra\chi,\kappa\ket=\int_{\dM}\Brk{(\phi,\frakT \kappa)_{\gD}-(\mu,\PttD \kappa)_{\gD}}\VolumeD \qquad \forall\kappa\in\BH^{k,m}(\M).
\end{gathered} 
\label{eq:integrability_conditions_simple_1}
\eeq
Moreover, the solution can be chosen to satisfy,
\[
\psi\in\Image{\bHg}.
\]
If the data has only Sobolev regularity, then \eqref{eq:simple_boundary_problem} is solvable at an appropriate level of regularity, and $\psi$ can be chosen to satisfy
\beq
\begin{split}
\Norm{\psi}_{W^{s+2,p}(\M)}&\lesssim \Norm{\chi}_{W^{s,p}(\M)}+\Norm{\phi}_{W^{s+2-1/p,p}(\dM)}+\Norm{\mu}_{W^{s+1-1/p,p}(\dM)}.
\end{split}
\label{eq:estimate_simple_boundary_problem_1}
\eeq
\end{theorem}

\begin{proof}
The necessity of the conditions in \eqref{eq:integrability_conditions_simple_1} is clear from the decomposition \eqref{eq:second_order_decompostion_initial} and by the integration by parts formula \eqref{eq:H_general_integral}. It is left to prove that they are also sufficient. Suppose $\chi$ satisfies \eqref{eq:integrability_conditions_simple_1}. By \thmref{thm:decompostion_1}, 
\[
\chi=\bHg^*\beta_{\chi}+\kappa_{\chi},
\]
where $\beta_\chi\in\ker(\PnnD,\frakT^*)$ and $\kappa_{\chi}\in \BH^{k,m}(\M)$. 
By means of \lemref{lemma:prescribed_boundary_data_second_order}
we produce $\omega\in\Omega^{k,m}(\M)$ satisfying
\[
\begin{gathered}
\omega|_{\dM}=0 \qquad 
\nabg_{\dr}\omega|_{\dM}=0  \qquad
\PnnD \bHg\omega=\phi \qquad 
\frakT^*\bHg\omega=\mu. 
\end{gathered}
\]
We set $\eta=\bHg^*\bHg\omega$, which by \thmref{thm:decompostion_1}
decomposes as follows,
\[
\eta=\bHg^*\beta_{\eta}+\kappa_{\eta},
\]
where $\beta_\eta\in\ker(\PnnD,\frakT^*)$ and $\kappa_{\eta}\in \BH^{k,m}(\M)$.
We set
\[
\psi=\beta_{\chi}+\bHg\omega-\beta_{\eta},
\]
which satisfies by construction,
\[
\begin{gathered}
\bHg^*\psi=\chi-\kappa_{\chi}+\kappa_{\eta}
\qquad 
\PnnD\psi=\phi \qquad \frakT^*\psi=\mu.
\end{gathered}
\]
We will prove that $\kappa_{\eta}-\kappa_{\chi}=0$. By setting $\kappa_{\eta}-\kappa_{\chi}=\kappa\in\BH^{k,m}(\M)$ and integrating by parts, we find
\[
\begin{split}
\bra \kappa,\kappa\ket &=
\bra \bHg^*\psi-\chi,\kappa\ket=
-\bra\chi,\kappa\ket+\int_{\dM}\Brk{(\mu,\PttD \kappa)_{\gD}-(\phi,\frakT \kappa)_{\gD}}\VolumeD,
\end{split}
\]
which vanishes by the second line in \eqref{eq:integrability_conditions_simple_1}. Hence $\kappa_{\eta}-\kappa_{\chi}=0$, and therefore $\bHg\psi=\chi$.
By \propref{prop:gauge_freedom_1}, we may choose $\beta_{\chi},\beta_{\eta}\in\Image{\bHg}$, so indeed $\psi\in\Image{\bHg}$.

Finally, all the tools used in the proof have a corresponding Sobolev version. By \lemref{lemma:prescribed_boundary_data_second_order},
\[
\Norm{\bHg\omega}_{W^{s+2,p}(\M)}\lesssim\Norm{\phi}_{W^{s+2-1/p,p}(\dM)}+\Norm{\mu}_{W^{s+1-1/p,p}(\dM)}.
\]
On the other hand, invoking again the choice of gauge in \propref{prop:gauge_freedom_1},
\[
\begin{aligned}
\Norm{\beta_{\eta}}_{W^{s+2,p}} &\lesssim\Norm{\bHg^* \bHg\omega}_{W^{s,p}} \lesssim 
\Norm{\bHg\omega}_{W^{s+2,p}} \\
\Norm{\beta_{\chi}}_{W^{s+2,p}} & \lesssim \Norm{\chi}_{W^{s,p}}.
\end{aligned}
\]
Combining, we obtain \eqref{eq:estimate_simple_boundary_problem_1}.
\end{proof}

We can improve the solvability conditions if we assume more regularity for the data, yielding a finite-dimensional obstruction:

\begin{theorem}
\label{thm:integrability_conditions_equivalence}
If the boundary data in \eqref{eq:simple_boundary_problem} is of at least $H^2$-regularity, the following integrability conditions are equivalent to the ones in \eqref{eq:integrability_conditions_simple_1}:
\beq
\begin{gathered}
\chi\in\ker{(\bHg^*,\bFg,\bFg^*)} \\
(\PnnD,\frakT^*,\PntD,\frakF,\PtnD,\frakF^*)(\chi-\bHg^*\omega)=0 \\
\bra\chi,\upsilon\ket= \int_{\dM}\Brk{(\phi,\frakT \upsilon)_{\gD}-(\mu,\PttD \upsilon)_{\gD}}\VolumeD 
\qquad \forall\upsilon\in\BH^{k,m}_{\NN}(\M)
\end{gathered} 
\label{eq:integrability_conditions_differential_simple}
\eeq
where $\omega\in H^4\Omega^{k+1,m+1}(\M)$ is an arbitrary double form extending the boundary data, i.e. $\PnnD\omega=\phi$ and $\frakT^*\omega=\mu$ (the arbitrariness is due to \eqref{eq:operators_inherit_conditions}). 
\end{theorem}

\begin{proof}
From \thmref{thm:decompostion_1}, the first line in \eqref{eq:integrability_conditions_simple_1} and the first line in \eqref{eq:integrability_conditions_differential_simple} are equivalent. The second line in \eqref{eq:integrability_conditions_simple_1} implies the third line in \eqref{eq:integrability_conditions_differential_simple}.

The second line in \eqref{eq:integrability_conditions_simple_1} implies that
\[
\bra\chi,\kappa\ket=\int_{\dM}\Brk{(\phi,\frakT \kappa)_{\gD}-(\mu,\PttD \kappa)_{\gD}}\VolumeD 
\EwR{eq:H_general_integral}\bra \bHg^*\omega,\kappa\ket \qquad \forall\kappa\in\BHkm.
\]
Thus,
\[
\chi -  \bHg^*\omega \perp \BHkm.
\]
Moreover,
\[
\chi -  \bHg^*\omega \in\ker(\bHg^*,\bFg,\bFg^*),
\]
from which follows that
\[
\chi -  \bHg^*\omega \perp \EE^{k,m}(\M), \CE^{k,m}(\M), \EC^{k,m}(\M).
\]
I.e.,
\[
\chi -  \bHg^*\omega \in \CC^{k,m}(\M),
\]
which implies that 
\[
(\PnnD,\frakT^*,\PntD,\frakF,\PtnD,\frakF^*)(\chi -  \bHg^*\omega) = 0.
\]


In the other direction, suppose that \eqref{eq:integrability_conditions_differential_simple} holds. We need to prove that the second line in  \eqref{eq:integrability_conditions_simple_1} holds, namely, that
\[
\bra\chi,\kappa\ket=\int_{\dM}\Brk{(\phi,\frakT \kappa)_{\gD}-(\mu,\PttD \kappa)_{\gD}}\VolumeD \qquad \forall\kappa\in\BH^{k,m}(\M).
\]

By \thmref{thm:BH_decompose1}, every $\kappa\in\BHkm$ decomposes $L^2$-orthogonally into
\[
\BH^{k,m}_{\bHg+\bFg+\bFg^*}(\M) \oplus \BH_{\NN}^{k,m}(\M),
\]
hence by the third line of \eqref{eq:integrability_conditions_differential_simple}, it suffices to obtain the integrability condition for $\kappa \in \BH^{k,m}_{\bHg+\bFg+\bFg^*}(\M)$. That is, that for all $\bHg\theta+\bFg\epsilon+\bFg^*\varrho\in\BH^{k,m}_{\bHg+\bFg+\bFg^*}(\M)$
\beq
\begin{split}
\bra\chi,\bHg\theta+\bFg\epsilon+\bFg^*\varrho\ket&=\int_{\dM}(\phi,\frakT( \bHg\theta+\bFg\epsilon+\bFg^*\varrho))_{\gD}\VolumeD\\
&-\int_{\dM}(\mu,\PttD(\bHg\theta+\bFg\epsilon+\bFg^*\varrho))_{\gD}\VolumeD.
\end{split}
\label{eq:new_integrability}
\eeq
By the first line in \eqref{eq:integrability_conditions_differential_simple}, for every $\theta,\epsilon,\varrho$,
\[
\bra\bHg^*\chi,\theta\ket  = 0.
\]
Integrating by parts,  combining with the second line in \eqref{eq:integrability_conditions_differential_simple}, and integrating by parts twice more
\[
\begin{split}
\bra\chi, \bHg\theta\ket &=  \int_{\dM}\Brk{(\frakT^* \chi,\PttD\theta)_\gD - (\PnnD \chi,\frakT\theta)_\gD}\VolumeD \\
&= \int_{\dM}\Brk{(\frakT^* \bHg^*\omega,\PttD\theta)_\gD - (\PnnD  \bHg^*\omega,\frakT\theta)_\gD}\VolumeD \\
&= \bra\bHg^*\omega, \bHg\theta\ket \\
&=  \int_{\dM}\Brk{(\PnnD\omega,\frakT\bHg\theta)_{\gD} - (\frakT^*\omega,\PttD\bHg\theta)_{\gD}}\VolumeD \\
&=  \int_{\dM}\Brk{(\phi,\frakT\bHg\theta)_{\gD} - (\mu,\PttD\bHg\theta)_{\gD}}\VolumeD.
\end{split}
\]

%
%

Similarly, since $\bFg^*\chi=0$,
\[
\begin{split}
\bra\chi,\bFg\epsilon\ket &= \int_{\dM}\Brk{(\frakF^*\chi,\PtnD\epsilon)_{\gD}-(\PntD\chi,\frakF \epsilon)_{\gD}}\VolumeD  \\
&= \int_{\dM}\Brk{(\frakF^*\bHg^*\omega,\PtnD\epsilon)_{\gD}-(\PntD\bHg^*\omega,\frakF \epsilon)_{\gD}}\VolumeD  \\
&= \bra\bHg^*\omega,\bFg\epsilon\ket \\ 
&=  \int_{\dM}\Brk{(\PnnD\omega,\frakT\bFg\epsilon)_{\gD} - (\frakT^*\omega,\PttD\bFg\epsilon)_{\gD}}\VolumeD \\
&=  \int_{\dM}\Brk{(\phi,\frakT\bFg\epsilon)_{\gD} - (\mu,\PttD\bFg\epsilon)_{\gD}}\VolumeD.
\end{split}
\]
By the same argument,
\[
\bra\chi,\bFg^*\varrho\ket = \int_{\dM}\Brk{(\phi,\frakT\bFg^*\varrho)_{\gD} - (\mu,\PttD\bFg^*\varrho)_{\gD}}\VolumeD.
\]
Adding up the last three equations, we recover the desired result.
\end{proof}

The last two theorems can be reformulated for $\bHg^*$ replaced by $\bHg$, $\bFg$ and $\bFg^*$ and with $\NN$ boundary conditions replaced by $\TT$, $\NT$ and $\TN$ boundary conditions.

These results can be also reformulated for symmetric forms, using the theory developed in \secref{sec:symmetric_forms}. 

\begin{theorem}
\label{thm:basic_boundary_value_problem_symmetric}
Given $\chi\in\Theta^{k}(\M)$ and $\phi,\mu\in \Theta^{k}(\dM)$, the boundary-value problem
\[
\bHg^*\psi=\chi 
\qquad
(\PnnD,\frakT^*)\psi = (\phi,\mu)
\]
is solvable for $\psi\in\Theta^{k+1}(\M)$ if and only if the following integrability conditions are satisfied:
\[
\begin{gathered}
\chi\in\ker{(\bHg^*,\bFg)} \qquad 
(\PnnD,\frakT^*,\PntD,\frakF)(\chi-\bHg^*\omega)=0 \\
\bra\chi,\upsilon\ket= \int_{\dM}\Brk{(\phi,\frakT \upsilon)_{\gD}-(\mu,\PttD \upsilon)_{\gD}}\VolumeD \qquad \forall\upsilon\in\S\BH^{k}_{\NN}(\M),
\end{gathered} 
\]
where $\omega\in\Theta^{k}(\M)$ is an arbitrary form extending the boundary data, i.e., $\PnnD\omega=\phi$, $\frakT^*\omega=\mu$. The solution $\psi$ can be chosen to satisfy
\[
\psi\in\Image{\bHg|_{\Theta^{k}(\M)}}.
\]
If the data has Sobolev regularity, then $\psi$ can be chosen to satisfy an estimate
\[
\begin{split}
\Norm{\psi}_{W^{s+2,p}(\M)}&\lesssim \Norm{\chi}_{W^{s,p}(\M)}+\Norm{\phi}_{W^{s+2-1/p,p}(\dM)}+\Norm{\mu}_{W^{s+1-1/p,p}(\dM)}.
\end{split}
\]
\end{theorem}

\begin{proof}
The proof is step-by-step identical to the proof of \thmref{thm:basic_boundary_value_problem}, using the machinery of symmetric decompositions (\thmref{thm:decompostion_symmetric}). 
\end{proof}
\appendix
\section{Proofs of technical lemmas and propositions}

\begin{PROOF}{\propref{prop:dg_commutes_starV}}
For $\psi = \omega\otimes V$, where $\omega\in\Omega^k(\M)$ and $V\in \Gamma(\Lambda^mT^*\M)$,
\[
\begin{split}
\starG^V\dg\psi &\EwR{eq:leibnitz_rule_vector} \starG^V(d\omega\otimes V + (-1)^m\omega\wedge\nabg V) \\
&\EwR{eq_HodgeV} d\omega\otimes \starG^V V + (-1)^m\omega\wedge\starG^V \nabg V \\
&= d\omega\otimes \starG^V V + (-1)^m\omega\wedge \nabg \starG^V V \\
&\EwR{eq:leibnitz_rule_vector} \dg\starG^V\psi,
\end{split}
\]
where in the passage to the third line we used the fact that $\nabg_X$ commutes with the metric operation $\starG^V$.
\end{PROOF}

\begin{PROOF}{\lemref{lem:3.2}}
For every $\beta\in\Omega^{d-1-k}(\M)$,
\[
\begin{split}
(\beta, \starG(\alpha\wedge\omega))\, \VolumeG &= (-1)^{dk + d+k+1}  \beta \wedge \alpha\wedge\omega  \\
&= (-1)^{d+1} (\beta\wedge\alpha , \starG\omega)\, \VolumeG \\
&= (-1)^k (\alpha\wedge \beta , \starG\omega)\, \VolumeG \\
&= (-1)^k (\beta ,  i_{\alpha^{\sharp}} \starG\omega)\, \VolumeG,
\end{split}
\]
where in the passage to the last line we used the fact that $(\alpha\wedge)$ is the adjoint of $i_{\alpha^{\sharp}}$.
\end{PROOF}

\begin{PROOF}{\lemref{lem:3.4}}
By duality, it suffices to prove the first statement, and 
by bilinearity, it suffices to consider $\psi = \alpha\otimes V$ and $\vp = \beta\otimes W$. 
Moreover, as the identities are tensorial, we may prove them pointwise. We prove just for $A\in\Omega^{1,1}(\M)$, as the general case is proved similarly. Since $A$ is symmetric, there exists at every $p\in\M$  an orthonormal frame field $\{E_i\}$ for $T_{p}\M$,  such that,
\[
A = \sum_i (i_{E_i} i_{E_i}^V  A)\, \vartheta^i \otimes (\vartheta^i)^T,
\]
where $\{\vartheta^i\}$ is the dual coframe.
It follows that
\[
\begin{split}
(\psi,A\wedge \vp) &= \sum_i (i_{E_i} i_{E_i}^V  A) (\alpha\otimes V, (\vartheta^i\wedge \beta)\otimes (\vartheta^i \wedge W)) \\
&= \sum_i (i_{E_i} i_{E_i}^V  A) (i_{E_i} \alpha\otimes i_{E_i}^V V, \beta\otimes W) \\
&= \sum_i (i_{E_i} i_{E_i}^V  A) (i_{E_i} i_{E_i}^V \psi, \vp).
\end{split}
\]
On the other hand,
\[
\begin{split}
\trace_A \psi 
&= (-1)^{dk+dm} \stargEps\stargEps^V (A\wedge) \stargEps\stargEps^V \psi  \\
&= (-1)^{dk+dm} \sum_i 
(i_{E_i} i_{E_i}^V  A) 
(\stargEps (\vartheta^i\wedge \stargEps \omega))\otimes(\starG^V (\vartheta^i \wedge  \starG^V V)) \\
&\EwR{eq:star_alpha_omega}
\sum_i  (i_{E_i} i_{E_i}^V  A)  (i_{E_i} \omega)\otimes (i_{E_i} V) \\
&= \sum_i (i_{E_i} i_{E_i}^V  A) i_{E_i} i_{E_i}^V \psi,
\end{split}
\]
thus proving that
\[
(\psi,A\wedge \vp) = (\trace_A \psi ,\psi).
\]
\end{PROOF}

\begin{PROOF}{\lemref{lem:3.5}}
The first identity follows from the symmetry of $A$, as
\[
\begin{split}
(\trace_A \psi)^T &\EwR{eq:formulas_traceA}  (-1)^{dk+dm} (\starG\starG^V (A\wedge) \starG\starG^V\psi)^T \\
&\EwR{eq_HodgeV} (-1)^{dk+dm} \starG^V\starG (A\wedge) \starG^V\starG\psi^T \\
&\EwR{eq:formulas_traceA} \trace_{A}\psi^T.
\end{split}
\]
As for the second identity,
\[
\begin{split}
(i_A \psi)^T &\EwR{eq:formulas_traceA} (-1)^{dm + d} (\starG^V (A\wedge)  \starG^V\psi)^T \\
&\EwR{eq_HodgeV} (-1)^{dm + d} \stargEps (A\wedge)  \starG\psi^T \\
&\EwR{eq:formulas_traceA} i_A^*\psi^T.
\end{split}
\]
\end{PROOF}

\begin{PROOF}{\lemref{lem:bianchiG}}
It suffices to prove the claim for $\psi = \omega\otimes V$, where $\omega\in\Omega^{k,0}(\M)$ and $V\in \Omega^{0,m}(\M)$.
Let $\{E_i\}$ be a local orthonormal frame and let $\{\vartheta^i\}$ be its dual frame; note that 
\[
\g = \sum_i \vartheta^i \otimes \vartheta^i.
\]
Then,
\[
\begin{split}
i_\g \psi &\EwR{eq:formulas_traceA} (-1)^{dm + d} \starG^V (\g \wedge)  \starG^V (\omega \otimes V) \\
&= (-1)^{dm + d} \starG^V (\g \wedge)  (\omega  \otimes  \starG^V V) \\
&\EwR{eq:def_wedge} (-1)^{dm + d} \sum_i  (\vartheta^i\wedge \omega)  \otimes  \starG^V (\vartheta^i \wedge \starG^V V) \\
&\EwR{eq:star_alpha_omega}(-1)^{dm + m} \sum_i   (\vartheta^i\wedge \omega)  \otimes  i_{E_i}\starG^V  \starG^V V \\
&\EwR{eq:Hodge_inverse}  \sum_i   (\vartheta^i\wedge \omega)  \otimes  i_{E_i} V .
\end{split}
\]
The last line is precisely $\G\psi$ in a local frame. 
\end{PROOF}

\begin{PROOF}{\lemref{lem:commute_d_G}}
Take for example the first item in \eqref{eq:commutations_dg_G},
\[
\begin{split}
\G\dg &\EwR{eq:formulas_traceA} (-1)^{dm+d} \starG^V (\g\wedge) \starG^V  \dg \\
&\EwR{eq:dg_commutes_starV}  (-1)^{dm+d} \starG^V (\g\wedge) \dg \starG^V  \\
&\EwR{eq:g_wedge_commute_with_dg} -(-1)^{dm+d} \starG^V \dg (\g\wedge)  \starG^V \\
&\EwR{eq:dg_commutes_starV}  -(-1)^{dm+d} \dg \starG^V  (\g\wedge)  \starG^V \\
&\EwR{eq:formulas_traceA} - \dg\G.
\end{split}
\]
The three other relations are proved similarly. As for \eqref{eq:commutations_dg_bianchi},
\[
\begin{split}
\dg &\EwR{eq:formula_delta_nabla} (-1)^{dk+d+1} \starG \deltag \starG \\
&\EwR{eq:co-differential_trace} (-1)^{dk+d+1} \starG \brk{-\trace_\g\dgV - \dgV \trace_\g} \starG \\
&\EwR{eq:formulas_traceA}  \starG \brk{\starG \G \starG \dgV + \dgV\starG \G \starG} \starG \\
&\EwR{eq:dg_commutes_starV} \G \dgV + \dgV \G ,
\end{split}
\]
with the second relation obtained from the first by transposition.
\end{PROOF}

\begin{PROOF}{\propref{prop:duals_stars}}
We have
\[
\begin{split}
\starG^V\starG \dgV\dg \starG\starG^V\psi 
&\EwR{eq:dg_commutes_starV} (\starG^V \dgV \starG^V)(\starG \dg \starG\psi) 
\EwR{eq:formula_delta_nabla} (-1)^{dk+dm} \deltagV \deltag\psi \\
\starG \dgV\dg \starG\psi 
&\EwR{eq:dg_commutes_starV}  \dgV (\starG  \dg \starG) \psi \EwR{eq:formula_delta_nabla} (-1)^{dk+d+1} \dgV \deltag\psi \\
\starG^V \dgV\dg \starG^V\psi 
&\EwR{eq:dg_commutes_starV} (\starG^V \dgV \starG^V)  \dg \psi 
\EwR{eq:formula_delta_nabla} (-1)^{dm+d+1} \deltagV \dg\psi,
\end{split}
\]
and likewise for $\dg$ and $\dgV$ interchanged.
By the definitions \eqref{eq:HHFF}, we obtain the desired result. 
\end{PROOF}

\begin{PROOF}{\propref{prop:commute_H_G}}
It suffices to prove the first item as the other three follow by duality. The commutation of $\Hg$ with $(\g\wedge)$ is immediate as both $\dg$ and $\dgV$ anti-commute with $(\g\wedge)$. The commutation of $\Hg$ with $\G$ follows from \eqref{eq:commutations_dg_G} and \eqref{eq:commutations_dg_bianchi},
\[
\G\dg\dgV=-\dg\G\dgV=-\dg\dg+\dg\dgV\G,
\]
and
\[
\G\dgV\dg=\dg\dg-\dgV\G\dg=\dg\dg+\dgV\dg\G.
\]
Adding the two we obtain $\Hg\G = \G \Hg$.
Finally, since $\Hg$ commutes with $\G$ and $\Hg$ commutes with transposition,
\[
\GV \Hg \psi = (\G(\Hg\psi)^T)^T =  (\G \Hg\psi^T)^T = ( \Hg \G \psi^T)^T = 
( \Hg (\GV \psi)^T)^T = \Hg \GV \psi.
\]
\end{PROOF}

\begin{PROOF}{\lemref{lem:4.2}}
Eq.~\eqref{eq:tangent_operator_and_pullback} follows directly from the definition of $\frakt\omega$ as the range of $dj_\e$ is tangential vectors.
As for \eqref{eq:normal_operator_and_interior}, 
\[
\begin{split}
\frakn\omega(X_1,\dots,X_k) &=
\omega(X_1,\dots,X_k) - \omega(X_1^\parallel,\dots,X_k^\parallel) \\
&= \sum_{i=1}^k \omega(X_1,\dots,X_i^\bot,\dots, X_k) \\ 
&= \sum_{i=1}^k (-1)^{i-1}  dr(X_i) \idr \omega(X_1,\dots,\hat{X_i},\dots, X_k) \\ 
&= (dr \wedge  \idr \omega)(X_1,\dots,X_k).
\end{split}
\]
Since $dj_\e : T\calP_\e\to  (TU)^\parallel|_{\calP_\e}$ is an isomorphism, it follows that for every $X_1,\dots,X_k\in (TU)^\parallel|_{\calP_\e}$, which (with a slight abuse of notations) we identify with  $X_1,\dots,X_k\in T\calP_\e$,
\[
\frakt \omega(X_1,\dots,X_k) = \jEpsStar\omega(X_1,\dots,X_k).
\]
Finally, noting that $\frakn\omega$ is uniquely determined by $\idr\omega$ and that the latter is tangential, i.e., $\frakt \idr\omega = \idr\omega$, we conclude that $\frakn\omega$ is uniquely determined by $\jEpsStar \idr\omega$.
\end{PROOF}

\begin{PROOF}{\lemref{lem:PtPnStar}}
Let $\{E_i\}$ be an orthonormal frame for $TU$ with $E_d = \dr$; for $i<d$, we identify $E_i\in T\calP_\e$ with $E_i\in (TU)^\parallel|_{\calP_\e}$.  For $\sigma\in S_d$,
\[
\starG\omega(E_{\sigma(1)},\dots,E_{\sigma(d-k)}) = \sgn(\sigma) (-1)^{dk+k}\,\omega(E_{\sigma(d-k+1)},\dots,E_{\sigma(d)}).
\]
Suppose that $\sigma(j)\ne d$ for all $j\le d-k$ and without loss of generality suppose that $\sigma(d)=d$. Then,
\[
\begin{split}
\jEpsStar \starG\omega(E_{\sigma(1)},\dots,E_{\sigma(d-k)}) 
&= \sgn(\sigma) (-1)^{dk+k} (-1)^{k+1} \,\idr \omega(E_{\sigma(d-k+1)},\dots,E_{\sigma(d-1)}) \\
&= \sgn(\sigma)  (-1)^{dk+1} \, \jEpsStar \idr \omega(E_{\sigma(d-k+1)},\dots,E_{\sigma(d-1)}).
\end{split}
\]
Define $\tau\in S_{d-1}$ such that $\tau(j) = \sigma(j)$ for all $j\le d-1$. Since $\sgn(\tau) = \sgn(\sigma)$, 
\[
\begin{split}
\jEpsStar \starG\omega(E_{\tau(1)},\dots,E_{\tau(d-k)})
&= \sgn(\tau) (-1)^{dk+1} \,\jEpsStar \idr \omega(E_{\tau(d-k+1)},\dots,E_{\tau(d-1)}) \\
&=  (-1)^{d+1} \stargEps \jEpsStar \idr\omega(E_{\tau(1)},\dots,E_{\tau(d-k)}),
\end{split}
\]
i.e.,
\[
\jEpsStar  \starG\omega = (-1)^{d+1} \stargEps \jEpsStar \idr\omega.
\]
The second statement follows from the first by substituting $\omega\mapsto \starG\omega$.
\end{PROOF}

\begin{PROOF}{\lemref{lem:4.3}}
For every $Y_1,\dots,Y_k\in\frakX^\parallel(U) \simeq \frakX(\calP_\e)$,
\[
\begin{split}
\Pt\nabE_X\omega(Y_1,\dots,Y_k) &= \nabE_X\omega(Y_1,\dots,Y_k) \\
&\hspace{-2cm}= \nabE_X(\omega(Y_1,\dots,Y_k)) - \sum_{j=1}^k  \omega(Y_1,\dots,\nabg_X Y_j,\dots Y_k) \\
&\hspace{-2cm}\EwR{eq:connections_with_S} \nabE_X(\omega(Y_1,\dots,Y_k)) - \sum_{j=1}^k  \omega(Y_1,\dots,\nabla^{\gEps}_X Y_j,\dots Y_k) \\
&\hspace{-2cm}\qquad +  \sum_{j=1}^k  \frakh(X;Y_{j}) \omega(Y_1,\dots,\dr,\dots Y_k) \\
&\hspace{-2cm}= \nabE_X \Pt\omega(Y_1,\dots,Y_k) +  \sum_{j=1}^k  (-1)^{j+1} \frakh_\e(Y_{j};X) \idr \omega(Y_1,\dots,\hat{Y}_j,\dots, Y_k) \\
&\hspace{-2cm}= \nabE_X \Pt\omega(Y_1,\dots,Y_k) + (i_X^{V}\frakh_\e\wedge \Pn\omega)(Y_1,\dots,Y_k).
\end{split}
\]
\end{PROOF}

\begin{PROOF}{\lemref{lem:4.4}}
It suffices to prove the first statement for forms $\omega = \alpha\otimes V$, where $\alpha\in\Omega^k(U)$ and $V\in\Gamma(\E)\simeq \Omega^0(U;\E|_U)$. Since exterior derivatives commute with pullbacks and since $\Pn V=0$,
\[
\begin{split}
\Pt d^{\nabE}\omega &\EwR{eq:lebnitz_wedge} \Pt\brk{d\omega\otimes V + (-1)^k \omega\wedge \nabE V} \\
&\EwR{eq:Pt_Pn_wedge} \Pt d\alpha\otimes \Pt V + (-1)^k \Pt\alpha\wedge \Pt\nabE V \\
&\EwR{eq:commutator_Pt_nabla}  d \Pt \alpha\otimes \Pt V + (-1)^k \Pt\alpha\wedge \nabE \Pt V \\
&\EwR{eq:lebnitz_wedge} d^{\nabE} (\Pt \alpha \otimes \Pt V) 
\EwR{eq:Pt_Pn_wedge} d^{\nabE} \Pt\omega.
\end{split}
\]
As for the second statement, it follows by duality, substituting  $\omega\mapsto \starG\omega$ in the first statement, and using 
\eqref{eq:Hodge_star_normal_tangentA} and \eqref{eq:formula_delta_nabla}.
\end{PROOF}

\begin{PROOF}{\lemref{lem:4.5}}
It suffices to prove the first statement for scalar-valued forms: we need to prove that for all $\omega\in\Omega^k(U)$,
\[
\jEpsStar \idr d\omega = -  d \jEpsStar \idr\omega +
\jEpsStar\nabg_{\dr}\omega+ \calS_\e \jEpsStar \omega.
\]
By Cartan's magic formula, 
\[
\idr d\omega = - d \idr\omega + \calL_{\dr}\omega .
\]
By the Leibniz property of the Lie derivative and the relation between Lie and covariant differentiation,
\[
\begin{split}
\calL_{\dr}\omega(X_1,\dots,X_k) &= \dr(\omega(X_1,\dots,X_k)) - \sum_j \omega(X_1,\dots,\calL_{\dr}X_j,\dots X_k) \\
&= \dr(\omega(X_1,\dots,X_k)) - \sum_j \omega(X_1,\dots,\nabg_{\dr}X_j,\dots X_k)  \\
&\qquad +  \sum_j \omega(X_1,\dots,\nabg_{X_j}\dr,\dots X_k)   \\
&= \nabla_{\dr} \omega(X_1,\dots,X_k) -  \sum_j \omega(X_1,\dots,S(X_j),\dots X_k).
\end{split}
\]
Putting things together, using the fact that exterior differentials commute with pullbacks and the fact that the shape operator is determined by $S_\e$, we recover the desired result. The second statement follows once again by duality, substituting $\omega\mapsto\starG\omega$ into the first.
\end{PROOF}

\begin{PROOF}{\lemref{lem:4.11}}
For a vector field $X\in\frakX(\calP_\e)$,
\[
\begin{split}
\Ptt \nabg V(X) &= (\Pt(\nabg_X V)^T)^T \\
&= (\Pt\nabg_X V^T)^T\\ 
&\EwR{eq:commutator_Pt_nabla} (\nabla^{\gEps}_X \Pt V^T)^T + i_X^{V}\frakh_\e \wedge \Pn V^T)^T\\ 
&=  \nabla^{\gEps}_X (\Pt V^T)^T + i_X\frakh_\e \wedge (\Pn V^T)^T,
\end{split}
\]
and we note that $(\Pt V^T)^T = \Ptt V$ and $(\Pn V^T)^T = \Ptn V$.
\end{PROOF}

\begin{PROOF}{\lemref{lem:4.12}}
It suffices to prove the first assertion for $\psi=\omega\otimes V$ where $\omega\in\Omega^{k,0}(U)$ and
$V\in\Omega^{0,m}(U)$. Using the fact that $\Ptt d\omega = d\Ptt\omega$,
\[
\begin{split}
\Ptt \dg\psi &\EwR{eq:lebnitz_wedge} \Ptt(d\omega\otimes V + (-1)^k \omega\wedge\nabg V)  \\
&\EwR{eq:Ptt_Ptn_Pnt_Pnn_wedge} \Ptt d\omega \otimes \Ptt V + (-1)^k \Ptt\omega \wedge \Ptt \nabg V \\
&\EwR{eq:Ptt_nabla_V} d \Ptt \omega \otimes \Ptt V + (-1)^k \Ptt\omega \wedge (\nabg \Ptt V + \frakh_\e\wedge \Ptn V) \\
&\EwR{eq:Ptt_Ptn_Pnt_Pnn_wedge} \dgEps \Ptt\psi + \frakh_\e \wedge \Ptn\psi.
\end{split}
\]
The second assertion follows from the first by transposition and the symmetry of $\h_\e$.
The third assertion follows from the first after substituting $\psi\mapsto\starG^V\psi$.
The fourth assertion follows from the third by transposition.
\end{PROOF}

\begin{PROOF}{\lemref{lem:4.13}}
We start with
\[
\begin{split}
\Pnt \dg\psi &\EwR{eq:Ptt_et_al} (\Pt(\Pn\dg\psi)^T)^T \\
&\EwR{eq:exterior_covariant_derivative_and_normal} (\Pt(- \dg \Pn\psi + \Pt \nabg_{\dr}\psi -  
\calS_\e \Pt \psi)^T)^T \\
&= -(\Pt(\dg \Pn\psi)^T)^T + (\Pt(\Pt \nabg_{\dr}\psi)^T)^T - 
(\Pt(\calS_\e \Pt \psi)^T)^T \\
&= -(\Pt(\dg \Pn\psi)^T)^T + \Ptt \nabg_{\dr}\psi -  \calS_\e \Ptt\psi,
\end{split}
\]
and in the last step we used the fact that $\calS_\e \Ptt\psi = (\Pt(\calS_\e \Pt \psi)^T)^T$. It remains to examine the first term on the right-hand side: decomposing as before $\psi = \omega\otimes V$, we obtain
\[
\begin{split}
(\Pt(\dg \Pn\psi)^T)^T &\EwR{eq:Pt_Pn_wedge} (\Pt(\dg (\Pn\omega\otimes \Pt V))^T)^T \\
&= (\Pt(d\Pn\omega\otimes \Pt V + (-1)^k \Pn\omega\wedge \Pt \nabg V)^T)^T \\
&= (\Pt((d\Pn\omega)^T \otimes (\Pt V)^T + (-1)^k (\Pn\omega)^T\wedge (\Pt \nabg V)^T))^T \\
&= d\Pnt\omega\otimes \Ptt V + (-1)^k \Pnt\omega\wedge \Ptt \nabg V \\
&\EwR{eq:Ptt_nabla_V} d\Pnt\omega\otimes \Ptt V + (-1)^k \Pnt\omega\wedge\brk{\nabgEps \Ptt V + \frakh_\e\wedge \Ptn V} \\
&= \dgEps \Ptt\psi + (-1)^k \Pnt\omega\wedge \frakh_\e\wedge \Ptn V \\
&= \dgEps \Ptt\psi -  \frakh_\e\wedge \Pnn \psi,
\end{split}
\]
where in the passage to the fourth line we used the fact that $\Pt(d\Pn\omega)^T = (d\Pnt\omega)^T$,
and in the passage to the last line we used the fact that $\Pnn\psi = \Pnt\omega\otimes \Ptn V$.
The second assertion follows from the first by transposition; the other two follow after substituting $\psi\mapsto\starG^V\psi$.
\end{PROOF}

\begin{PROOF}{\lemref{lem:4.15}}
We have
\[
\begin{split}
\Ptt \dgV\dg\psi
& \EwR{eq:dnabla_pure_tangent}   \dgEpsV \Ptt\dg\psi + \frakh_\e \wedge \Pnt\dg\psi \\
& \EwR{eq:dnabla_pure_tangent}  \dgEpsV\brk{\dgEps \Ptt\psi + \frakh_\e \wedge \Ptn\psi} + \frakh_\e \wedge \Pnt\dg\psi   \\
&\EwR{eq:Codazzi_equation2} \dgEpsV\dgEps \Ptt\psi 
+  \dgEpsV\frakh_\e \wedge  \Ptn\psi 
- \frakh_\e \wedge \dgEpsV \Ptn\psi + \frakh_\e \wedge \Pnt\dg\psi.
\end{split}
\] 
Setting $\psi\mapsto\psi^T$ we obtain,
\[
\begin{split}
\Ptt \dg\dgV\psi &=
\dgEps\dgEpsV \Ptt\psi +  \dgEps\frakh_\e \wedge  \Pnt\psi  - \frakh_\e \wedge \dgEps \Pnt\psi + \frakh_\e \wedge \Ptn\dgV\psi
\end{split}
\]
By combining with \eqref{eq:T_commutator} and rearranging, we obtain the first assertion. 
The three other assertions are obtained by setting $\psi\mapsto\starG\starG^V\psi$, $\psi\mapsto\starG^V\psi$ and $\psi\mapsto\starG\psi$.
\end{PROOF}

\begin{PROOF}{\lemref{lem:4.16}}
We prove the first assertion. Integrating by parts twice,
\[
\begin{split}
\bra \dgV\dg\psi,\eta\ket &\EwR{eq:Green3} \bra\dg\psi,\deltagV\eta\ket - 
\int_{\dM}\Brk{(\PttD\dg\psi,\PtnD\eta)_\gD +
(\PntD\dg\psi,\PnnD\eta)_\gD}\,\VolumeD \\
&\EwR{eq:Green2} \bra\psi, \deltag \deltagV \eta \ket 
- \int_{\dM}\Brk{(\PttD\psi,\PntD\deltagV\eta)_\gD +
(\PtnD\psi,\PnnD\deltagV\eta)_\gD}\,\VolumeD \\
&\quad  - \int_{\dM}\Brk{(\PttD\dg\psi,\PtnD\eta)_\gD +
(\PntD\dg\psi,\PnnD\eta)_\gD}\,\VolumeD.
\end{split}
\]
Now,
\[
\begin{split}
\int_{\dM} (\PtnD\psi,\PnnD\deltagV\eta)_\gD\,\VolumeD 
&\EwR{eq:delta_pure_normal}
\int_{\dM} (\PtnD\psi,-\deltagDV\PnnD\eta - \trace_{S_0} \PtnD\eta)_\gD\,\VolumeD \\
&\hspace{-2cm}= -\int_{\dM}\Brk{(\dgDV\PtnD\psi, \PnnD\eta )_\gD +
(\PtnD\psi, \trace_{S_0} \PtnD\eta)_\gD}\,\VolumeD,
\end{split}
\]
where we used the fact that $\dM$ has no boundary, and
\[
\begin{split}
\int_{\dM}  (\PttD\dg\psi,\PtnD\eta)_\gD\,\VolumeD 
&\EwR{eq:dnabla_pure_tangent}  \int_{\dM}  ( \dgD \PttD\psi + \h_0 \wedge \PtnD\psi,\PtnD\eta)_\gD\,\VolumeD \\
&\hspace{-2cm}=  \int_{\dM}  \brk{(\PttD\psi,\deltagD\PtnD\eta)_\gD + 
 (\PtnD\psi,\trace_{S_0}\PtnD\eta)_\gD}\,\VolumeD.
\end{split}
\] 
Adding things together,
\[
\begin{split}
\bra \dgV\dg\psi,\eta\ket &= \bra\psi,\deltag\deltagV\eta\ket - \int_{\dM} (\PttD\psi,\PntD\deltagV\eta + \deltagD\PtnD\eta)_\gD\,\VolumeD \\
&\quad  + \int_{\dM} (\dgDV\PtnD\psi - \PntD\dg\psi, \PnnD\eta )_\gD\,\VolumeD,
\end{split}
\]
by replacing $\psi\mapsto\psi^T$ and $\eta\mapsto\eta^T$ and combining, we obtain the required result by the definitions \eqref{eq:HHFF} and \eqref{eq:T_commutator}. 
The second assertion follows by replacing $\psi\mapsto\starG^V\psi$ and $\eta\mapsto\starG^V\eta$.
\end{PROOF}

\begin{PROOF}{\lemref{lemma:prescribe_boundary_conditions}}
By a cutoff argument, it suffices to prove the existence of such a $\lambda$ in a neighborhood of $\dM$. 
Substituting the commutation relations proved in \secref{sec:commutation_d_delta} into 
\eqref{eq:T_commutator}, we obtain
\[
\begin{aligned}
\frakT\lambda &= -\dgD\PntD\lambda-\dgDV\PtnD\lambda-\frakh_0\wedge\PnnD\lambda-\smallhalf(\calS_0\PttD\lambda+(\calS_0\PttD\lambda^T)^T)+\PttD\nabg_{\dr}\lambda
\\
\frakF^*\lambda &= -\dgDV\PnnD\lambda-\deltagD\PttD\lambda+i_{\frakh_{0}}^*\PtnD\lambda-\smallhalf(\calS^*_0\PntD\lambda + (\calS_0(\PtnD\lambda)^T)^T) +\PntD\nabg_{\dr}\lambda
\\
\frakF\lambda & =-\dgD\PnnD\lambda-\deltagDV\PttD\lambda+i_{\frakh_{0}}\PntD\lambda
-\smallhalf(\calS_0\PntD\lambda + (\calS^*_0(\PtnD\lambda)^T)^T)+\PtnD\nabg_{\dr}\lambda \\
\frakT^*\lambda &= -\deltagD\PtnD\lambda-\deltagDV\PntD\lambda-\trace_{\frakh_{0}}\PttD\lambda-\smallhalf(\calS^*_0\PnnD\lambda+(\calS^*_0\PnnD\lambda^T)^T)+\PnnD\nabg_{\dr}\lambda.
\end{aligned}
\]
Set
\[
\begin{aligned}
\chi_1 &= \rho_1 + \dgD\phi_3  + \dgDV\phi_2 + \frakh_0\wedge\phi_4  + \smallhalf(\calS_0\phi_1+(\calS_0\phi_1^T)^T) 
\\
\chi_2 &= \rho_2 + \dgDV\phi_4 + \deltagD\phi_1 - i_{\frakh_{0}}^*\phi_2 + \smallhalf(\calS^*_0\phi_3 + (\calS_0(\phi_2)^T)^T) \\
\chi_3 &= \rho_3 + \dgD\phi_4 + \deltagDV\phi_1  - i_{\frakh_{0}}\phi_3
+ \smallhalf(\calS_0\phi_3 + (\calS^*_0(\phi_2)^T)^T) \\
\chi_4 &= \rho_4 + \deltagD\phi_2 + \deltagDV\phi_3 + \trace_{\frakh_{0}}\phi_1 + \smallhalf(\calS^*_0\phi_4 +(\calS^*_0\phi_4^T)^T),
\end{aligned}
\]
then $\lambda$ has to satisfy the boundary conditions
\[
\lambda|_{\dM} = \phi \equiv \phi_1 + dr\wedge\phi_2 + (dr)^T\wedge \phi_3 + dr\wedge(dr)^T\wedge \phi_4,
\]
and
\[
\nabg_{\dr}\lambda|_{\dM} = \chi \equiv \chi_1 + dr\wedge\chi_2 + (dr)^T\wedge \chi_3 + dr\wedge(dr)^T\wedge \chi_4,
\]
with $\chi_i\in W^{s-1-1/p,p}\VectorFormsdM$.

By the trace theorem, we may extend $\chi \in W^{s-1,p}\VectorFormsM$. We proceed to solve (locally) the non-characteristic linear equation,
\[
\nabg_{\dr}\lambda = \chi 
\qquad 
\lambda|_{\dM} = \phi,
\]
whose solution $\lambda \in W^{s,p}\VectorFormsM$ depends continuously on $\chi$ and $\phi$, hence satisfies the desired properties.
\end{PROOF}

\begin{PROOF}{\lemref{lemma:prescribed_boundary_data_second_order}}
Once again, using a cutoff argument, it suffices to construct such a $\lambda$ in a neighborhood of $\dM$.
The first part of the lemma will be proved if we show that the conditions \eqref{eq:prescribed_boundary_data_second_order} amount to fixing the boundary-values of $\lambda$ and its first three covariant derivative along $\dr$, as $\lambda$ can then be constructed via radial integration. 

Specifically, we observe that the first line in \eqref{eq:prescribed_boundary_data_second_order} amounts to $\lambda$ and its first derivative vanishing on $\dM$. In particular,
\[
\dg\lambda|_{\dM}=0 
\qquad 
\dgV\lambda|_{\dM}=0 
\qquad 
\deltag\lambda|_{\dM}=0 
\qquad 
\deltagV\lambda|_{\dM}=0.
\]
We show  that under these conditions, the prescription of 
\[
(\PttD\bHg^*\lambda, \,\,\PtnD\bFg^*\lambda, \,\,\PntD\bFg\lambda, \,\, \PnnD\bHg\lambda, \,\,
\frakT\bHg^*\lambda,\,\, \frakF\bFg^*\lambda,\,\,  \frakF^*\bFg\lambda, \,\, \frakT^*\bHg\lambda)
\]
amounts to the prescription of $\nabg_{\dr}\nabg_{\dr}\lambda|_{\dM}$ and $\nabg_{\dr}\nabg_{\dr}\nabg_{\dr}\lambda|_{\dM}$.

A direct calculation gives
\[
\begin{aligned}
&\PttD\bHg^*\lambda=\PnnD(\nabg_{\dr}\nabg_{\dr}\lambda) \qquad 
&\PtnD\bHg^*\lambda=0 \qquad 
&\PntD\bHg^*\lambda=0 \qquad 
&\PnnD\bHg^*\lambda=0 \\
&\PtnD\bFg^*\lambda=\PntD(\nabg_{\dr}\nabg_{\dr}\lambda) \qquad 
&\PttD\bFg^*\lambda=0 \qquad 
&\PnnD\bFg^*\lambda=0 \qquad 
&\PntD\bFg^*\lambda=0 \\
&\PntD\bFg\lambda=\PtnD(\nabg_{\dr}\nabg_{\dr}\lambda)  \qquad 
&\PttD\bFg\lambda=0 \qquad 
&\PnnD\bFg\lambda=0 \qquad 
&\PtnD\bFg^*\lambda=0 \\
&\PnnD\bHg\lambda=\PttD(\nabg_{\dr}\nabg_{\dr}\lambda) \qquad 
&\PtnD\bHg\lambda=0 \qquad 
&\PntD\bHg\lambda=0 \qquad 
&\PttD\bHg\lambda=0,
\end{aligned}
\]
and
\[
\begin{aligned}
\frakT\bHg^*\lambda &= \PnnD(\nabg_{\dr}\nabg_{\dr}\nabg_{\dr}\lambda)+Q_1(\nabg_{\dr}\nabg_{\dr}\lambda) \\
\frakF\bFg^*\lambda &= \PntD(\nabg_{\dr}\nabg_{\dr}\nabg_{\dr}\lambda)+Q_2(\nabg_{\dr}\nabg_{\dr}\lambda) \\
\frakF^*\bFg\lambda &= \PtnD(\nabg_{\dr}\nabg_{\dr}\nabg_{\dr}\lambda)+Q_3(\nabg_{\dr}\nabg_{\dr}\lambda) \\
\frakT^*\bHg\lambda &= \PttD(\nabg_{\dr}\nabg_{\dr}\nabg_{\dr}\lambda)+Q_4(\nabg_{\dr}\nabg_{\dr}\lambda),
\end{aligned}
\]
where $Q_i:\VectorFormsU\to \VectorFormsdM$ are  linear first-order differential boundary operators. 
Set on $\dM$
\[
\phi=\phi_1+(dr)^T\wedge\phi_2+dr\wedge \phi_3+dr\wedge(dr)^T\wedge \phi_4,
\]
$\tilde{\rho}_i = \rho_i + Q_i(\phi)$ and $\rho\in\VectorFormsdM$ by
\[
\rho=\tilde{\rho}_1+(dr)^T\wedge\tilde{\rho}_2+dr\wedge \tilde{\rho}_3+dr\wedge(dr)^T\wedge \tilde{\rho}_4.
\]
We extend $\phi$ and $\rho$ in a neighborhood of $\dM$, which by the trace theorem are of class $W^{s,p}\VectorFormsdM$ and $W^{s-1,p}\VectorFormsdM$, and can be taken to satisfy 
\[
\begin{split}
&\Norm{\phi}_{W^{s,p}(\M)}\lesssim \Norm{\phi}_{W^{s-1/p,p}(\dM)}\lesssim \sum_{i=1}^4\Norm{\phi_i}_{W^{s-1/p,p}(\dM)}
\\&\Norm{\rho}_{W^{s-1,p}(\M)}\lesssim \Norm{\rho}_{W^{s-1-1/p,p}(\dM)}\lesssim \sum_{i=1}^4\Norm{\tilde{\rho}_i}_{W^{s-1-1/p,p}(\dM)}.
\end{split}
\]
Consider then the non-characteristic third-order equation, 
\[
\begin{gathered}
\nabg_{\dr}\nabg_{\dr}\nabg_{\dr}\lambda=\rho \\
\lambda|_{\dM}=0, \qquad \nabg_{\dr}\lambda|_{\dM}=0, \qquad \nabg_{\dr}\nabg_{\dr}\lambda|_{\dM}=\phi|_{\dM}.
\end{gathered} 
\]

The solution is of class $\lambda\in W^{s+2,p}\VectorFormsKM$, satisfying
\[
\Norm{\lambda}_{W^{s+2,p}(\M)}\lesssim \sum_{i=1}^4\brk{\Norm{\phi_i}_{W^{s-1/p,p}(\dM)}+\Norm{\rho_i}_{W^{s-1-1/p,p}(\dM)}}
\] 
which is \eqref{eq:sobolev_estimate_boundary_data}. 
\end{PROOF}

\begin{PROOF}{\corrref{cor:full_problem_boundary}}
We first use \lemref{lemma:prescribe_boundary_conditions} to construct
$\omega_1\in\VectorFormsKM$ satisfying
\[
\begin{gathered}
(\PttD,\PntD,\PtnD,\PnnD)\omega_1 = (\phi_1,\phi_2,\phi_3,\phi_4) \\
(\frakT,\frakF^*,\frakF,\frakT^*)\omega_1 = (\rho_1,\rho_2,\rho_3,\rho_4) \\
\end{gathered}
\]
We then use \lemref{lemma:prescribed_boundary_data_second_order} to construct $\omega_2\in\VectorFormsKM$ satisfying
\[
\begin{gathered}
\omega_2|_{\dM}=0, \qquad \nabg_{\dr}\omega_2|_{\dM}=0 \\
(\PttD\bHg^*,\PtnD\bFg^*,\PntD\bFg,\PnnD\bHg)\omega_2 = (\mu_1,\mu_2,\mu_3,\mu_4) - (\PttD\bHg^*,\PtnD\bFg^*,\PntD\bFg,\PnnD\bHg)\omega_1\\
(\frakT\bHg^*,\frakF\bFg^*,\frakF^*\bFg,\frakT^*\bHg)\omega_2 = (\nu_1,\nu_2,\nu_3,\nu_4) -(\frakT\bHg^*,\frakF\bFg^*,\frakF^*\bFg,\frakT^*\bHg)\omega_1.
\end{gathered}
\]
Then, $\lambda=\omega_1+\omega_2$ satisfies the required conditions. Moreover, the estimate \eqref{eq:sobolev_estimate_boundary_data_full} follows by combining \eqref{eq:sobolev_estimate_boundary_data} and \eqref{eq:estimate_1st_bc}. 
\end{PROOF}


\providecommand{\bysame}{\leavevmode\hbox to3em{\hrulefill}\thinspace}
\providecommand{\MR}{\relax\ifhmode\unskip\space\fi MR }
\providecommand{\MRhref}[2]{%
  \href{http://www.ams.org/mathscinet-getitem?mr=#1}{#2}
}
\providecommand{\href}[2]{#2}
\pagebreak


\providecommand{\bysame}{\leavevmode\hbox to3em{\hrulefill}\thinspace}
\providecommand{\MR}{\relax\ifhmode\unskip\space\fi MR }
\providecommand{\MRhref}[2]{%
  \href{http://www.ams.org/mathscinet-getitem?mr=#1}{#2}
}
\providecommand{\href}[2]{#2}
\begin{thebibliography}{Tay11b}

\bibitem[Air63]{Air63}
G.B. Airy, \emph{On the strains in the interior of beams}, Phil. Trans. Roy.
  Soc. London \textbf{153} (1863), 49--80.

\bibitem[Bel92]{Bel92}
E.~Beltrami, \emph{Osservazioni sulla nota precedente}, Atti Real Accad. Naz.
  Lincei Rend. \textbf{5} (1892), 141--142.
  
\bibitem[Cal61]{Cal61} E. Calabi, \emph{On compact, Riemannian manifolds with constant cur-
vature. I}, Proc. Sympos. Pure Math, vol.~III, American Mathe-
matical Society, (1961), 155--180.

\bibitem[dR84]{deR84}
G.~de~Rham, \emph{Differentiable manifolds}, Springer, 1984.

\bibitem[EL83]{EL83}
J.~Eells and L.~Lemaire, \emph{Selected topics in harmonic maps}, CBMS Regional
  Conference Series in Mathematics, vol.~50, Amer. Math. Soc., 1983.

\bibitem[Gra70]{Gra70}
A.~Gray, \emph{Some relations between curvature and characteristic classes},
  Math. Ann. \textbf{184} (1970), 257--267.

\bibitem[Gur72]{Gur72}
M.E. Gurtin, \emph{The linear theory of elasticity}, Mechanics of Solides
  (C.~Truesdell, ed.), vol.~II, Springer Verlag, 1972.

\bibitem[H\"07]{Hor07}
L.~H\"ormander, \emph{The analysis of linear partial differential operators},
  vol. III, Springer, 2007.

\bibitem[KL21]{KL21b}
R.~Kupferman and R.~Leder, \emph{Double forms: On {Saint-Venant}
  compatibility and stress potentials in manifolds with boundary and constant
  sectional curvature}, arXiv preprint arXiv:2104.05794, 2021.

\bibitem[Kul72]{Kul72}
R.S. Kulkarni, \emph{On the {Bianchi} identities}, Math. Ann. \textbf{199}
  (1972), 175--204.

\bibitem[Lee18]{Lee18}
J.M. Lee, \emph{Introduction to {Riemannian} manifolds}, second edition ed.,
  Graduate Texts in Mathematics, Springer, 2018.

\bibitem[MSK14]{MSK14}
M.~Moshe, E.~Sharon, and R.~Kupferman, \emph{The plane stress state of
  residually stressed bodies: A stress function approach}, arXiv preprint
  arXiv:1409.6594, 2014.

\bibitem[MSK15]{MSK15}
\bysame, \emph{Elastic interactions between two-dimensional geometric defects},
  Phys. Rev. E 92 \textbf{92} (2015), 062403.

\bibitem[Pet16]{Pet16}
P.~Petersen, \emph{Riemannian geometry}, third ed., Springer, 2016.

\bibitem[Rud91]{Rud91}
W.~Rudin, \emph{Functional analysis}, 2nd ed., McGraw-Hill, 1991.

\bibitem[Sch95]{Sch95b}
G.~Schwarz, \emph{Hodge decomposition -- a method for solving boundary value
  problems}, Lecture notes in mathematics, Springer, 1995.

\bibitem[Tay11a]{Tay11a}
M.E. Taylor, \emph{Partial differential equations}, vol.~I, Springer, 2011.

\bibitem[Tay11b]{Tay11b}
\bysame, \emph{Partial differential equations}, vol.~II, Springer, 2011.

\bibitem[Tay11c]{Tay11c}
\bysame, \emph{Partial differential equations}, vol. III, Springer, 2011.

\bibitem[Tru59]{Tru59}
C.~Truesdell, \emph{Invariant and complete stress functions for general
  continua}, Arch. Rat. Mech. Anal. \textbf{4} (1959), 1--29.

\bibitem[Wlo87]{Wlo87}
J.~Wloka, \emph{Partial differential equations}, Cambridge University Press,
  1987.

\end{thebibliography}

\end{document}